\documentclass[11pt]{amsart}

\usepackage{mathrsfs}
\usepackage{amsfonts,amssymb,amsmath,amsthm}
\usepackage{latexsym}
\usepackage{bm}
\usepackage{stmaryrd}

\usepackage{geometry}
\usepackage{indentfirst}
\usepackage{titletoc}
\usepackage{appendix}
\usepackage{enumitem}
\usepackage{needspace}
\usepackage{changepage}
\usepackage{etoolbox}

\usepackage{graphicx}
\usepackage{float}
\usepackage{booktabs}
\usepackage{longtable}
\usepackage{array}
\usepackage{arydshln}
\usepackage{tikz}
\usepackage{xcolor}
\usetikzlibrary{arrows.meta,decorations.markings}

\usepackage{cite}
\makeatletter
\patchcmd{\@sect}{\@tocwrite{#1}{#8}}{\@tocwrite{#1}{#7}}{}{\PackageError{aos}{Cannot patch displayed section contents}{}}
\patchcmd{\@sect}{\@tocwrite{#1}{#8}}{\@tocwrite{#1}{#7}}{}{\PackageError{aos}{Cannot patch run-in section contents}{}}
\makeatother
\usepackage[colorlinks=true,linkcolor=blue!45!black,citecolor=green!30!black,urlcolor=blue!50!black]{hyperref}

\allowdisplaybreaks[4]
\newcommand{\blackhyperref}[2]{%
  {\hypersetup{linkcolor=black}\hyperref[#1]{#2}}%
}

\newcommand{\R}{\mathbb R}

\newcommand{\Z}{\mathbb Z}
\newcommand{\Nat}{\mathbb N}

\newcommand{\calE}{\mathcal E}

\newcommand{\calH}{\mathcal H}

\newcommand{\calM}{\mathcal M}

\newcommand{\Id}{\mathrm{Id}}

\DeclareMathOperator{\diam}{diam}

\DeclareMathOperator{\dist}{dist}

\usepackage{microtype}
\newcommand{\D}{\mathbb D}

\newcommand{\1}{\mathbf1}
\newcommand{\dd}{\,\mathrm d}
\newcommand{\Real}{\operatorname{Re}}
\newcommand{\Ai}{\operatorname{Ai}}
\newcommand{\Var}{\operatorname{Var}}
\newcommand{\calT}{\mathcal T}

\DeclareMathOperator*{\essinf}{ess\,inf}

\numberwithin{equation}{section}
\newtheorem{theorem}{Theorem}[section]
\newtheorem{lemma}[theorem]{Lemma}

\newtheorem{prop}[theorem]{Proposition}

\newtheorem*{maintheorem}{Main Theorem}

\theoremstyle{definition}
\newtheorem{defi}[theorem]{Definition}
\newtheorem*{defi*}{Definition}
\newtheoremstyle{compactintro}{0pt}{}{\normalfont}{}{\bfseries}{.}{5pt plus 1pt minus 1pt}{}
\theoremstyle{compactintro}

\AddToHook{env/question/before}{\vspace{-.5\baselineskip}}
\AddToHook{env/observation/before}{\vspace{-.5\baselineskip}}
\theoremstyle{definition}
\newtheorem{remark}[theorem]{Remark}
\newtheorem*{remark*}{Remark}
\AddToHook{env/remark/begin}{\emergencystretch=2em}
\AddToHook{env/remark*/begin}{\emergencystretch=2em}

\theoremstyle{definition}

\begin{document}

\title[Asymptotic observable symmetry II]{Asymptotic observable symmetry for spacetime observability of waves II: Smooth bounded domains}

\author{Jingrui Niu}
\address[Jingrui Niu]{Institute for Advanced Study in Mathematics, Harbin Institute of Technology, Harbin 150001, China}
\email{jingrui.niu@hit.edu.cn}

\author{Shengquan Xiang}
\address[Shengquan Xiang]{School of Mathematical Sciences, Peking University, Beijing 100871, China}
\email{shengquan.xiang@math.pku.edu.cn}
\date{}

\begin{abstract}
Continuing \cite{NWX} and Part I \cite{AOSI}, we further investigate symmetry for spacetime observability of wave equations. Building on asymptotic observable symmetry (AOS) introduced in \cite{AOSI}, we study symmetry under glancing propagation and caustics. We prove that, on every smooth bounded domain in $\R^2$ and for every positive observation time, there is a fixed measurable set satisfying the essential spacetime measurable geometric control condition. However, the observability inequality fails on this set.
\end{abstract}

\hypersetup{pdftitle={Asymptotic observable symmetry for spacetime observability of waves II: Smooth bounded domains},pdfauthor={Jingrui Niu and Shengquan Xiang}}

\maketitle

\setcounter{tocdepth}{1}
\tableofcontents

\section{Introduction}\label{sec:intro}

\subsection{Observability, geometry, and symmetry}\label{subsec:intro-obs}

Let $\Omega\subset\R^2$ be a smooth bounded domain. We consider the Dirichlet wave equation
\begin{equation}\label{eq:wave}
\begin{cases}
\partial_t^2u-\Delta u=0,& \forall (t,x)\in\R\times\Omega,\\
u=0,& \forall (t,x)\in\R\times\partial\Omega,\\
u[0]=(u_0,u_1)\in H_0^1(\Omega)\times L^2(\Omega).
\end{cases}
\end{equation}
Here $u[t]:=\left(u(t,\cdot),\partial_tu(t,\cdot)\right)$, and the conserved energy is
\begin{equation*}
E[u](t):=\frac12\int_\Omega(|\nabla u|^2+|\partial_tu|^2)\dd x.
\end{equation*}
For $T>0$ and a measurable set $G\subset(0,T)\times\Omega$, observability means that there is $C_G>0$ such that
\begin{equation}\label{eq:obs}
E[u](0)\le C_G\int_G|\partial_tu(t,x)|^2\dd x\dd t
\end{equation}
for every solution of \eqref{eq:wave}.

The literature on controllability and observability of wave equations is extensive. Here, we briefly review some related work and refer to Part~I \cite[Introduction]{AOSI} for further background and references.

For early results and methods in wave controllability, we refer to Russell \cite{Russell1971a} and Lions \cite{Lions1988}. In the classical smooth setting, the geometric control condition (GCC) describes observability through generalized rays reaching the observation region within the prescribed time; see Bardos--Lebeau--Rauch \cite{BLR1992} and Burq--G\'erard \cite{BG1997}. Further developments include observability with rough coefficients \cite{DE2017,BDLR2024}, and uniform observability estimates under geometric control \cite{LL2016}, the spacetime formulation \cite{LRLTT2017}, as well as regional and partial observability and control \cite{DEZ2025}.

For wave equations with lower-order terms and potentials, we refer to works \cite{Zhang2000,Zhang2000Potential,DZZ2008}. For unique continuation under partial analyticity assumptions, see \cite{Tataru1995,Hormander1997,RZ1998}; for quantitative results and applications to approximate control, see \cite{LL2019}. On a suitable bounded domain, Liao--L\"u--Zhang \cite{LLZ2026} construct a smooth real-valued space and time dependent potential and an open observation set satisfying GCC for which a nonzero wave vanishes on the observed cylinder.

For related work on wave decay and stabilization, see \cite{Lebeau1996,Burq1998,BG2020}. For semilinear wave equations, we refer to Zuazua \cite{Zuazua1990} for exact controllability and to Joly--Laurent \cite{JL2013} for stabilization under geometric control. For controllability of geometric wave equations, we also refer to \cite{KX2024,CKX2025,CKXgeneral2025}.

\vspace{2mm}

The observation sets considered here are only measurable. Recall the essential specetime measurable GCC condition used in \cite{AOSI}.

\begin{defi*}
A measurable set $G\subset(0,T)\times\Omega$ satisfies the essential spacetime measurable GCC if there exists $\eta>0$ such that
\begingroup
\renewcommand{\theHequation}{gcc}
\begin{equation}\tag{GCC}\label{eq:gcc}
\essinf_{\rho\in S^*\Omega}D_G(\rho)\ge\eta>0 \textrm{ \; \; where \; } D_G(\rho):=\int_0^T\1_G(t,\pi\Phi_t\rho)\dd t.
\end{equation}
\endgroup
Here $\Phi_t$ is the regular unit-speed billiard flow, $\pi:S^*\Omega\to\Omega$ is the spatial projection, and the essential infimum is taken with respect to Liouville measure.
\end{defi*}

This condition concerns regular interior rays and is unchanged by spacetime null modifications of $G$; see \cite[Lemma~B.2 and Proposition~B.3]{AOSI}. For measurable observation, the classical smooth-weight argument does not directly apply: weak convergence of high-frequency energy can  bypass to the indicator $\1_G$. We refer to \cite[Appendix~B.4]{AOSI} for this distinction. \vspace{2mm}

Observability and control on measurable sets have been studied for heat equations \cite{AEWZ2014,WWZZ2019,BM2023}, Schr\"odinger equations \cite{BZ2019,BurqZhu2025}, the periodic KdV equation \cite{NWXKdV2026}, and wave equations \cite{NWX,HX}. We refer to Wang \cite{WangICM2026} for a recent survey on the geometric characteristics of observable regions.

Ray coverage and wave observation are different assertions. A ray records one propagation direction, whereas the observed velocity is the sum of the components arriving at a spacetime point. In the torus setting, Niu--Wang--Xiang \cite{NWX} introduced the {\it observable symmetry condition} (OSC), excluding nontrivial symmetries whose transported profiles cancel on the observation region. There, OSC together with GCC fully characterizes observability. Huang--Xiang \cite{HX} extend the symmetry theory to the interval.

\subsection{The main result}\label{subsec:intro-main}

Part~I \cite{AOSI} introduces \emph{asymptotic observable symmetry} (AOS) and uses radial reflection to construct a fixed measurable observation set on the disc satisfying \eqref{eq:gcc} but failing observability, for a fixed positive observation time. In this paper, we use local glancing propagation to obtain this conclusion on {\it every smooth bounded domain} and for {\it every positive observation time}.

\begin{maintheorem}\phantomsection\label{thm:main}
Let $\Omega\subset\R^2$ be a smooth bounded domain (not necessarily simply connected). For every $T>0$, there is a Borel set $G\subset(0,T)\times\Omega$ which satisfies \eqref{eq:gcc}\footnote{In fact, in our construction, every unit-speed generalized ray spends a uniformly positive amount of time in $G\cup\left((0,T)\times\partial\Omega\right)$; see Theorem~\ref{thm:curved-gcc}.} but for which observability fails.

More precisely, there is $\eta>0$ such that
\begin{equation}\label{eq:main}
\essinf_{\rho\in S^*\Omega}D_G(\rho)\ge\eta.
\end{equation}
There are also solutions $\{u_k\}_{\mathbb{N}^*}$ of \eqref{eq:wave} such that $E[u_k](0)\equiv1$ and
\begin{equation}\label{eq:main-nonobs}
\int_G|\partial_tu_k|^2\dd x\dd t\longrightarrow0 \textrm{ \; as \;  } k\rightarrow \infty.
\end{equation}
\end{maintheorem}
The construction uses one positively curved boundary arc and a short time interval. The passage to every prescribed $T>0$ is given in Section~\ref{subsec:gcc-energy}. See Theorem~\ref{thm:curved-gcc} for a stronger and more detailed version of the \hyperref[thm:main]{Main Theorem}.

\begin{remark}
The generated approximate velocities $S_k$ used in the construction below describe the leading parts of the time derivatives of the high-frequency waves $u_k$. The profile matching is exact, whereas its realization by solutions is asymptotic.

This suggests that the conclusion of the \hyperref[thm:main]{Main Theorem} could be extended to wave equations with fixed real first-order coefficients smooth on $[0,T]\times\overline\Omega$ and a bounded real-valued potential, with the same Dirichlet boundary condition. In the glancing construction, the two branches have the same limiting tangential propagation direction. The first-order terms could therefore be incorporated into a common amplitude through a transport equation, preserving the profile matching. The potential term could be absorbed into the realization error by an energy estimate. The same observation-set construction and GCC argument would then apply, with a further reduction of the chosen scales if needed.
\end{remark}

\subsection{A brief recall of AOS}\label{subsec:intro-aos}

We use the definition of AOS in \cite[Definition~2.5]{AOSI}; the full framework and the reflection example are given in \cite[Sections~2--4]{AOSI}. We recall only the objects needed below. For a two-branch construction, fixed compatible profiles $f,g$ and their matching region $\mathcal M$ satisfy
\begin{equation*}
f(\alpha)=g(\beta),\quad \forall(\alpha,\beta)\in\mathcal M.
\end{equation*}
At scale $\lambda_k\to0^+$, the phases and amplitude generate the approximate velocity
\begin{gather*}
S_k=A_{\lambda_k}\left(f(\phi_k)-g(\psi_k)\right), \\
 \Gamma_k:=\mathcal U_k\cap\{(\phi_k,\psi_k)\in\mathcal M\},
\end{gather*}
where $\mathcal U_k$ is its support region. Thus $S_k=0$ on $\Gamma_k$. The profiles and their symmetry remain fixed while the propagation scales vary.

The analytic requirement is approximation by exact wave velocities. For the present construction, this approximation is used on a fixed short interval $(0,\tau)$, with $0<\tau<T$ independent of $k$. Set $Q_\tau:=(0,\tau)\times\Omega$. We construct one fixed observation set $G_\tau\subset Q_\tau$ and exact waves satisfying
\begin{gather*}
\partial_tu_k=S_k+r_k \textrm{ \; with \; } E[u_k](0)=1,\\
\|\1_{G_\tau}S_k\|_{L^2(Q_\tau)}+\|r_k\|_{L^2(Q_\tau)}\rightarrow0  \textrm{ \; as \; }  k\rightarrow \infty.
\end{gather*}
This gives failure of \eqref{eq:obs} with $T=\tau$. For the prescribed time $T$, the exact waves continue to evolve and we take $G=G_\tau\subset(0,T)\times\Omega$. 

We retain the terminology and notation of \cite[Section~2]{AOSI}, with $\varepsilon_k=\lambda_k$. The supplementary discussion in \cite[Appendix~A]{AOSI} explains compatible profiles, velocity approximation, and strict localization; the estimates required for the present construction are proved below.

\subsection{The glancing construction in general domains}\label{subsec:intro-glancing}

The radial construction of \cite[Sections~3--4]{AOSI} uses the rotational invariance of the disc. A general domain need not have this symmetry. At a positively curved boundary, however, near-tangent waves still have two normal branches connected by a turning caustic. We use these branches to construct compatible profiles with matching regions, then place them in moving packets. We call this construction \emph{AOS via glancing propagation}. \vspace{1.5mm}

We outline the main ideas here. More detailed strategies for the construction on the disc and on general domains are presented in Sections~\ref{subsec:gl-intuition} and~\ref{subsec:curved-intuition}, respectively, while Section~\ref{subsec:gcc-sketch} outlines the proof of GCC.

\noindent{\it 1. Near-tangent propagation on a local arc.} Every smooth bounded planar domain has a positively curved boundary arc; see Lemma~\ref{lem:local-geometry}. We construct high-frequency waves that propagate nearly tangentially along this arc and concentrate in a thin boundary layer. The boundary reflection and the turning caustic connect the two normal branches. Fix a sufficiently small time $\tau\in(0,T)$, independently of the sequence of scales, so that all packets remain near this arc. \vspace{1.5mm}

\noindent{\it 2. Fixed profiles in moving packet regions.} The glancing caustic relates the two profiles through the Hilbert transform, while the boundary reflection imposes a quarter-period translation. We choose the diamond profiles $f_*,g_*$ (see Section~\ref{subsec:diamond}) to satisfy these compatibility relations and to have equal nonzero constant values on intervals. This gives matching rectangles in the phase variables, whose pullbacks define the matching regions inside the moving packet regions $\mathcal U_k$.
\begin{gather*}
S_k=A_{\lambda_k}\left(f_*(\phi_k)-g_*(\psi_k)\right) \textrm{ \; thus \; } S_k=0\quad\text{outside }\mathcal U_k,\\
\textrm{ then further define \; }\Gamma_k\subset\mathcal U_k\quad\text{ such that }\quad \1_{\Gamma_k}S_k=0.
\end{gather*}
The same profiles and symmetry are used at every scale; the phases $\phi_k,\psi_k$, amplitude $A_{\lambda_k}$ and moving packet region $\mathcal U_k$ change. With $s$ denoting local boundary arclength, we place the packets in disjoint intervals of $s-t$, so that different packets remain disjoint while moving along the chosen arc. \vspace{1.5mm}

\noindent{\it 3. A fixed observation region and a sequence of waves.} Inspired by  the radial example of \cite[Section~1.4]{AOSI}, we observe everything outside the moving packet region $\mathcal U_k$. Inside each $\mathcal U_k$, we retain $\Gamma_k$ together with a small additional observation region $\calE_k$. Thus the fixed set is
\begin{equation*}
G_\tau:=Q_\tau\setminus\bigcup_{k\ge1}\left(\mathcal U_k\setminus(\Gamma_k\cup\calE_k)\right).
\end{equation*}
Restricting the union to the first $N$ moving packet regions gives decreasing finite-stage sets $G^{(N)}$ with limit $G_\tau$; see Figure~\ref{fig:curved-intro-stages}. The additional regions $\calE_k$ allow for phase errors in the ray estimates. Their parameters are chosen so that their contribution to the observed profile energy tends to zero. The waves satisfy $E[u_k](0)=1$ and $\partial_tu_k=S_k+r_k$. Support and symmetry matching yield
\begin{equation*}
\|\1_{G_\tau}\partial_tu_k\|_{L^2(Q_\tau)} \le\|r_k\|_{L^2(Q_\tau)}+\|\1_{\calE_k}S_k\|_{L^2(Q_\tau)}\rightarrow0.
\end{equation*}

\noindent{\it 4. GCC realization.} It remains to choose the same construction parameters so that $G_\tau$ also satisfies \eqref{eq:gcc}. The ray estimate has two levels. At one scale, we separate the ray geometries and use profile matching to obtain a positive observed-time bound. To combine the scales for positive glancing rays,\footnote{Positive glancing rays move in the same tangential direction as the moving packet regions, while negative glancing rays move in the opposite direction.} we use a threshold depending on the ray: it can meet at most one moving packet region above this threshold, while all smaller moving packet regions lie in one thin boundary collar. The time spent in that collar is controlled separately. Negative glancing rays and non-glancing rays are handled by the small total packet width and the shallow packet support, respectively. Together, these estimates give
\begin{equation*}
\essinf_{\rho\in S^*\Omega}\int_0^\tau\1_{G_\tau}(t,\pi\Phi_t\rho)\dd t\ge\eta_\tau>0.
\end{equation*}

\smallskip\noindent For a larger prescribed time, we keep this observation set inside $(0,\tau)\times\Omega$ and add no observation outside that interval. Its ray lower bound and the same non-observability sequence remain valid.

\begin{figure}[htbp]
\centering
\includegraphics[width=\textwidth]{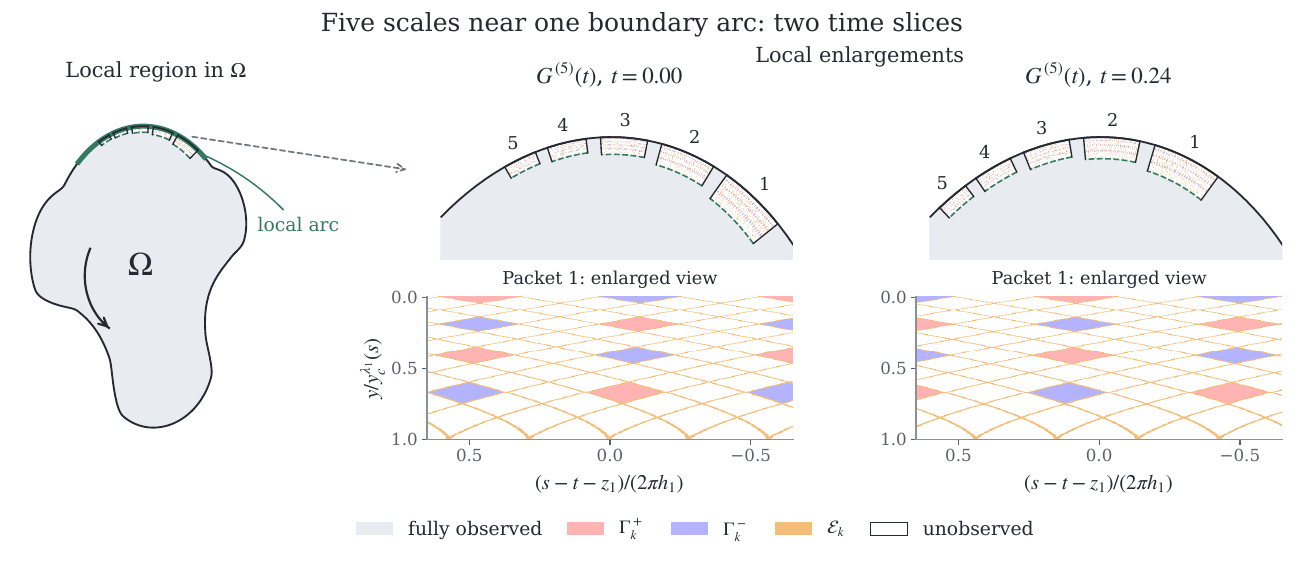}
\caption{The same $G^{(5)}$ at two short times near a positively curved boundary arc. Gray is fully observed; red and blue mark profile matching; orange is the additional observation $\calE_k$; white is unobserved. The lower row enlarges part of packet 1 in normalized depth.}\label{fig:curved-intro-stages}
\end{figure}

\Needspace{12\baselineskip}
\subsection{Caustics, Airy scaling, and wave estimates}\label{subsec:intro-caustics}

Caustics have been studied extensively in connection with concentration, dispersion and Strichartz estimates for wave equations. These studies show both losses caused by caustics and the role of their detailed spacetime structure in obtaining sharper estimates \cite{ILP2014,ILLP2023}. The present paper uses the same local propagation background to study symmetry and cancellation in observation on measurable spacetime sets. \vspace{1.5mm}

\noindent\textit{Classical glancing and Airy analysis.} The local analysis of glancing rays and diffracted waves was developed in works including Friedlander \cite{Friedlander1976}, Melrose \cite{Melrose1976} and Melrose--Taylor \cite{MelroseTaylor1986}. Near a simple turning point, the leading normal equation and its decaying solution take the form
\begin{equation*}
h^2V''(\zeta)+\zeta V(\zeta)=0, \quad V(\zeta)=\Ai(-\zeta/h^{2/3}).
\end{equation*}
We use the convention $\zeta>0$ on the oscillatory side and $\zeta<0$ on the decaying side. Thus $h^{2/3}$ is the transition scale in the Airy coordinate. Away from the turning point on the oscillatory side, the Airy asymptotics yield two geometric-optics branches with a fixed relative phase. We refer to \cite{Olver} for the asymptotics and to \cite[Appendix~A.4.2]{AOSI} for an explicit derivation in the disc. \vspace{1.5mm}

\noindent\textit{Boundary concentration and wave estimates.} Smith--Sogge \cite{SmithSogge2007} study spectral-cluster bounds on manifolds with boundary, where whispering-gallery concentration exhibits the normal scale $h^{2/3}$. Blair--Smith--Sogge \cite{BlairSmithSogge2009} establish Strichartz estimates for waves on manifolds with boundary. Dispersion inside strictly convex domains is studied by Lebeau \cite{Lebeau2006}, while Ivanovici \cite{Ivanovici2010,Ivanovici2012} constructs counterexamples to certain Strichartz estimates valid in the boundaryless setting. Ivanovici--Lebeau--Planchon \cite{ILP2014} prove sharp dispersive estimates in the Friedlander model and describe the role of swallowtail caustics; \cite{ILP2021New} gives further Strichartz counterexamples in a two-dimensional model. Ivanovici--Lascar--Lebeau--Planchon \cite{ILLP2023} treat general strictly convex domains and use the spacetime distribution of concentration to improve on Strichartz bounds obtained from the worst dispersive decay alone. The Airy normal form describes the local structure near a simple turning point. The swallowtail caustics studied in these works arise in the geometry of multiply reflected wave fronts. Recent work also treats dispersion with Neumann boundary conditions \cite{IvanoviciSanchez2026}; the present paper uses Dirichlet conditions. \vspace{1.5mm}

\noindent\textit{Airy scaling and turning depth.} We retain $h$ for the semiclassical parameter and $\lambda$ for the small glancing parameter. Our scale choice for exact-wave realization gives, by \eqref{eq:scales} and \eqref{eq:turning},
\begin{equation*}
h\sim\lambda^{9/2},\quad y_c^\lambda(s)\sim\lambda^2,  \textrm{ \; and \; } h^{2/3}\sim\lambda^3\ll\lambda^2.
\end{equation*}
On the chosen positively curved arc, $|\partial_y\zeta_\lambda|$ is bounded above and below by positive constants. Thus, at the base semiclassical scale $h$, the Airy transition occurs in a layer of thickness comparable to $h^{2/3}$ around the turning curve, whereas the distance from the boundary to that curve is of order $\lambda^2$. The dimensionless ratio $a^{3/2}/h$ used for a normal geometric scale $a$ in \cite{Ivanovici2010,ILP2014} corresponds, at a scale comparable to $\lambda^2$, to
\begin{equation*}
\frac{\lambda^3}{h}\sim\lambda^{-3/2}\rightarrow\infty.
\end{equation*}
We use the two-branch expansion away from the turning curve on the oscillatory side and integral Airy bounds across the turning layer; see Sections~\ref{subsubsec:curved-scales} and~\ref{subsec:curved-correction}. \vspace{1.5mm}

\noindent\textit{Relation with the present construction.} The Airy connection and the boundary condition impose compatibility on the profiles.  Our diamond profiles satisfy these compatibility relations and also provide matching plateaux. We then localize the generated velocities, approximate them by exact wave velocities, and coordinate all scales in one fixed measurable observation set. The remaining requirement is a uniform ray observation-time bound for that same set. Thus the Airy analysis supplies the propagation structure, while the profile matching and the fixed-set GCC argument address a different question from dispersive or Strichartz bounds.

\subsection{History of the project and use of artificial intelligence}

The history of the project is described in Part I~\cite[Section~1.6]{AOSI}. In early August 2026, an exploration with ChatGPT pointed to the glancing and caustic structure in the disc. Its possible connection with AOS had not previously been considered by the second author. After studying this construction, the second author realized that its essential ingredients were local and could be adapted to general smooth bounded domains. ChatGPT models contributed extensively to the technical development of the construction, including asymptotic calculations, parameter analysis, and intermediate derivations. They were also used for figures and during LaTeX typesetting.

\medskip\noindent\emph{Organization of the paper.} Section~\ref{sec:glancing-disc} develops AOS via glancing propagation on the disc. Section~\ref{sec:curved} constructs the phases and exact waves near a positively curved boundary arc. Section~\ref{sec:gcc} proves the ray estimates for GCC and completes the \hyperref[thm:main]{Main Theorem}. Appendix~\ref{subsec:meas-boundary} discusses the generalized rays used in the local construction. Appendix~\ref{sec:phase} supplies the phase, Airy, and return calculations. We refer to \cite[Section~1.5]{AOSI} for the broader literature review.

\begingroup
\renewcommand{\addcontentsline}[3]{}
\subsection*{Acknowledgments}
\noindent Shengquan Xiang is partially supported by NSFC 12301562 and by the Shanghai Qiguang Natural Science Development Foundation. Jingrui Niu is supported by the start-up funding of HIT(No. AUGA5610303426). Part of this work was done when Jingrui Niu was visiting Peking University. They would like to thank the hospitality and the financial support of these institutions.
\endgroup

\Needspace{20\baselineskip}
\section{AOS via glancing on the disc}\label{sec:glancing-disc}

In this section, we develop AOS via glancing on the unit disc $\D:=\{x\in\R^2:|x|<1\}$. Unlike the radial reflection construction in \cite[Section~3]{AOSI}, the two branches propagate near the boundary. The method extends to general smooth domains in Section~\ref{sec:curved}.

We shall construct a compatible profile, localize its two normal branches in one moving packet region, and combine the matching regions at different scales into one fixed observation region. We prove finite-profile velocity approximation in this section. To recover the full profiles and obtain non-observability, we use the energy and profile-tail estimates of Section~\ref{sec:curved}; the \hyperref[eq:gcc]{GCC} argument is left to Section~\ref{sec:gcc}.

The section is organized as follows.
\begin{itemize}[leftmargin=*,itemsep=3pt]
\item Section~\ref{subsec:gl-propagation} derives the near-tangent geometry; Section~\ref{subsec:gl-compat} derives the profile compatibility.
\item Section~\ref{subsec:diamond} constructs the diamond profile and its matching rectangles $\mathcal R$.
\item Section~\ref{subsec:gl-family} defines $S,\Gamma$, proves finite-profile velocity approximation and varies the scales to obtain $S_k,\Gamma_k$.
\item Proposition~\ref{prop:gl-observation} assembles $G^{(N)}\downarrow G$, including the additional observation $\calE_k$.
\item Section~\ref{subsec:gl-disc-real} recovers the full profiles and obtains non-observability, using the constant-curvature case of the estimates in Section~\ref{sec:curved}.
\end{itemize}

\setcounter{subsection}{-1}
\subsection{The intuition of glancing propagation and AOS}\label{subsec:gl-intuition}

The roles of profiles, symmetry matching, scaled approximate velocities, and asymptotics remain those of \cite[Section~1.4]{AOSI}. Here, however, propagation takes place between the boundary and a nearby caustic. We explain the construction in seven steps. Background for the propagation, compatibility, velocity approximation, and localization discussed in Steps~1--4 and~7 is given in \cite[Appendix~A]{AOSI}.

\vspace{2mm} \noindent{\it 1. From radial propagation to boundary--caustic geometry.} As in \cite[Section~3.1]{AOSI}, consider a separated mode $e^{im\theta-i\omega t}R(r)$ with angular order $m$ and temporal frequency $\omega$. Its radial equation is
\begin{equation*}
R''+\frac1rR'+\omega^2\left(1-\frac{b^2}{r^2}\right)R=0 \textrm{ \; \; with \; }
 b=\frac{|m|}{\omega}.
\end{equation*}
The oscillatory region is determined by
\begin{equation*}
p_r^2=1-\frac{b^2}{r^2}>0 \quad\Longleftrightarrow\quad b<r<1.
\end{equation*}
In \cite[Section~3]{AOSI}, $b$ tends to zero. Here we take $b$ close to $1$: the ray travels between the boundary and a nearby caustic, with two signs of the normal momentum and a common tangential direction. The radial equation describes this normal structure, while propagation is nearly tangential.

\vspace{2mm} \noindent{\it 2. Compatible profiles.} Passing through the caustic and imposing the Dirichlet condition relate the two branch profiles. After choosing the boundary phase shift to be a quarter-period, these requirements become
\begin{equation*}
g=\calH f \text{ \; and \;  }f=\tau_{1/4}g=g(\cdot+1/4).
\end{equation*}
The Hilbert transform $\calH$ records the caustic transfer, while the translation $\tau_{1/4}$ records the boundary relation. Together they require $f=\tau_{1/4}\calH f$, defining the compatible class in Section~\ref{subsec:gl-compat}. In particular, we choose the compatible pair $f_*,g_*$ from the diamond construction of Section~\ref{subsec:diamond}.

\vspace{2mm} \noindent{\it 3. Generated approximate velocities $S_k$.} Use the angular coordinate $\theta$ and set $y=1-r$. The generated velocity has the form
\begin{equation*}
S_k=A_{\lambda_k}\left(f_*(\phi_k)-g_*(\psi_k)\right),
\end{equation*}
with a common amplitude specified in the next step. The two phases are
\begin{align*}
\phi_k=\frac{b_{\lambda_k}\theta-t+R_{\lambda_k}(y)}{2\pi h_k}+\alpha_k \textrm{ \; and \; }
 \psi_k=\frac{b_{\lambda_k}\theta-t-R_{\lambda_k}(y)}{2\pi h_k}+\alpha_k.
\end{align*}
Their tangential and temporal parts agree, while their normal actions have opposite signs.

\vspace{2mm} \noindent{\it 4. Localization of $S_k$ in the moving packet region $\mathcal U_k$.} The Airy model $\Ai\left(-\frac{\zeta_{\lambda_k}(y)}{h_k^{2/3}}\right)$ explains the normal concentration near the turning depth $y_c^{\lambda_k}\sim\lambda_k^2/2$, where $\zeta_{\lambda_k}$ is given by
\begin{equation*}
\begin{cases}
\zeta_{\lambda_k}(y)>0,&0<y<y_c^{\lambda_k},\\
\zeta_{\lambda_k}(y)<0,&y>y_c^{\lambda_k}.
\end{cases}
\end{equation*}
These are the oscillatory and decaying sides, respectively. For $S_k$, we retain the oscillatory-side expression and extend it by zero beyond the turning curve. Tangential localization is supplied by the smooth envelope 
\begin{equation*}
\chi\left(\frac{\theta-t-z_k}{w_k}\right) \textrm{ \; with \; } \chi\in C_c^\infty((-1/2,1/2);\R).
\end{equation*}

It moves at speed one, the common limiting tangential speed of the two branches. Set
\begin{gather*}
I_k:=z_k+(-w_k/2,w_k/2),\\
\mathcal U_k:=\{\theta-t\pmod{2\pi}\in I_k,\quad 0<y<y_c^{\lambda_k}\}.
\end{gather*}
On the angular lift satisfying $|\theta-t-z_k|<w_k/2$, the common amplitude is
\begin{equation*}
A_{\lambda_k}(t,\theta,y):=
\begin{cases}
c_k(\lambda_kw_k)^{-1/2}\chi\!\left(\dfrac{\theta-t-z_k}{w_k}\right)(\zeta_{\lambda_k}(y))^{-1/4},&\forall (t,x)\in\mathcal U_k,\\[1mm]
0,&\forall (t,x)\notin\mathcal U_k,
\end{cases}
\end{equation*}
where $c_k$ is the energy-normalizing factor supplied in Section~\ref{subsec:gl-disc-real}. Hence the generated velocity vanishes outside the moving packet region:
\begin{equation*}
S_k=A_{\lambda_k}\left(f_*(\phi_k)-g_*(\psi_k)\right)=\1_{\mathcal U_k}S_k.
\end{equation*}

Here $w_k$ is the tangential width, $\theta=t+z_k\pmod{2\pi}$ is the moving center and $y_c^{\lambda_k}\sim\lambda_k^2/2$ is the normal depth. The two normal branches are localized in a single moving packet region, whereas the radial construction of \cite[Section~3]{AOSI} uses two opposite angular packets. Under $\lambda_k^2/w_k\to0$, Lemma~\ref{lem:gl-finite} proves velocity approximation for the finite-profile version of this localized expression, including the normal zero extension.

\vspace{2mm} \noindent{\it 5. Symmetry matching inside the moving packet region $\mathcal U_k$.} The preceding localization gives support in $\mathcal U_k$. We now retain observation where the two profile values agree, producing additional cancellation inside the moving packet region. The diamond supplies two fixed reference rectangles,
\begin{equation*}
\mathcal R=(F_+\times J_+)\cup(F_-\times J_-), \quad\text{with }f_*(\xi)=g_*(\eta)=\pm1\quad\text{on }F_\pm\times J_\pm.
\end{equation*}
Pulling them back by the phases gives
\begin{align*}
\Gamma_k&=\mathcal U_k\cap\{(\phi_k,\psi_k)\pmod1\in\mathcal R\},\\*
\1_{\Gamma_k}S_k &=\1_{\Gamma_k}A_{\lambda_k}\left(f_*(\phi_k)-g_*(\psi_k)\right)=0.
\end{align*}

\begin{figure}[htbp]
\centering
\includegraphics[width=\textwidth]{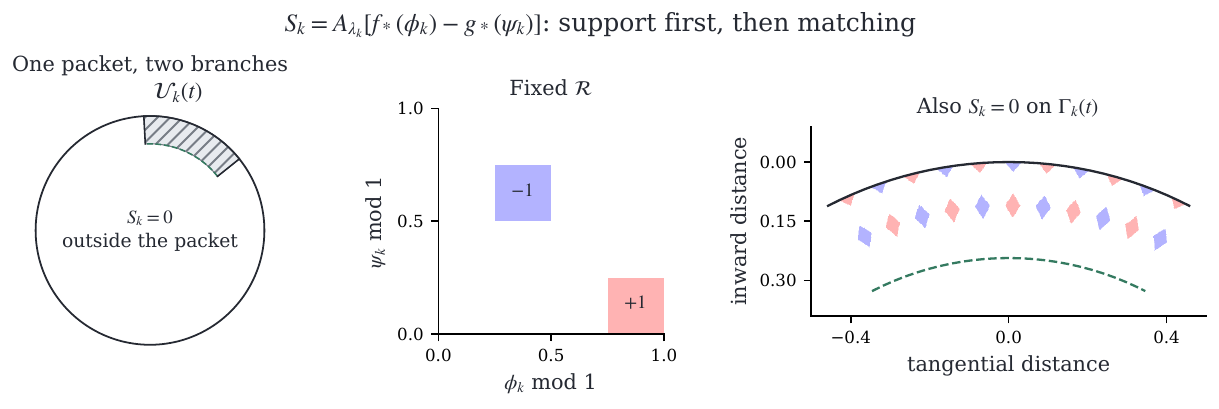}
\caption{Support and matching in the glancing disc construction. Hatching marks possible support in $\mathcal U_k(t)$; red and blue denote the reference values $f_*=g_*=+1$ and $f_*=g_*=-1$. The pullback $\Gamma_k(t)$ is enlarged in a tangential--inward orthonormal frame without normal stretching. White inside the moving packet region lies outside the chosen matching region; the dashed curve is the finite-order turning curve.}\label{fig:gl-intuition}
\end{figure}

\vspace{2mm} \noindent{\it 6. Scales and one fixed observation region.} Choose the intervals $I_k$ disjoint. Since all envelopes move with the coordinate $\theta-t$, the moving packet regions remain disjoint in spacetime. The profiles and $\mathcal R$ stay fixed while $\lambda_k,w_k\to0$ change the physical depth, width and matching cells.

For the later ray comparison we also retain a small additional observation region $\calE_k$ near the reference corner values; its precise phase pullback is \eqref{eq:gl-extra}. Set
\begin{gather*}
\mathcal H_k=\mathcal U_k\setminus(\Gamma_k\cup\calE_k),\\
G^{(N)}=\left((0,T)\times\D\right)\setminus\bigcup_{k\le N}\mathcal H_k ,\text{ \; and \; }G=\bigcap_{N\ge1}G^{(N)}.
\end{gather*}
We observe everything outside the moving packet regions and retain $\Gamma_k\cup\calE_k$ within each moving packet region already used. By disjointness,
\begin{equation*}
G^{(N)}\downarrow G,\text{ \; and \; } G^{(N)}\cap\mathcal U_k=G\cap\mathcal U_k=\Gamma_k\cup\calE_k \quad\forall k\in\{1,\ldots,N\}.
\end{equation*}
As Figure~\ref{fig:gl-intro-stages} illustrates, both the positions of the moving packet regions and their matching cells change with time.

\begin{figure}[htbp]
\centering
\includegraphics[width=\textwidth]{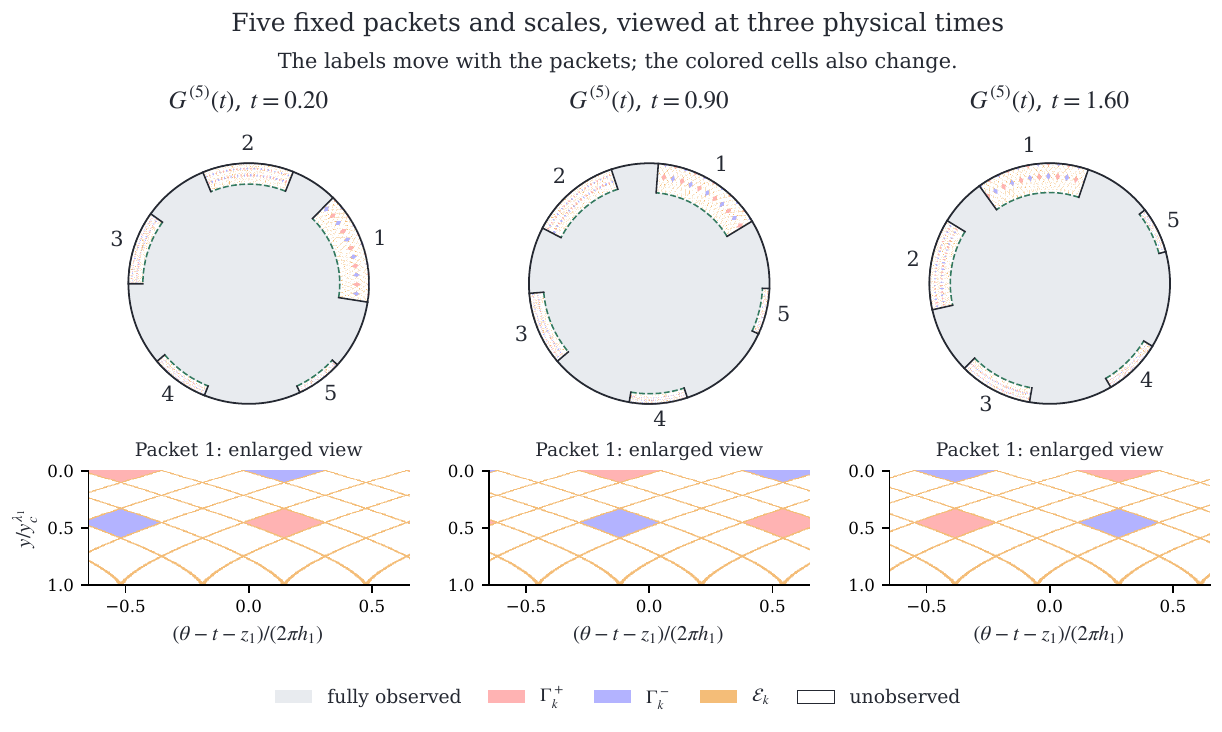}
\caption{The same $G^{(5)}$ at three physical times in the disc, with all parameters fixed. }\label{fig:gl-intro-stages}
\end{figure}

\vspace{2mm} \noindent{\it 7. Velocity approximation and non-observability.} Proposition~\ref{prop:gl-disc-real} realizes this diamond family by normalized solutions of \eqref{eq:wave}. With coordinated choices of the scales,
\begin{equation*}
E[u_k](0)=1,\quad \partial_tu_k=S_k+r_k ,\text{ \; and \; }\|r_k\|_{L^2((0,T)\times\D)}\rightarrow0.
\end{equation*}
Support and matching give, for $k\le N$,
\begin{equation*}
\1_{G^{(N)}}\partial_tu_k =\1_{G^{(N)}}r_k+\1_{\calE_k}S_k,
\end{equation*}
and the same identity holds on $G$. Proposition~\ref{prop:gl-disc-real} also controls the additional observed energy, so
\begin{equation*}
\|\1_G\partial_tu_k\|_{L^2((0,T)\times\D)} \le\|r_k\|_{L^2((0,T)\times\D)} +\|\1_{\calE_k}S_k\|_{L^2((0,T)\times\D)}\rightarrow0.
\end{equation*}

\subsection{Near-tangent geometry and the two branch phases}\label{subsec:gl-propagation}

Recall the separated modes from \cite[Section~3.1]{AOSI}:
\begin{equation*}
u(t,r,\theta)=\Real\left(e^{-i\omega t}e^{im\theta}R(r)\right).
\end{equation*}
Their radial equation is
\begin{equation*}
R''+\frac1rR'+\left(\omega^2-\frac{m^2}{r^2}\right)R=0.
\end{equation*}
With the angular orders fixed before the radial indices tend to infinity in \cite[Section~3.2]{AOSI}, we have $|m|/\omega\to0$, giving the phases $t-r$ and $t+r$ away from the origin.
\vspace{1.5mm}

To study another propagation regime, keep track of the ratio
\begin{equation}\label{eq:gl-ratio}
b:=\frac{|m|}{\omega} \text{ \;\; and  \; } \omega^{-2}\left(R''+\frac{1}{r}R'\right)+\left(1-\frac{b^2}{r^2}\right)R=0.
\end{equation}
The angular factor $e^{im\theta}$ retains the tangential oscillation; the radial equation describes the normal behavior of this nearly tangential wave. Let $W(r):=r^{\frac{1}{2}}R(r)$. Then, $W$ satisfies the following equation
\begin{equation*}
W''+\left(\omega^{2}\left(1-\frac{b^2}{r^2}\right)+\frac{1}{4r^2}  \right)W=0.
\end{equation*}
This becomes a standard form of the turning-point problem (see \cite[Chapter 11.3]{Olver}). The function $1-\frac{b^2}{r^2}$ has a zero at $r=b$, and its derivative does not vanish at $r=b$, i.e., a simple turning point. To construct the travelling wave solution, based on the WKB expansion, for $b\in(0,1)$, we introduce 
\begin{equation}\label{eq:circular-action}
R_b(r):=\int_b^r\sqrt{1-\frac{b^2}{\rho^2}}\dd\rho,\quad\forall r\in[b,1].
\end{equation}
A direct computation gives that $R_b(r)=\sqrt{r^2-b^2}-b\arccos(b/r)\text{, and }R_b(b)=0$. According to \cite{Olver}, for this turning-point problem, we also need to introduce the following Airy coordinate defined via the following equation
\begin{equation*}
\zeta(r)\left(\zeta'(r)\right)^2=1-\frac{b^2}{r^2}, \text{ \; with \; }\zeta'(r)>0.
\end{equation*}

Therefore, we obtain the following explicit formula
\begin{equation}
\zeta(r)=
\begin{cases}
\left(\frac{3}{2}R_b(r)\right)^{\frac{2}{3}},& \forall b\leq r<1;\\
-\left(\frac{3}{2}\int_r^b\sqrt{\frac{b^2}{\rho^2}-1}\dd\rho\right)^{\frac{2}{3}},&\forall  r<b.
\end{cases}
\end{equation}
Now we set $\widetilde{W}(r):=\left(r|\zeta'(r)|\right)^{-\frac{1}{2}}\Ai(-\omega^{\frac{2}{3}}\zeta(r))$, where $\Ai$ is the standard Airy function. The following lemma implies that $\widetilde{W}$ is the leading profile of \eqref{eq:gl-ratio} near the boundary.

\begin{lemma}\label{lem:W-form}
For $b$ sufficiently close to $1$, $|\zeta'|\sim1, |\zeta(r)|\sim|r-b|$. Moreover, there exists a constant $C>0$ such that $\widetilde{W}(r)=\left(r|\zeta'(r)|\right)^{-\frac{1}{2}}\Ai(-\omega^{\frac{2}{3}}\zeta(r))$ satisfies
\begin{equation*}
\left(\omega^{-2}(\partial_r^2+\frac{1}{r}\partial_r)+1-\frac{b^2}{r^2}\right)\widetilde{W} =\omega^{-2}c_b\widetilde{W},\quad \|c_b\|_{L^\infty}\le C.
\end{equation*}
\end{lemma}
\begin{proof}
The direct computation gives the proof. Notice the Airy equation $\Ai''(z)=z\Ai(z)$. Let $a_b(r):=\left(r|\zeta'(r)|\right)^{-\frac{1}{2}}$. We know that $a'_b(r)=-\frac{1}{2}a_b(r)(\frac{1}{r}+\frac{\zeta''(r)}{\zeta'(r)})$. Therefore,
\begin{align*}
\left(\omega^{-2}(\partial_r^2+\frac{1}{r}\partial_r)+1-\frac{b^2}{r^2}\right)\widetilde{W}=&a_b\left(-\zeta(\zeta')^2+1-b^2r^{-2}\right)\Ai(-\omega^{\frac{2}{3}}\zeta)\\
+&\omega^{-2}(a_b''+r^{-1}a_b')\Ai(-\omega^{\frac{2}{3}}\zeta)-\omega^{-\frac{4}{3}}(2a_b'\zeta'+a_b\zeta''+r^{-1}a_b\zeta')\Ai'(-\omega^{\frac{2}{3}}\zeta).
\end{align*}
Using the explicit formula of $a_b=\left(r|\zeta'(r)|\right)^{-\frac{1}{2}}$ and $c_b=\frac{a_b''+r^{-1}a_b'}{a_b}$, we obtain
\begin{equation*}
\left(\omega^{-2}(\partial_r^2+\frac{1}{r}\partial_r)+1-\frac{b^2}{r^2}\right)\widetilde{W} =\omega^{-2}c_b\widetilde{W},\quad \|c_b\|_{L^\infty}\le C.
\end{equation*}
\end{proof}
\begin{remark}
Heuristically, $b\theta-t\pm R_b(r)$ are exact ray phases (up to constants) on the annulus away from the turning point. We make the ansatz $R(r)=a(r,\omega)e^{i\omega \Phi(r)}$. Then, by the WKB method, we derive
\begin{equation*}
0= \left(1-\frac{b^2}{r^2}-(\Phi')^2\right)a+i\omega^{-1}\left(2\Phi'a'+\left(\Phi''+\frac{\Phi'}r\right)a\right) +\omega^{-2}\left(a''+\frac1ra'\right).
\end{equation*}
Hence, the eikonal equation gives $\Phi'=\pm\sqrt{1-\frac{b^2}{r^2}}$ (denoted by $p_r$ in Figure~\ref{fig:glancing}). Hence, the radial equation provides the phase $\pm R_b(r)$. Back to the wave mode (for example, $m>0$), $u(t,r,\theta)=\Real\left(e^{-i\omega t}e^{im\theta}R(r)\right)=\Real\left(a(r,\omega)e^{i\omega( b\theta-t\pm R_b(r))}\right)$. The associated oscillatory formulas for waves are asymptotic, and cease to be separate regular expressions at the caustic; see Figure~\ref{fig:glancing}. For a rigorous proof, see Lemma~\ref{lem:gl-airy}.
\end{remark}

\begin{figure}[tbp]
\centering
\includegraphics[width=.9\textwidth]{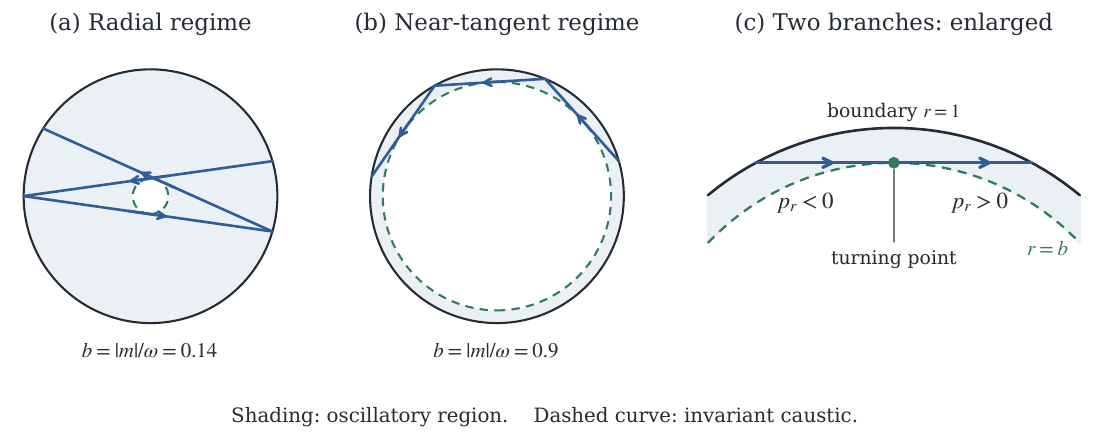}
\caption{Caustics and the two branches near the turning point.}\label{fig:glancing}
\end{figure}

\noindent The regular local model is the Airy function, satisfying
\begin{equation*}
\Ai''(z)=z\Ai(z),\text{ \; and  \;}\Ai(z)\rightarrow0\quad\text{as }z\to+\infty.
\end{equation*}
See also \cite[Appendix~A.4.2]{AOSI} for the connection from circular rays to the Airy equation. We keep the convention of the semiclassical parameter $h=\omega^{-1}$. One can find the Airy asymptotic formula in \cite{Olver}:
\begin{equation}\label{eq:airy-osc}
\Ai(-x)=\frac{1}{\sqrt{\pi}x^{\frac{1}{4}}}\left(\sin\left(\frac{2}{3}x^{\frac{3}{2}}+\frac\pi4\right)+O(x^{-\frac{3}{2}}) \right), \text{ for }x \text{ large and positive.}
\end{equation}
Recall that $\widetilde{W}(r)=\left(r|\zeta'(r)|\right)^{-\frac{1}{2}}\Ai\left(-\frac{\zeta(r)}{h^{\frac{2}{3}}}\right)$ and set $\widetilde{u}_h(t,r,\theta)=\Real\left(e^{-i\omega t}e^{im\theta}\widetilde{W}(r)\right) $. The next lemma shows the leading profile of the wave mode.
\begin{lemma}\label{lem:gl-airy}
Let $b\in(0,1)$ be sufficiently close to $1$, let $h>0$, and set $I_b:=\left[\frac{1+b}{2},1\right]$. There exists a constant $C=C(b,h)>0$ such that, for every $t,\theta\in\R$ and $r\in I_b$,
\begin{equation*}
\widetilde{u}_h(t,r,\theta) = \frac{h^{1/6}}{2\sqrt{\pi}(r^2-b^2)^{1/4}} \Real\left( -i\left( e^{\frac{i}{h}(b\theta-t+R_b(r))+\frac{i\pi}{4}} - e^{\frac{i}{h}(b\theta-t-R_b(r))-\frac{i\pi}{4}} \right) \right) +e_h(t,r,\theta),
\end{equation*}
where $\|e_h\|_{L^\infty(\R\times I_b\times\R)} \le C(b,h)$.
\end{lemma}
\begin{proof}
Set $r_b:=(1+b)/2$, so that $I_b=[r_b,1]$. For $r\in I_b$, we have
\[
\zeta(r)=\left(\frac32R_b(r)\right)^{2/3}>0, \;\; \;R_b(r)\ge R_b(r_b)>0.
\]
The remainder estimate in \eqref{eq:airy-osc} extends to all $x>0$ after enlarging the constant, since $\Ai(-x)$ is bounded for $0<x\le1$ and $x^{-1/4}\le x^{-7/4}$ there. Thus, for every $h>0$,
\begin{equation*}
\Ai\left(-\frac{\left(\frac{3}{2}R_b(r)\right)^{\frac{2}{3}}}{h^{\frac{2}{3}}}\right)=\frac{h^{\frac{1}{6}}}{\sqrt{\pi}|\zeta|^{\frac{1}{4}}}\left(\sin\left(\frac{R_b(r)}{h}+\frac{\pi}{4}\right)+O\left(\frac{h}{R_b(r)}\right)\right).
\end{equation*}

By the definition of $\zeta$, $\left|r^2(\zeta'(r))^2\zeta(r)\right|=\left|r^2\left(\frac{b^2}{r^2}-1\right)\right| =r^2-b^2$. Hence
\begin{equation*}
\widetilde W(r)=\frac{h^{1/6}}{\sqrt{\pi}(r^2-b^2)^{1/4}} \left(\sin\left(\frac{R_b(r)}h+\frac\pi4\right) +O\left(\frac h{R_b(r)}\right) \right).
\end{equation*}
Replacing the $\sin$ function, we have
\begin{equation*}
e^{-i\omega t}e^{im\theta}\widetilde{W}(r)=\frac{h^{1/6}}{2i\sqrt{\pi}(r^2-b^2)^{1/4}} \left( e^{\frac{i}{h}(b\theta-t+R_b(r))+\frac{i\pi}{4}} - e^{\frac{i}{h}(b\theta-t-R_b(r))-\frac{i\pi}{4}} \right) +\mathcal E_h(t,r,\theta),
\end{equation*}
where, uniformly in $t,\theta\in\R$ and $r\in I_b$,
\[
|\mathcal E_h(t,r,\theta)| \le C_0\frac{h^{7/6}}{(r^2-b^2)^{1/4}R_b(r)}.
\]
Taking real parts yields the asserted formula with $e_h:=\Real\mathcal E_h$. Since $r\ge r_b>b$ and $R_b$ is increasing,
\[
\|e_h\|_{L^\infty(\R\times I_b\times\R)} \le C_0\frac{h^{7/6}}{(r_b^2-b^2)^{1/4}R_b(r_b)} =:C(b,h)<\infty.
\]
This proves the lemma. 
\vspace{1.5mm}

To verify the Dirichlet boundary condition, we need $\widetilde{W}(1)=0$, which requires that $\Ai\left(-\frac{\zeta(1)}{h^{2/3}}\right)=0$. Here we choose a positive zero $z_j$ of $\Ai(-z)$, with $j=j(b)\to\infty$ as $b\to1$, and choose $h$ accordingly. Later we will deal with their superposition in Section~\ref{subsec:gl-compat}. Then, we have $3R_b(1)=2hz_j^{3/2}>0$. On the other hand, we have the explicit formula for $R_b(1)$. We introduce a small parameter $\lambda>0$ by setting $b=\cos\lambda$, so that the limit $b\to1$ corresponds to $\lambda\to0$. Therefore,
\begin{align*}
R_b(1)=\sqrt{1-b^2}-b\arccos(b)=\sin\lambda-\lambda\cos\lambda=\frac13\lambda^3+O(\lambda^4).
\end{align*}
This implies that $R_b(1)\sim\frac13\lambda^3$ and
\[
\frac{h}{\lambda^3} = \frac{3R_b(1)}{2\lambda^3z_j^{3/2}} \sim\frac{1}{2z_j^{3/2}} \rightarrow0.
\]
Uniformly for $r\in I_b=[(1+b)/2,1]$, we have $r^2-b^2\sim\lambda^2,\;R_b(r)\sim\lambda^3$. Hence the asymptotic parameter satisfies $h/R_b(r)\sim h/\lambda^3\to0$ uniformly on $I_b$. Therefore,
\begin{equation*}
\frac{h^{7/6}}{(r^2-b^2)^{1/4}R_b(r)} \sim\left(\frac{h}{\lambda^3}\right)^{7/6} \rightarrow0.
\end{equation*}
We conclude that
\begin{equation*}
\widetilde{u}_h(t,r,\theta) = \frac{h^{1/6}}{2\sqrt{\pi}(r^2-b^2)^{1/4}} \Real\left(-i\left(e^{\frac{i}{h}(b\theta-t+R_b(r))+\frac{i\pi}{4}}-e^{\frac{i}{h}(b\theta-t-R_b(r))-\frac{i\pi}{4}} \right)\right)+O\left(\left(\frac{h}{\lambda^3}\right)^{7/6}\right),
\end{equation*}
uniformly for $r\in I_b$ and $t,\theta\in\R$. The leading profile is a combination of two travelling waves. We will use this fact to construct the approximate solutions later.
\end{proof}

\subsection{Deriving the profile compatibility}\label{subsec:gl-compat}

The periodic Hilbert transform is defined by the Fourier multiplier
\begin{equation*}
\widehat{\calH q}(n)=-i\,\mathrm{sgn}(n)\widehat q(n),\quad \forall n\in\Z,
\end{equation*}
with $\mathrm{sgn}(0)=0$. We follow the convention of \cite[Section~2]{AOSI}, and set the compatible profile space\footnote{$L_1=\Id\quad\text{and}\quad L_2=\calH$ are the branch maps.}
\begin{equation}\label{eq:gl-space}
X:=\{q\in L^2(\R/\Z;\R):\widehat q(0)=0,\ q=\tau_{1/4}\calH q\}.
\end{equation}
We now derive this choice from the caustic and boundary relations, then identify its Fourier frequencies. From Lemma~\ref{lem:gl-airy}, for a single mode, we construct the travelling wave profile. Now we want to deal with the superposition of different modes: 
\begin{align*}
\sum_{\ell} \frac{d_\ell B_\ell(r)}{2i}\left(e^{\frac{i}{h_\ell}\left(b_\ell\theta-t+R_{b_\ell}(r)\right)+\frac{i\pi}{4}}-e^{\frac{i}{h_\ell}\left(b_\ell\theta-t-R_{b_\ell}(r)\right)-\frac{i\pi}{4}}\right),
\end{align*}
where $d_\ell$ are the superposition coefficients and $B_\ell(r):=\frac{h_\ell^{1/6}}{\sqrt{\pi}(r^2-b_\ell^2)^{1/4}}$ is the leading radial amplitude of the $\ell$th mode for $r>b_\ell$. Then, we explain these constraints for a profile in $X$.
\begin{enumerate}
    \item \textit{The relation imposed by the caustic.} Let $H\in L^2(\R/\Z;\mathbb{C})$ have only positive Fourier frequencies and set
\begin{equation}\label{eq:fg}
f(\xi):=\Real\left(e^{i\pi/4}H(e^{2\pi i\xi})\right) \text{ \; and \; } g(\xi):=\Real\left(e^{-i\pi/4}H(e^{2\pi i\xi})\right).
\end{equation}
Then $f,g\in L^2_{\mathrm{loc}}(\R)$ are $1$-periodic real functions. Since $H$ has only positive frequencies, the Hilbert transform acts by
\begin{equation*}
\calH f=\Real(-ie^{i\pi/4}H)=\Real(e^{-i\pi/4}H)=g.
\end{equation*}
Thus the caustic fixes $g$ once $f$ is chosen. This relation allows the two profiles to arise from the same Airy sum, through its two branch factors.
\item \textit{The relation required at the boundary.} The boundary imposes a second constraint. If $\xi$ denotes the first normalized phase at $r=1$, the branch difference there is proportional to
\begin{equation*}
f(\xi)-g\left(\xi-\frac{R_b(1)}{\pi h}\right).
\end{equation*}
The caustic relation is $g=\calH f$. We choose the boundary phase shift to be a quarter-period:
\begin{equation}\label{eq:lock-intuition}
-\frac{R_b(1)}{\pi h}=\frac14\pmod1.
\end{equation}
The boundary and caustic requirements therefore combine as
\begin{equation*}
f=\tau_{1/4}g\quad\text{and}\quad g=\calH f \quad\rightarrow\quad f=\tau_{1/4}\calH f.
\end{equation*}
Here the quarter-period is chosen, and the Hilbert-transform relation is imposed by propagation. At this stage the boundary relation concerns the leading profiles. Lemma~\ref{lem:gl-finite} imposes the exact Dirichlet condition on the approximating waves by using Airy zeros. Moreover, Section~\ref{sec:curved} extends the construction to varying curvature.
\end{enumerate}

The following proposition characterizes the space $X$ defined in \eqref{eq:gl-space}.%
\begin{prop}[Profile characterization]\label{prop:loop}
Let $f\in L^2(\R/\Z;\R)$ with $\widehat f(0)=0$. Then $f\in X$ if and only if its positive Fourier coefficients vanish outside $n\equiv1\pmod4$.
\end{prop}

\begin{proof}
For $n>0$, compatibility reads
\begin{equation*}
(1+ie^{\pi in/2})\widehat f(n)=0,\quad -ie^{\pi in/2}=1\ \Longleftrightarrow\ n\equiv1\pmod4.
\end{equation*}
The negative-frequency condition is the complex conjugate, and the constant term is zero by hypothesis. This proves both directions.
\end{proof}
We would like to retain the plateau geometry of the radial example in \cite[Section~1.4]{AOSI}: equal constant values on two intervals give a rectangle of matching phase pairs.

\begin{lemma}\label{lem:step-loop}
A real mean-zero finite-step function satisfying $f=\tau_{1/4}\calH f$ is zero.
\end{lemma}

\begin{proof}
Compatibility would make $\calH f=\tau_{-1/4}f$ bounded. However, for a circle interval $(a,b)$,
\begin{equation*}
\calH\1_{(a,b)}(\xi) =\frac1\pi\log\left|\frac{\sin\pi(\xi-a)}{\sin\pi(\xi-b)}\right|.
\end{equation*}
Every nonzero jump therefore gives a logarithmic singularity that jumps at other locations cannot cancel. Hence $f$ has no jumps; its zero mean then gives $f=0$.
\end{proof}

\subsection{Constructing the diamond profile}\label{subsec:diamond}

The preceding section determines which profiles are compatible with propagation. We now choose a profile in this class with a positive-measure matching region. Lemma~\ref{lem:step-loop} shows that the construction in \cite[Section~1.4]{AOSI} cannot directly apply to our current situation. Therefore, we introduce a new example profile.
\begin{defi}[Diamond profile]\label{def:diamond}
Set
\begin{equation}\label{eq:diamond}
H_*(z):=\frac{\sqrt2}{I_0}\int_0^z\frac{dw}{\sqrt{1-w^4}},\quad\forall z\in\D,
\end{equation}
where $I_0:=\int_0^1\frac{dr}{\sqrt{1-r^4}}$ is a normalizing constant. The corresponding boundary profiles are
\begin{equation*}
f_*(\xi):=\Real\left(e^{i\pi/4}H_*(e^{2\pi i\xi})\right) \text{ \; and \; } g_*(\xi):=\Real\left(e^{-i\pi/4}H_*(e^{2\pi i\xi})\right),
\end{equation*}
as in \eqref{eq:fg}. We call $q=f_*$ the diamond profile.
\end{defi}

\begin{prop}\label{prop:diamond}
We have the following properties
\begin{enumerate}
    \item The function $H_*$ defined in \eqref{eq:diamond} extends continuously to the closed disc and satisfies $H_*(iz)=iH_*(z)$. 
    \item $f_*,g_*$ satisfy $g_*=\calH f_*, f_*=\tau_{1/4}g_*\text{ \; and \; }|f_*|,|g_*|\leq1$.
    \item $f_*=\pm1$ on $F_\pm$ and $g_*=\pm1$ on $J_\pm$, where
\begin{equation}\label{eq:plateaus}
F_+=[3/4,1],\quad F_-=[1/4,1/2],\quad J_+=[0,1/4],\quad J_-=[1/2,3/4],
\end{equation}
are understood modulo one. Consequently the reference matching region
\begin{equation}\label{eq:rectangles}
\mathcal R:=(F_+\times J_+)\cup(F_-\times J_-),\quad\text{with }|\mathcal R|=1/8,
\end{equation}
has $f_*(\xi)=g_*(\eta)=\pm1$ on these two rectangles.
\end{enumerate}
\end{prop}

\noindent\emph{Idea of diamond profile.} The diamond meets two requirements. First, the compatibility class requires
\begin{equation*}
H(z)=\sum_{m\ge0}d_{4m+1}z^{4m+1} \quad\Longleftrightarrow\quad H(iz)=iH(z).
\end{equation*}
Second, we need intervals on which the two real projections have equal constant values. Writing $H=H_{\rm R}+iH_{\rm I}$, these projections are
\begin{equation*}
f=\frac{H_{\rm R}-H_{\rm I}}{\sqrt2},\text{ \; and \; }g=\frac{H_{\rm R}+H_{\rm I}}{\sqrt2}.
\end{equation*}
Thus a straight side on $H_{\rm R}-H_{\rm I}=\pm\sqrt2$ makes $f=\pm1$, and a side on $H_{\rm R}+H_{\rm I}=\pm\sqrt2$ makes $g=\pm1$. These four lines bound a diamond, preserved by a quarter-turn, which satisfies both geometric requirements.

Therefore, we construct $H_*$ by considering the Riemann mapping from the unit disc to a square. Place the preimages of the four corners at $z_0\in\{1,i,-1,-i\}$. Integration with the normalization $H_*(1)=\sqrt2$ gives \eqref{eq:diamond}.

\begin{proof}[Proof of Proposition~\ref{prop:diamond}]
By the definition in \eqref{eq:diamond}, for any $w\in\D\subset\mathbb{C}$, $\Real(1-w^4)\geq 1-|w|^4>0$. Thus, we may choose a holomorphic branch of the square root on the right half-plane such that $q(z):=(1-z^4)^{-\frac{1}{2}}$ is holomorphic in $\D$ and $q(0)=1$. By Cauchy's integral theorem, we know that $H_*$ in \eqref{eq:diamond} is well-defined. We expand $q$ into a power series $(1-w^4)^{-\frac{1}{2}}=\sum_{m=0}^{\infty}\frac{\binom{2m}{m}}{4^m}w^{4m}$. Then,
\begin{equation}\label{eq:diamond-coeff}
H_*(z)=\sum_{m=0}^{\infty}d_{*,4m+1}z^{4m+1}, \text{ \;\; where }d_{*,4m+1}:=\frac{\sqrt2}{I_0}\frac{\binom{2m}{m}}{4^m(4m+1)},\quad\forall m\geq0,
\end{equation}
and other coefficients $d_{*,n}\equiv0$, $\forall n\ge1\text{ \; with \; }n\not\equiv1\pmod4$. By Stirling's formula, we have $\frac{\binom{2m}{m}}{4^m}\sim\frac{1}{\sqrt{\pi m}}$ as $m\to\infty$. Thus, the coefficients are $O\left((m+1)^{-3/2}\right)$, the series converges uniformly on the closed disc, and $H_*$ extends continuously to the circle. We also have $H_*(iz)=iH_*(z)$ and $H_*(1)=\sqrt2$.
\vspace{1.5mm}

For $f_*$ and $g_*$ defined in \eqref{eq:fg}, the relation $H_*(iz)=iH_*(z)$ gives
\begin{align*}
f_*(\xi+1/4)&=\Real\left(ie^{i\pi/4}H_*(e^{2\pi i\xi})\right)=-g_*(\xi),\\
g_*(\xi+1/4)&=\Real\left(ie^{-i\pi/4}H_*(e^{2\pi i\xi})\right)=f_*(\xi).
\end{align*}
For $n>0$, direct computation gives
\begin{align*}
\widehat{\calH f_*}(n)=e^{-i\frac{\pi}{2}}\widehat{f_*}(n)=e^{-i\frac{\pi}{2}}\widehat{\Real\left(e^{i\frac{\pi}{4}}H_*(e^{2\pi i\xi})\right)}(n)=\widehat{g_*}(n).
\end{align*}
For negative frequencies the identity follows by complex conjugation, and both sides vanish at $n=0$. Thus
\begin{equation}\label{eq:diamond-rel}
g_*=\calH f_*,\quad f_*(\xi)=g_*(\xi+1/4),\text{ \; and \; }|f_*|,|g_*|\le1.
\end{equation}

To verify the value on edges, for $0<\beta<\frac{1}{4}$, we write $z=e^{2\pi i\beta}$, and direct differentiation gives
\begin{equation*}
1-e^{8\pi i\beta}=-2i e^{4\pi i\beta}\sin(4\pi\beta), \quad \frac{d}{d\beta}H_*(e^{2\pi i\beta})=\frac{2\pi}{I_0}\frac{e^{i\frac{3\pi}{4}}}{\sqrt{\sin(4\pi\beta)}}.
\end{equation*}
Therefore, $\frac{d}{d\beta}g_*(\beta)=\Real\left(e^{-i\pi/4}\frac{d}{d\beta}H_*(e^{2\pi i\beta})\right)=\Real\left(i\frac{2\pi}{I_0}\frac{1}{\sqrt{\sin(4\pi\beta)}}\right)=0$, for $0<\beta<\frac{1}{4}$. This implies that $g_*$ is a constant on $J_+$. Using that $g_*(0)=1$, we know that $g_*=1$ on $J_+$.  The other three sides follow from $H_*(iz)=iH_*(z)$ and the relations between $f_*$ and $g_*$. %
Then, we can easily deduce that
\begin{equation*}
f_*(\xi)-g_*(\eta)=
\begin{cases}
1-1=0,&(\xi,\eta)\in F_+\times J_+,\\
-1-(-1)=0,&(\xi,\eta)\in F_-\times J_-.
\end{cases}
\end{equation*}
The rectangles are disjoint and have total area $2(1/4)^2=1/8$, proving \eqref{eq:rectangles}; see Figure~\ref{fig:diamond}.
\end{proof}

\begin{figure}[H]
\centering
\includegraphics[width=.9\textwidth]{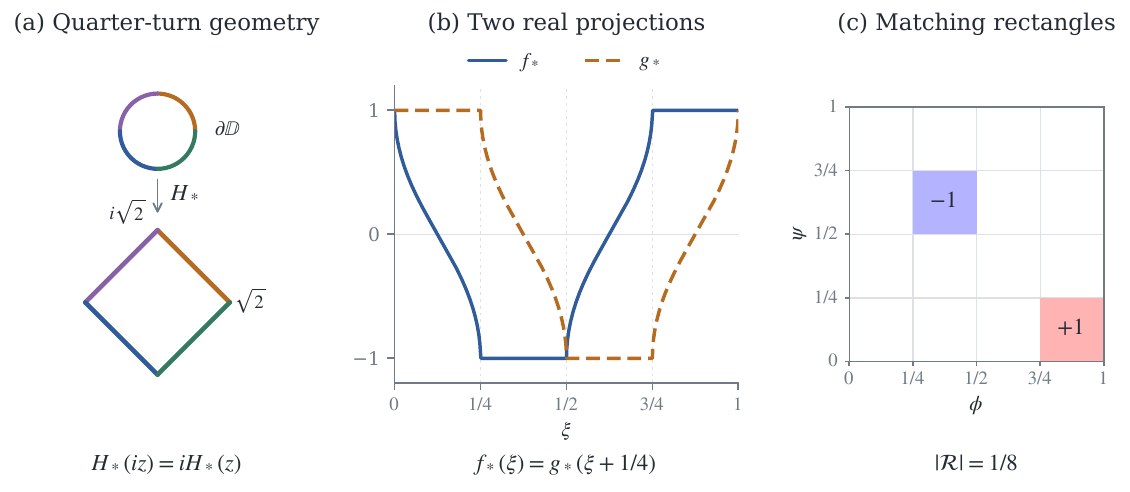}
\caption{The diamond determines both the symmetry and the matching region. A quarter-turn acts on source and image together, giving $f_*(\xi)=g_*(\xi+1/4)$. Red and blue mark the equal-sign plateau pairs in \eqref{eq:rectangles}, whose total area is $1/8$.}\label{fig:diamond}
\end{figure}

The reference geometry is now fixed: $q=f_*\in X$, the branch profiles are $f_*$ and $g_*=\calH f_*$, and we choose $\calM_q=\mathcal R$.

\subsection{Profile symmetry, scaled approximate velocities, and matching regions}\label{subsec:gl-family}

We now construct generated approximate velocities $S_k$ and their matching regions in physical space. The normal oscillatory region becomes thinner with the scale. We proceed in five steps.
\begin{itemize}[leftmargin=*,itemsep=3pt]
\item Choose phases $\phi,\psi$ retaining the quarter-period boundary shift.
\item Localize both branches in one moving packet region with a common amplitude $A_{\lambda,w}$, and define $S_{\lambda,w,z,\alpha}$.
\item Pull back the reference rectangles $\mathcal R$ to obtain the physical matching region $\Gamma_{\lambda,w,z,\alpha}$.
\item Prove velocity approximation for each fixed finite-profile expression.
\item Vary the scales while keeping $f_*,g_*$ and $\mathcal R$ fixed, obtaining the family $S_k,\Gamma_k$.
\end{itemize}

\subsubsection{Phases, scales, and the moving packet region}
By Lemma~\ref{lem:gl-airy}, the exact ray phases are $b\theta-t\pm R_b(r)$. In what follows, we would like to work in boundary normal coordinates $(\theta,y)$, where $y=1-r$ denotes the inward normal distance to the boundary. We also recall some notations before proceeding to further discussion. From this point onward, we redefine $\lambda$ by the boundary action, writing $\beta$ for the angle used in Lemma~\ref{lem:gl-airy}:
\begin{align*}
\lambda:=(3R_b(1))^{1/3},\quad R_b(1)=\frac{1}{3}\lambda^3.
\end{align*}

For $b=\cos{\beta}$ close to $1$, by expanding the series at $\beta=0$, we obtain the relation:
\begin{equation*}
b=1-\frac{1}{2}\lambda^2+\frac{1}{120}\lambda^4+O(\lambda^6).
\end{equation*}
Let $\zeta(y)$ denote the exact signed Airy coordinate in the normal variable $y=1-r$, which is smooth across its turning point $y=1-b\sim\frac{1}{2}\lambda^2$. Near the boundary, we write
\begin{equation*}
\zeta(y)=\zeta_{\lambda}(y)+O(\lambda^6)  \textrm{ \; and \; } \zeta_{\lambda}(y)=2^{1/3}\left(\frac{\lambda^2}{2}-y-\frac{3}{10}y^2+\frac{2}{15}\lambda^2y\right),\quad\text{for }y\in(0,y_c^\lambda).
\end{equation*}
Here $y_c^\lambda$ denotes the first positive zero of the polynomial $\zeta_\lambda$. Then, we use the approximate phases $\Psi^{\pm}_{\lambda}$ defined via
\begin{equation}\label{eq:gl-disc-phase}
b_{\lambda}:=1-\frac{1}{2}\lambda^2+\frac{1}{120}\lambda^4  \textrm{ \; and \; }  \Psi^{\pm}_{\lambda}(t,\theta,y):=b_\lambda\theta-t\pm R_{\lambda}(y),
\end{equation}
where $R_{\lambda}(y):=\frac{2}{3}(\zeta_\lambda(y))^{3/2}$. These formulas specialize \eqref{eq:W0} and \eqref{eq:airy-coeff} to $\kappa=1$; Section~\ref{sec:curved} treats varying curvature.
\vspace{1.5mm}

We next choose the semiclassical parameter $h$ for the period-one profiles. The proof of Lemma~\ref{lem:gl-finite} requires two errors to vanish: $\lambda^6/h$ from the disc phase defect, and $h/\lambda^4$ from comparing the modewise Airy phases with the common phases. We take $h\sim\lambda^{9/2}$ and coordinate it with the boundary action by setting
\begin{equation}\label{eq:gl-disc-scales}
j:=\lceil\lambda^{-3/2}\rceil   \textrm{ \; and \; } h:=\frac{\lambda^3}{3\pi(j-1/4)}.
\end{equation}
The boundary phase difference is therefore
\begin{align*}
-\frac{R_\lambda(0)}{\pi h}=-\frac{\lambda^3}{3\pi h}=-(j-1/4)=1/4\pmod1.
\end{align*}
Thus the scaled phases retain the quarter-period boundary relation exactly. Since $j\sim\lambda^{-3/2}$, the two errors satisfy
\begin{equation*}
\frac{\lambda^6}{h}\sim\lambda^{3/2}\rightarrow 0 \textrm{ \; and \; }
\frac h{\lambda^4}\sim\lambda^{1/2}\rightarrow 0.
\end{equation*}
The same choice also gives $\lambda^5/h\sim\lambda^{1/2}\to0$, the stronger residual requirement used for varying curvature in Section~\ref{subsubsec:curved-scales}. The integer $j$ becomes the base Airy-zero index in Lemma~\ref{lem:gl-finite}.

Finally, choose a tangential width $w$ and a center moving along $\theta=t+z$. Combining this interval with the normal layer $0<y<y_c^\lambda$ gives the moving packet region used below.
\begin{defi}[Moving packet region]\label{def:gl-packet}
Fix $T>0$, a sufficiently small $\lambda>0$, a width $0<w<1$ and a center $z\in\R/2\pi\Z$. Define
\begin{equation}\label{eq:gl-disc-packet}
\mathcal U_{\lambda,w,z}:=\left\{(t,x)\in(0,T)\times\D:
\begin{aligned}
&\theta(x)-t-z\pmod{2\pi}\in(-w/2,w/2),\\
&0<y(x)<y_c^\lambda
\end{aligned}
\right\}.
\end{equation}
\end{defi}
Fix a $z\in\R$. On $\mathcal U_{\lambda,w,z}$, we can choose the unique lift of $\theta$ satisfying $|\theta-t-z|<w/2$, due to the choice $w<1<2\pi$. At time $t$, the moving packet region has tangential coordinate width $w$, center $\theta=t+z\pmod{2\pi}$ and normal depth $y_c^\lambda\sim\frac{1}{2}\lambda^2$.

\subsubsection{Generated approximate velocities and matching regions}\label{subsubsec:gl-packet}

We now define the generated velocity on this moving packet region and identify its matching region. For a real cutoff function $\chi\in C_c^\infty((-1/2,1/2))$, we define the localized travelling wave in $\mathcal U_{\lambda,w,z}$. Its amplitude has three factors:
\begin{equation}\label{eq:gl-disc-amp}
A_{\lambda,w}=\underbrace{(\lambda w)^{-1/2}}_{\text{mass scaling}} \underbrace{\chi\left(\frac{\theta-t-z}{w}\right)}_{\text{tangential localization}} \underbrace{(\zeta_\lambda)^{-1/4}}_{\text{Airy branch amplitude}} \quad\text{ on \; }\mathcal U_{\lambda,w,z}.
\end{equation}
We extend it by zero outside the moving packet region. The first factor compensates for the mass scale; the exact energy normalization will use an additional constant $c_k$. We set
\begin{gather}
S_{\lambda,w,z,\alpha}(t,\theta,y):=A_{\lambda,w}(t,\theta,y)\left(f_*(\frac{\Psi^+_{\lambda}}{2\pi h}+\alpha)-g_*(\frac{\Psi^-_{\lambda}}{2\pi h}+\alpha)\right). \label{eq:gl-disc-S}
\end{gather}
Extend $S_{\lambda,w,z,\alpha}$ by zero outside the moving packet region and on its turning curve. Define the physical matching regions by $\Gamma_{\lambda,w,z,\alpha}:=\Gamma_{\lambda,w,z,\alpha}^+\cup\Gamma_{\lambda,w,z,\alpha}^-$, where
\begin{equation}\label{eq:gl-disc-Gamma}
\begin{split}
\Gamma_{\lambda,w,z,\alpha}^+&:=\mathcal U_{\lambda,w,z}\cap\{\frac{\Psi^+_{\lambda}}{2\pi h}+\alpha\pmod1\in F_+,\ \frac{\Psi^-_{\lambda}}{2\pi h}+\alpha\pmod1\in J_+\},\\
\Gamma_{\lambda,w,z,\alpha}^-&:=\mathcal U_{\lambda,w,z}\cap\{\frac{\Psi^+_{\lambda}}{2\pi h}+\alpha\pmod1\in F_-,\ \frac{\Psi^-_{\lambda}}{2\pi h}+\alpha\pmod1\in J_-\}.
\end{split}
\end{equation}
\begin{prop}\label{prop:gl-support}
Fix $T>0$ and let $\lambda>0$ be sufficiently small. Choose the diamond profiles $f_*,g_*$ and rectangles $\mathcal R$ from Proposition~\ref{prop:diamond}, and a real nonzero envelope $\chi\in C_c^\infty((-1/2,1/2))$,  and $\mathcal U_{\lambda,w,z}$ in Definition~\ref{def:gl-packet}. Then $S_{\lambda,w,z,\alpha}\in L^2((0,T)\times\D)$, and the following two properties hold.
\begin{enumerate}[label=(\roman*),leftmargin=*,itemsep=1mm]
\item \emph{(Support in the moving packet region)}. The function $S_{\lambda,w,z,\alpha}$ vanishes outside $\mathcal U_{\lambda,w,z}$.%
\item \emph{(Additional cancellation through profile matching)}. The two full profiles $f_*,g_*$ have equal values on the reference rectangles, and
\begin{equation}\label{eq:gl-disc-zero}
\1_{\Gamma_{\lambda,w,z,\alpha}}S_{\lambda,w,z,\alpha}=0.
\end{equation}
\end{enumerate}
\end{prop}

\begin{proof}
Support follows from the cutoff and zero extension in the definition. At each fixed time, 
\begin{align*}
\|A_{\lambda,w}(t,\cdot)\|_{L^2(\D)}^2=\frac1{\lambda w}\int\chi\left(\frac{\theta-t-z}{w}\right)^2\dd \theta\int_0^{y_c^\lambda}(\zeta_\lambda(y))^{-1/2}(1-y)\dd y\sim\frac{w\lambda}{\lambda w}=1.
\end{align*}
Since $f_*,g_*$ are bounded, $S_{\lambda,w,z,\alpha}$ belongs to spacetime $L^2$. Now we verify the additional cancellation through profile matching. By the definition of $\Gamma_{\lambda,w,z,\alpha}^{\pm}$ and Proposition~\ref{prop:diamond}, 
\begin{equation*}
S_{\lambda,w,z,\alpha}=
\begin{cases}
A_{\lambda,w}(1-1)=0,&\text{on }\Gamma_{\lambda,w,z,\alpha}^+,\\
A_{\lambda,w}\left(-1-(-1)\right)=0,&\text{on }\Gamma_{\lambda,w,z,\alpha}^-.
\end{cases}
\end{equation*}
\end{proof}

\begin{figure}[htbp]
\centering
\includegraphics[width=.8\textwidth,
  trim={0mm 3mm 0mm 2mm},
  clip]{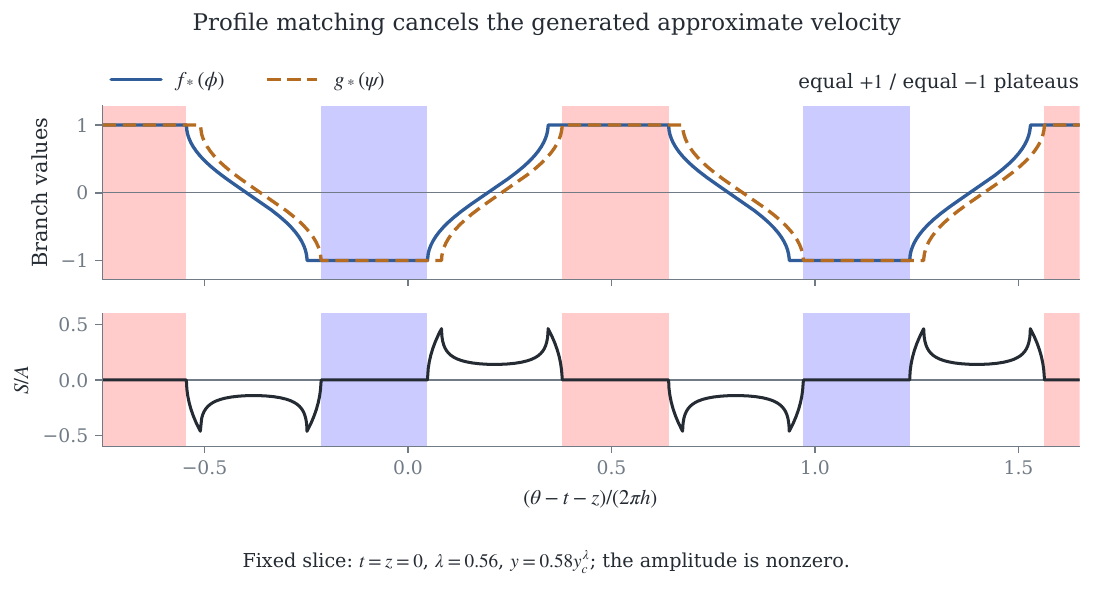}
\caption{Above: $f_*(\phi)$ and $g_*(\psi)$. Below: $S/A=f_*(\phi)-g_*(\psi)$}
\end{figure}

For our phases, we have $\frac{\phi+\psi}{2}=\frac{b_\lambda \theta-t}{2\pi h}+\alpha$, and $\frac{\phi-\psi}{2}=\frac{R_\lambda(y)}{2\pi h}$. Boundary position and time move the two arguments together; normal position separates them. In particular, at fixed time,
\begin{equation*}
\partial_\theta\phi=\partial_\theta\psi=\frac{b_\lambda}{2\pi h} \text{ \; and \; }\partial_y\phi=-\partial_y\psi=\frac{\partial_yR_\lambda}{2\pi h}.
\end{equation*}
Thus both physical variables enter the matching, and the reference rectangles become curved cells in the thin annular moving packet region; see Figure~\ref{fig:gl-slices}.

\begin{figure}[htbp]
\centering
\includegraphics[width=.8\textwidth,
  trim={0mm 3mm 0mm 1mm},
  clip]{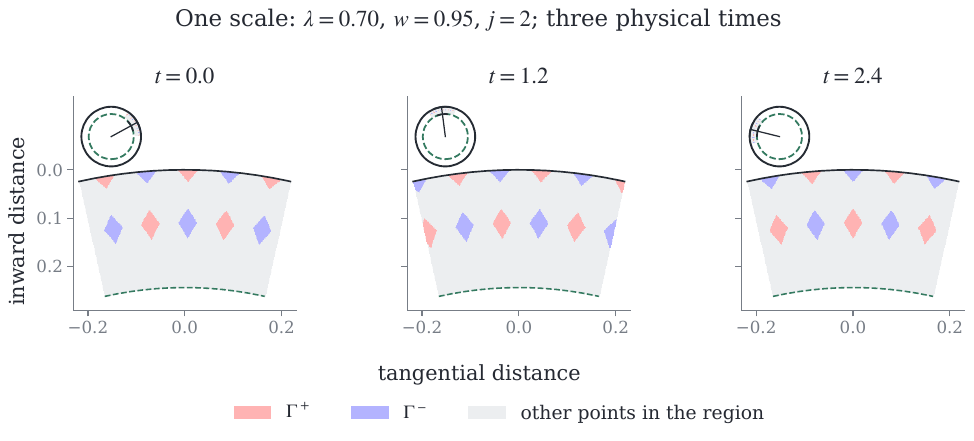}
\caption{Three time slices of one disc matching region.}\label{fig:gl-slices}
\end{figure}

\subsubsection{Finite-profile realization}\label{subsubsec:gl-finite}

The next lemma proves velocity approximation for the finite-profile version of this localized expression, including its zero extension beyond the turning curve. Fourier truncation preserves the linear compatibility relations, but need not preserve the plateau regions. The full profiles in Proposition~\ref{prop:gl-support} will therefore continue to define the matching region. See also \cite[Appendices~A.3 and~A.5.2]{AOSI} for the discussion of strict localization and Airy assembly.

\begin{lemma}\label{lem:gl-finite}
Fix $T,\chi,z,\alpha$ as in Proposition~\ref{prop:gl-support} and an integer $\ell\ge1$. Let
\begin{equation*}
H_*^{(\ell)}(e^{2\pi i\xi}):=\sum_{\substack{0<n\le\ell\\n\equiv1 \mod{4}}}d_{*,n}e^{2\pi in\xi},
\end{equation*}
where $d_{*,n}$ are the coefficients in \eqref{eq:diamond-coeff}. Define the travelling profiles by
\begin{equation*}
f_*^{(\ell)}(\xi):=\Real\left(e^{i\pi/4}H_*^{(\ell)}(e^{2\pi i\xi})\right), \text{ \; and \; }g_*^{(\ell)}(\xi):=\Real\left(e^{-i\pi/4}H_*^{(\ell)}(e^{2\pi i\xi})\right).
\end{equation*}
Use the scales \eqref{eq:gl-disc-scales} and let $0<w=w(\lambda)<1$ satisfy $\lambda^2/w\to0$. With the convention in Proposition~\ref{prop:gl-support}, define
\begin{equation}\label{eq:gl-finite-S}
S_{\lambda,w,\ell}:=A_{\lambda,w}\left(f_*^{(\ell)}(\frac{\Psi^+_{\lambda}}{2\pi h}+\alpha)-g_*^{(\ell)}(\frac{\Psi^-_{\lambda}}{2\pi h}+\alpha)\right).
\end{equation}
There are real solutions $u_{\lambda,\ell}$ of \eqref{eq:wave} on the disc such that
\begin{equation}\label{eq:gl-finite-real}
\|\partial_tu_{\lambda,\ell}-S_{\lambda,w,\ell}\|_{L^2((0,T)\times\D)}\rightarrow0 \quad\text{as }\lambda\to0^+.
\end{equation}
\end{lemma}

\noindent\emph{Idea of approximation.} For fixed $\ell$, we construct a smooth Airy sum $v_\ell$ satisfying the Dirichlet condition. Its velocity gives the two profiles on the oscillatory side, while integral estimates control the turning region and the normal tail. We then correct its small wave-equation residual:
\begin{equation*}
S_{\lambda,w,\ell} \underset{\substack{\text{Airy}\\\text{comparison}}}{\approx} \partial_tv_\ell \underset{\substack{\text{energy}\\\text{correction}}}{\approx} \partial_tu_{\lambda,\ell}.
\end{equation*}
More precisely, the proof controls the two terms in
\begin{align*}
\|\partial_tu_{\lambda,\ell}-S_{\lambda,w,\ell}\|_{L^2((0,T)\times\D)} \le\|\partial_tv_\ell-S_{\lambda,w,\ell}\|_{L^2((0,T)\times\D)}+C_T\|(\partial_t^2-\Delta)v_\ell\|_{L^1(0,T;L^2(\D))}.
\end{align*}

\begin{proof}[Proof of Lemma~\ref{lem:gl-finite}]
The Fourier cutoff $\ell$ is fixed throughout the proof. 

\noindent\emph{Step 1. An Airy sum satisfying the boundary condition.} As in Lemma~\ref{lem:gl-airy}, let $z_j$ be the positive zeros of $\Ai(-z)$. Their asymptotics \cite{Olver} give
\begin{equation*}
z_j^{3/2}=\frac{3\pi}{2}(j-1/4)+O(j^{-1}).
\end{equation*}
For each retained frequency, choose
\begin{equation*}
h_n:=h/n,\quad j_n:=nj-(n-1)/4 ,\text{ \; and \; }\lambda_n^3:=2h_nz_{j_n}^{3/2}.
\end{equation*}
The condition $n\equiv1\pmod4$ makes $j_n\in\Nat$ an integer and
\begin{equation*}
\Ai\left(-\zeta_{\lambda_n}(0)/h_n^{2/3}\right)=\Ai(-z_{j_n})=0, \lambda_n^3=\lambda^3\left(1+O_\ell(j^{-2})\right).
\end{equation*}
Thus the boundary condition is exact, while the modewise normal scales approach the common scale.
\vspace{1.5mm}

Fix $1<Y_0<Y_1$ and a real cutoff $\rho\in C^{\infty}$ such that $\rho=1$ on $[0,Y_0]$ and $\rho=0$ on $[Y_1,\infty)$. For $\theta\in\R$, set
\begin{equation}\label{eq:gl-airy-wave}
v_\ell^{\rm lift}:=\Real\left( -2\sqrt{\frac{\pi}{\lambda w}} \chi\left(\frac{\theta-t-z}{w}\right)\rho\left(\frac{y}{\lambda^2}\right) \sum_{\substack{0<n\le\ell\\n\equiv1\mod{4}}}d_{*,n}h_n^{5/6} \Ai\left(-\frac{\zeta_{\lambda_n}(y)}{h_n^{2/3}}\right) e^{i\frac{b_{\lambda_n}\theta-t}{h_n}+2\pi in\alpha} \right).
\end{equation}
The factor $h_n^{5/6}$ gives a common leading velocity amplitude, since $h_n^{5/6}h_n^{1/6}h_n^{-1}=1$. On the support of the normal cutoff derivatives, $\zeta_{\lambda_n}\le -c\lambda^2$ for sufficiently small $\lambda$. The Airy function and its derivatives decay exponentially there. Thus $\rho$ removes a decaying tail of the smooth approximate wave. The zero extension of the target at the turning curve is justified separately in the next step.

Periodize both the phase and the envelope by setting
\begin{equation*}
v_\ell(t,\theta,y):=\sum_{m\in\Z}v_\ell^{\rm lift}(t,\theta-2\pi m,y).
\end{equation*}
Since $w<1$, at most one summand is active at each point. This defines a smooth real function on the disc, extended by zero beyond the normal collar, with $v_\ell|_{\partial\D}=0$. On the moving packet region it uses exactly Definition~\ref{def:gl-packet}. 

\vspace{2mm} \noindent\emph{Step 2. Recovery of the common profile.} We first compare the velocities away from the turning curve, then control the mass near it. The Airy zeros slightly change the phases. For $t\in[0,T]$, the explicit formulas \eqref{eq:gl-disc-phase} give
\begin{equation*}
\frac{|(b_{\lambda_n}-b_\lambda)\theta|}{h_n} \le C_{\ell,T}\frac{\lambda^2j^{-2}}h \le C_{\ell,T}\frac h{\lambda^4}\rightarrow0.
\end{equation*}
On a regular oscillatory band $y_c^\lambda-y\ge\delta\lambda^2$, with $\delta>0$ fixed, they also give
\begin{equation*}
\frac{|R_{\lambda_n}-R_\lambda|}{h_n}\le C_{\ell,\delta}j^{-1} ,\text{ \; and \; } \frac{\zeta_{\lambda_n}}{\zeta_\lambda}=1+O_{\ell,\delta}(j^{-2}).
\end{equation*}

Consequently each mode in \eqref{eq:gl-airy-wave} has a common amplitude and common phases. Summing its two branches gives
\begin{align*}
\Real\sum_{\substack{0<n\le\ell\\n\equiv1\mod{4}}}d_{*,n} \left(e^{i\pi/4}e^{2\pi in\phi}-e^{-i\pi/4}e^{2\pi in\psi}\right)=f_*^{(\ell)}(\phi)-g_*^{(\ell)}(\psi),
\end{align*}
where $\phi:=\frac{b_\lambda \theta-t+R_\lambda(y)}{2\pi h}+\alpha \text{ \; and \; } \psi:=\frac{b_\lambda \theta-t-R_\lambda(y)}{2\pi h}+\alpha$. The derivative of the envelope contributes $O_{\ell,T}(h/w)$ in spacetime $L^2$, which tends to zero. This proves the comparison with \eqref{eq:gl-finite-S} away from the turning curve on the oscillatory side. On the other side, Airy decay gives a vanishing contribution.
\vspace{1.5mm}

To include the turning curve, we use integral estimates. The change of variable $x=-\zeta_{\lambda_n}/h_n^{2/3}$ and the Airy bounds yield
\begin{align*}
\int_0^{Y_1\lambda^2}|\Ai(-\zeta_{\lambda_n}/h_n^{2/3})|^2\dd y \le C_\ell h_n^{1/3}\lambda,\;\int_0^{Y_1\lambda^2}|\Ai'(-\zeta_{\lambda_n}/h_n^{2/3})|^2\dd y \le C_\ell h_n^{-1/3}\lambda^3.
\end{align*}
More precisely, on $|y-y_c^\lambda|\le\delta\lambda^2$ the first integral is at most $C_\ell h_n^{1/3}\lambda \left(\sqrt\delta+j^{-1}+(h/\lambda^3)^{1/3}\right)$. Here the last two terms account for the displacement of the modewise turning curves and the Airy transition scale. The target has the same $\sqrt\delta$ bound because its squared normal amplitude is proportional to $(\zeta_\lambda)^{-1/2}$, whose integral on the strip is $O(\lambda\sqrt\delta)$.

Let $\mathcal V_{\lambda,\delta}:=\{(t,x):0<t<T,\ |y(x)-y_c^\lambda|\le\delta\lambda^2\}$. Including the factor $(\lambda w)^{-1}$ and integrating the squared envelope in $\theta$, we obtain
\begin{equation*}
\limsup_{\lambda\to0^+}\left( \|\1_{\mathcal V_{\lambda,\delta}}\partial_tv_\ell\|_{L^2}^2 +\|\1_{\mathcal V_{\lambda,\delta}}S_{\lambda,w,\ell}\|_{L^2}^2\right) \le C_{\ell,T}\sqrt\delta.
\end{equation*}
Combining this with the regular-band comparison, then letting $\delta\to0^+$, proves
\begin{equation}\label{eq:gl-finite-leading}
\|\partial_tv_\ell-S_{\lambda,w,\ell}\|_{L^2((0,T)\times\D)}\rightarrow0.
\end{equation}

\vspace{2mm} \noindent\emph{Step 3. The remainder and the exact wave.} In the disc coordinates, write $r=1-y$. Then
\begin{equation*}
\Delta=\partial_y^2-r^{-1}\partial_y+r^{-2}\partial_{\theta}^2.
\end{equation*}
For one mode, abbreviate $b_n=b_{\lambda_n}$, $\zeta_n=\zeta_{\lambda_n}$, $\eta=(\theta-t-z)/w$ and
\begin{equation*}
B_n(y):=\Ai\left(-\zeta_n(y)/h_n^{2/3}\right),\; \vartheta_n:=(b_n\theta-t)/h_n+2\pi n\alpha, \text{ \; and \; }F_n:=\chi(\eta)B_ne^{i\vartheta_n}.
\end{equation*}
Before the normal cutoff, direct differentiation gives
\begin{equation*}
(\partial_t^2-\Delta)F_n=e^{i\vartheta_n}\left( \frac{D_n}{h_n^2}\chi B_n +\frac{2i(1-r^{-2}b_n)}{h_nw}\chi'B_n +\frac{\zeta_n''-r^{-1}\zeta_n'}{h_n^{2/3}}\chi\Ai'(-\zeta_n/h_n^{2/3}) +\frac{1-r^{-2}}{w^2}\chi''B_n \right),
\end{equation*}
where $D_n:=b_n^2r^{-2}+\zeta_n(\zeta_n')^2-1$. On the normal collar, \eqref{eq:gl-disc-phase} gives
\begin{equation*}
D_n=O_\ell(\lambda^6),\;1-r^{-2}b_n=O_\ell(\lambda^2),\; \zeta_n''-r^{-1}\zeta_n'=O_\ell(1), \text{ \; and \; } 1-r^{-2}=O(\lambda^2).
\end{equation*}
Applying the two Airy integral bounds, and summing over the fixed modes, we deduce
\begin{equation*}
\|(\partial_t^2-\Delta)v_\ell\|_{L^1(0,T;L^2(\D))}\le C_{\ell,T}\left(\frac{\lambda^6}{h}+\lambda+\frac{\lambda^2}{w} +\frac{h\lambda^2}{w^2}\right)+o(1)\rightarrow0.
\end{equation*}
The normal cutoff contributes exponentially small terms. All the displayed terms vanish: $h\sim\lambda^{9/2}$, $\lambda^2/w\to0$, and
\begin{equation*}
\frac{h\lambda^2}{w^2}=\frac h{\lambda^2}\left(\frac{\lambda^2}{w}\right)^2\rightarrow0.
\end{equation*}
Periodization preserves these bounds on every fixed finite time interval.

Let $u_{\lambda,\ell}$ be the exact Dirichlet wave with initial state $v_\ell[0]$. The difference has zero initial and boundary data and forcing $-(\partial_t^2-\Delta)v_\ell$. The energy estimate and \eqref{eq:gl-finite-leading} therefore give
\begin{align*}
\|\partial_tu_{\lambda,\ell}-S_{\lambda,w,\ell}\|_{L^2((0,T)\times\D)} \le C_T\|(\partial_t^2-\Delta)v_\ell\|_{L^1(0,T;L^2(\D))}+\|\partial_tv_\ell-S_{\lambda,w,\ell}\|_{L^2((0,T)\times\D)} \rightarrow0.
\end{align*}
This finishes the proof of the lemma.
\end{proof}

\subsubsection{The same symmetry at different scales}

Choose $w_k\to0^+$ and then $\lambda_k\to0^+$ with $\lambda_k^2/w_k\to0$. Let $h_k$ be given by \eqref{eq:gl-disc-scales}, and allow center parameters $z_k$ and common phase shifts $\alpha_k$. After multiplication by the normalization constants $c_k$, for simplicity, we write
\begin{equation}\label{eq:gl-disc-family}
\begin{aligned}
A_{k}&:=c_kA_{\lambda_k,w_k}, S_k:=c_kS_{\lambda_k,w_k,z_k,\alpha_k}=A_{k}(f_*(\phi_k)-g_*(\psi_k)),\\
\Gamma_k&:=\Gamma_{\lambda_k,w_k,z_k,\alpha_k},\ \1_{\Gamma_k}S_k=0.
\end{aligned}
\end{equation}
Set $\theta=t+z_k+w_k\sigma$. Then
\begin{equation}\label{eq:gl-moving}
\begin{split}
\phi_k&=\frac{b_{\lambda_k}(z_k+w_k\sigma)+(b_{\lambda_k}-1)t+R_{\lambda_k}(y)}{2\pi h_k}+\alpha_k,\\
\psi_k&=\frac{b_{\lambda_k}(z_k+w_k\sigma)+(b_{\lambda_k}-1)t-R_{\lambda_k}(y)}{2\pi h_k}+\alpha_k.
\end{split}
\end{equation}
Thus the envelope moves at unit speed, but the internal matching pattern also changes through the common phase drift $\frac{(b_{\lambda_k}-1)t}{2\pi h_k}$. 

\begin{remark}\label{rem:gl-S-params}
We choose 
\begin{equation*}
S_k=c_k\,S\!\left(q,\chi;z_k,\alpha_k;\lambda_k,w_k\right),\text{ where }q=f_*, \text{ \; and \; } g_*=\calH q.
\end{equation*}
Here $S$ is given by \eqref{eq:gl-disc-S}. The profile $q$ and envelope $\chi$ are fixed; $z_k$ determines the position of the moving packet region, while $\alpha_k$ is the common phase shift of the two profiles.

The parameters $\lambda_k,w_k$ determine the normal depth and tangential width, while \eqref{eq:gl-disc-scales} fixes $h_k$ from $\lambda_k$. Thus the main propagation scales are $(h_k,\lambda_k^2,w_k)$, where $h_k\sim\lambda_k^{9/2}$, and $\lambda_k^2/w_k\to0$. Proposition~\ref{prop:gl-disc-real} provides bounded energy-normalizing factors $c_k$ using Fourier cutoffs.
\end{remark}

\begin{figure}[H]
\centering
\includegraphics[width=\textwidth]{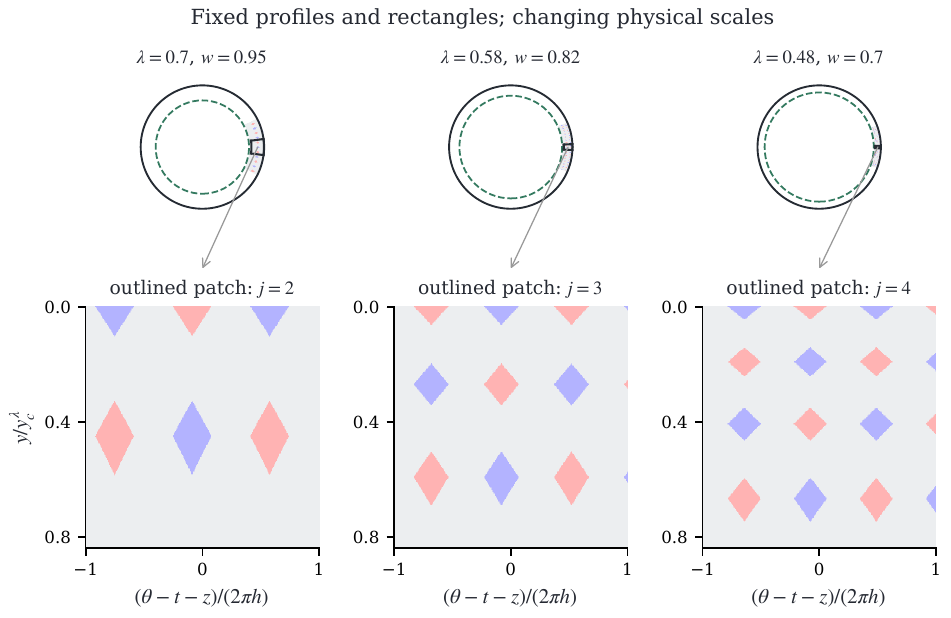}
\caption{One fixed diamond symmetry at three scales.}\label{fig:gl-scales}
\end{figure}

\subsection{A fixed observation region}\label{subsec:gl-observation}

We combine the moving packet regions at all scales into one observation region. The additional regions $\calE_k$ are pullbacks of strips near the reference corner values. They accommodate the phase errors in the ray comparison of Section~\ref{sec:gcc}. All sets below are subsets of $(0,T)\times\D$. Before we state our next proposition, we recall the notations. The diamond profiles $f_*,g_*$ and a real cutoff $\chi$, and $S_k,\Gamma_k$ are as in Proposition~\ref{prop:gl-support}. The moving packet regions of Definition~\ref{def:gl-packet} are
\begin{equation*}
\mathcal U_k:=\mathcal U_{\lambda_k,w_k,z_k} =\{(t,x):\theta(x)-t\pmod{2\pi}\in I_k,\quad 0<y(x)<y_c^{\lambda_k}\}.
\end{equation*}

Choosing reference widths $\gamma_k\to0^+$, set $\mathcal K:=\{0,1/4,1/2,3/4\}$, and define
\begin{equation}\label{eq:gl-extra}
\calE_k:=\{(t,x)\in\mathcal U_k:\dist(\phi_k,\mathcal K)\le4\gamma_k \text{ or } \dist(\psi_k,\mathcal K)\le4\gamma_k\}.
\end{equation}
The distances here are taken on $\R/\Z$. Define the hidden part of each moving packet region by
\begin{equation*}
\mathcal H_k:=\mathcal U_k\setminus(\Gamma_k\cup\calE_k).
\end{equation*}
The finite-stage and fixed observation regions are
\begin{equation}\label{eq:gl-disc-G}
G^{(N)}:=\left((0,T)\times\D\right)\setminus\bigcup_{k\le N}\mathcal H_k,\;G:=\bigcap_{N\ge1}G^{(N)} =\left((0,T)\times\D\right)\setminus\bigcup_{k\ge1}\mathcal H_k.
\end{equation}
\begin{prop}\label{prop:gl-observation}
Fix $T>0$ and adopt the conventions above. Choose centers $z_k$ and widths $0<w_k<1$, with $w_k\to0^+$, so that the intervals
\begin{equation}\label{eq:gl-intervals}
I_k:=z_k+(-w_k/2,w_k/2)\subset\R/2\pi\Z
\end{equation}
are pairwise disjoint. Fix common phase shifts $\alpha_k\in\R/\Z$, sufficiently small scales $\lambda_k\to0^+$ and factors $c_k>0$. Then, these Borel sets satisfy
\begin{equation}\label{eq:gl-disc-stages}
\begin{gathered}
G^{(N)}\downarrow G,\\
G^{(N)}\cap\mathcal U_k=G\cap\mathcal U_k=\Gamma_k\cup\calE_k, \quad \forall k\in\{1,\ldots,N\}.
\end{gathered}
\end{equation}
Moreover,
\begin{equation}\label{eq:gl-disc-observed}
\1_{G^{(N)}}S_k=\1_GS_k=\1_{\calE_k}S_k, \quad \forall k\in\{1,\ldots,N\}.
\end{equation}
\end{prop}

\begin{remark}\label{rem:gl-params}
With the propagation model fixed, the set $G$ can be written as
\begin{equation}\label{eq:gl-G-param}
G=G\left(q,T;\{z_k,\alpha_k\}_{k\ge1};\{\lambda_k,w_k,\gamma_k\}_{k\ge1}\right), \text{ where }q=f_*\quad\text{and}\quad g_*=\calH q.
\end{equation}
The positions and propagation scales are as in Remark~\ref{rem:gl-S-params}. The additional reference width $\gamma_k$ changes $\calE_k$ and $G$ but not $S_k$. Proposition~\ref{prop:gl-disc-real} allows any $\gamma_k\to0$ and Section~\ref{sec:gcc} chooses it as a function of $\lambda_k$ for GCC.
\end{remark}

\begin{proof}[Proof of Proposition~\ref{prop:gl-observation}]
Note that $I_k$ are disjoint intervals, which implies disjoint moving packet regions in spacetime, regardless of their normal depths. Proposition~\ref{prop:gl-support}(ii) gives $\1_{\Gamma_k}S_k=0$. We also retain $\calE_k$, the physical pullback of strips of reference width $4\gamma_k$. 

Now we try to combine the scales. The set $G^{(N)}$ observes everything outside the first $N$ moving packet regions and retains $\Gamma_k\cup\calE_k$ inside each of them. Adding one scale gives
\begin{align*}
G^{(N+1)} &=G^{(N)}\setminus\mathcal H_{N+1}\\
&=\left(G^{(N)}\setminus\mathcal U_{N+1}\right) \cup\Gamma_{N+1}\cup\calE_{N+1}.
\end{align*}
The earlier moving packet regions are disjoint from $\mathcal H_{N+1}$, proving \eqref{eq:gl-disc-stages}.

By Proposition~\ref{prop:gl-support}(i), $S_k$ vanishes outside $\mathcal U_k$. Thus, for $k\le N$,
\begin{align*}
\1_{G^{(N)}}S_k=\1_{G^{(N)}\cap\mathcal{U}_k}S_k=\1_{\Gamma_k\cup\calE_k}S_k=\1_{\calE_k}S_k.
\end{align*}
The last equality uses $\1_{\Gamma_k}S_k=0$. The same calculation on $G$ proves \eqref{eq:gl-disc-observed}.

\end{proof}

The observation stages are illustrated in Figure~\ref{fig:gl-stages}. In contrast with \cite[equation~(3.45)]{AOSI}, $\1_GS_k=\1_{\calE_k}S_k$ need not vanish. Its smallness requires an energy estimate because the amplitudes concentrate. Proposition~\ref{prop:gl-disc-real} gives this estimate and non-observability; Section~\ref{sec:gcc} makes further parameter choices to obtain \hyperref[eq:gcc]{GCC} on the same $G$.

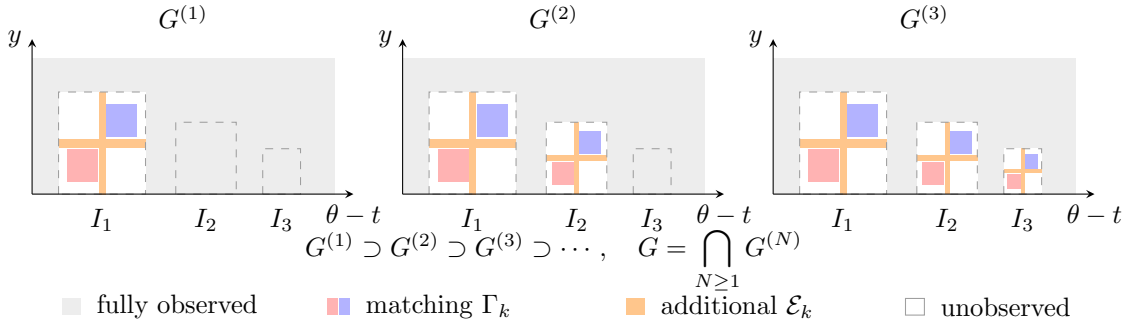
\begin{figure}[htbp]
\centering
\begin{tikzpicture}[x=1cm,y=1cm,>=stealth,font=\small]
  \foreach \stage/\shift in {1/0,2/4.9,3/9.8}{
    \begin{scope}[xshift=\shift cm]
      \node at (2.0,2.35) {$G^{(\stage)}$};
      \fill[gray!14] (0,0) rectangle (4.0,1.8);
      \foreach \k/\left/\width/\depth in {1/0.35/1.15/1.35,2/1.9/0.8/0.95,3/3.05/0.5/0.6}{
        \begin{scope}[shift={(\left,0)},xscale=\width,yscale=\depth]
          \ifnum\k>\stage
          \else
            \fill[white] (0,0) rectangle (1,1);
            \fill[red!30] (0.10,0.12) rectangle (0.45,0.44);
            \fill[blue!30] (0.55,0.56) rectangle (0.90,0.88);
            \fill[orange!45] (0.46,0) rectangle (0.54,1);
            \fill[orange!45] (0,0.46) rectangle (1,0.54);
          \fi
          \draw[gray!75,dashed] (0,0) rectangle (1,1);
        \end{scope}
        \node[anchor=north] at ({\left+\width/2},-0.06) {$I_{\k}$};
      }
      \draw[->] (0,0) -- (4.25,0) node[below] {$\theta-t$};
      \draw[->] (0,0) -- (0,2.05) node[left] {$y$};
    \end{scope}
  }
  \node at (6.9,-0.88) {$G^{(1)}\supset G^{(2)}\supset G^{(3)}\supset\cdots,\quad G=\displaystyle\bigcap_{N\ge1}G^{(N)}$};
  \fill[gray!14] (0.4,-1.6) rectangle (0.65,-1.38);
  \node[anchor=west] at (0.72,-1.49) {fully observed};
  \fill[red!30] (3.9,-1.6) rectangle (4.04,-1.38);
  \fill[blue!30] (4.06,-1.6) rectangle (4.2,-1.38);
  \node[anchor=west] at (4.28,-1.49) {matching $\Gamma_k$};
  \fill[orange!45] (7.85,-1.6) rectangle (8.1,-1.38);
  \node[anchor=west] at (8.18,-1.49) {additional $\calE_k$};
  \draw[gray!75] (11.55,-1.6) rectangle (11.8,-1.38);
  \node[anchor=west] at (11.88,-1.49) {unobserved};
\end{tikzpicture}

\caption{Finite-stage observation regions at a fixed time, in coordinates $(\theta-t,y)$.}\label{fig:gl-stages}
\end{figure}

\Needspace{8\baselineskip}
\subsection{Exact waves and the asymptotic realization formula}\label{subsec:gl-disc-real}

We now realize $S_k$ by time derivatives of normalized solutions of \eqref{eq:wave} and show that their observation on $G$ tends to zero.

\begin{prop}\label{prop:gl-disc-real}
Fix $T>0$, the diamond profiles $f_*,g_*$ from Proposition~\ref{prop:diamond}, and a real cutoff $\chi\in C_c^\infty((-1/2,1/2))$. Prescribe widths $0<w_k<1$ with $w_k\to0$, centers $z_k\in\R/2\pi\Z$, common phase shifts $\alpha_k\in\R/\Z$, and $h_k$ determined by \eqref{eq:gl-disc-scales}. For any positive upper bounds $\overline\lambda_k$, one can choose Fourier cutoffs $\ell_k\to\infty$, then sufficiently small scales satisfying
\begin{equation*}
0<\lambda_k\le\overline\lambda_k,\;\lambda_k\to0, \text{ \; and \; }\lambda_k^2/w_k\to0.
\end{equation*}
There are constants $0<c_*\le c_k\le C_*$ and real solutions $u_k$ of \eqref{eq:wave} with the following properties.

\begin{enumerate}[label=(\roman*),leftmargin=*,itemsep=1.5mm]
\item \emph{(Asymptotic realization).} With $S_k=c_kS_{\lambda_k,w_k,z_k,\alpha_k}$ and $\Gamma_k$ from Proposition~\ref{prop:gl-support},
\begin{equation}\label{eq:gl-disc-real}
E[u_k](0)=1,  \; \partial_tu_k=S_k+r_k, \text{ \; and \; } \|r_k\|_{L^2((0,T)\times\D)}\rightarrow0.
\end{equation}
The same profiles $f_*,g_*$ and reference matching region $\mathcal R$ are used for every $k$. In particular,
\begin{equation*}
\1_{\Gamma_k}S_k=0,  \; \1_{\Gamma_k}\partial_tu_k=\1_{\Gamma_k}r_k, \text{ \; and \; } \liminf_k\|S_k\|_{L^2((0,T)\times\D)}>0.
\end{equation*}
\item \emph{(Observation on the fixed set and non-observability).} If the intervals $I_k$ in \eqref{eq:gl-intervals} are pairwise disjoint, choose any $\gamma_k\to0^+$ and define $G^{(N)},G$ by Proposition~\ref{prop:gl-observation}, using these same scales. Then
\begin{equation}\label{eq:gl-disc-wave-obs}
\1_{G^{(N)}}\partial_tu_k =\1_{G^{(N)}}r_k+\1_{\calE_k}S_k, \; \forall k\in\{1,\ldots,N\},
\end{equation}
and the same identity holds with $G$ in place of $G^{(N)}$. Moreover,
\begin{equation*}
\|\1_{\calE_k}S_k\|_{L^2((0,T)\times\D)}^2 \le C_T(\gamma_k+h_k/w_k)\rightarrow0.
\end{equation*}
Consequently,
\begin{equation*}
\|\1_G\partial_tu_k\|_{L^2((0,T)\times\D)}^2 \le2\|r_k\|_{L^2((0,T)\times\D)}^2+C_T(\gamma_k+h_k/w_k) \rightarrow0,
\end{equation*}
so observability fails on $G$. Here $C_T$ is independent of $k$ and the reference widths $\gamma_k$.
\end{enumerate}
\end{prop}

\noindent\emph{Idea of realization.} Lemma~\ref{lem:gl-finite} realizes each fixed finite-profile expression by an exact wave. Its proof first reads $S_{\lambda,w,\ell}$ from the two Airy branches and then corrects the residual of $v_\ell$. To recover the full diamond, we need a profile-tail estimate uniform in the scale and energy bounds uniform along the chosen sequence of Fourier cutoffs. The same bounded factor $c_k$ then normalizes the wave and the full-profile expression $S_k$; the latter retains exact matching.

We give a sketch of these remaining steps. The uniform energy and propagated-tail estimates are proved in Sections~\ref{subsubsec:curved-energy}--\ref{subsubsec:curved-tail} for general domains and apply here with constant curvature.

\begin{proof}[Sketch of the proof]
For fixed $\ell$, take the approximate and exact waves constructed in Lemma~\ref{lem:gl-finite}. The periodization in its proof is valid for every fixed $T>0$. With $\kappa=1$, the energy and propagated-tail estimates in Sections~\ref{subsubsec:curved-energy}--\ref{subsubsec:curved-tail} give
\begin{gather*}
c\le E[v_\ell](0)\le C\quad\text{for sufficiently small $\lambda$ at each fixed $\ell$},\\
\|S_{\lambda,w,\ell}-S_{\lambda,w,z,\alpha}\|_{L^2((0,T)\times\D)} \le C_T\left(\|f_*^{(\ell)}-f_*\|_{L^2}+\|g_*^{(\ell)}-g_*\|_{L^2}\right).
\end{gather*}
The constants are independent of $\ell$ in this order of choices, and the profile norms are on $\R/\Z$. Then, choose $\ell_k\to\infty$, then $\lambda_k\le\overline\lambda_k$ sufficiently small for the finite-profile error in \eqref{eq:gl-finite-real} and the required energy bounds, with $\lambda_k\to0$ and $\lambda_k^2/w_k\to0$. Set $c_k=E[v_{\ell_k}](0)^{-1/2}$ and $u_k=c_ku_{\lambda_k,\ell_k}$. Therefore, we have
\begin{align*}
\|\partial_tu_k-S_k\|_{L^2((0,T)\times\D)} \le c_k\|\partial_tu_{\lambda_k,\ell_k}-S_{\lambda_k,w_k,\ell_k}\|_{L^2((0,T)\times\D)}+c_k\|S_{\lambda_k,w_k,\ell_k}-S_{\lambda_k,w_k,z_k,\alpha_k}\|_{L^2((0,T)\times\D)}.
\end{align*}
Lemma~\ref{lem:gl-finite} controls the first term, and the uniform profile-tail estimate controls the second. Since $c_k$ is bounded, we obtain
\begin{equation*}
\|\partial_tu_k-c_kS_{\lambda_k,w_k,z_k,\alpha_k}\|_{L^2((0,T)\times\D)}\rightarrow0.
\end{equation*}
The nonvanishing bound follows from the first-mode lower bound in Section~\ref{subsubsec:curved-energy}, together with the same profile-tail estimate. We omit the repeated energy and tail calculations.
\vspace{1.5mm}

Proposition~\ref{prop:gl-support}(ii) gives the matching assertion. For disjoint intervals, we apply Proposition~\ref{prop:gl-observation} after the scales have been chosen. Multiplying \eqref{eq:gl-disc-real} by $\1_{G^{(N)}}$ or $\1_G$ then gives the observation identities.

For the additional observed energy, apply the period-counting estimate in Section~\ref{subsubsec:curved-diagonal} on each interval in $\R$. At a fixed spatial point, only $O_T(1)$ such intervals meet $[0,T]$; on each of them the phases have time derivative $-(2\pi h_k)^{-1}$. Summing their contributions and integrating the normal amplitude as in \eqref{eq:extra-energy}, we obtain
\begin{equation*}
\|\1_{\calE_k}S_k\|_{L^2((0,T)\times\D)}^2 \le C_T(\gamma_k+h_k/w_k)\rightarrow0.
\end{equation*}
Here $h_k\sim\lambda_k^{9/2}$ and $\lambda_k^2/w_k\to0$. Combining this estimate with the realization error proves non-observability on $G$ for every fixed $T>0$.
\end{proof}

\section{AOS on smooth bounded domains}\label{sec:curved}

The glancing disc construction of Section~\ref{sec:glancing-disc} has given a concrete family of the form $A\left(f_*(\phi)-g_*(\psi)\right)$ and an observation rule on disjoint moving packet regions. Its two profiles satisfy the same boundary--caustic compatibility at every scale.

We now keep these profiles and ask how to construct the propagation phases when the \emph{curvature varies}. The new difficulty is geometric: the normal depth and tangential action can no longer be chosen independently of the boundary point. The diamond profile, the symmetry, and the rule defining $G^{(N)}$ and $G$ do not change.

Throughout this section, $\Omega$ is a smooth bounded planar domain. We work near one positively curved boundary arc for a fixed time $0<T<T_{\rm loc}$, with the geometry chosen in Section~\ref{subsec:curved-local}. Theorem~\ref{thm:realization} realizes the same diamond family by normalized solutions of \eqref{eq:wave} on the whole domain $\Omega$ and proves non-observability on the common set $G$.
\vspace{1.5mm}

The section is organized as follows.
\begin{itemize}[leftmargin=*,itemsep=3pt]
\item Section~\ref{subsec:curved-local} fixes the local geometry and a short observation time near a positively curved boundary arc.
\item Section~\ref{subsec:curved-phases} constructs the phases with a common boundary action and coordinates the scales.
\item Section~\ref{subsec:curved-boundary} imposes the Dirichlet condition and recovers a common amplitude. Propositions~\ref{prop:curved-support} and~\ref{prop:curved-observation} define $S_k$ and $G^{(N)},G$.
\item Proposition~\ref{prop:finite-airy} estimates the residual of $v_\ell$ and corrects it to an exact wave $u_\ell$.
\item Section~\ref{subsec:curved-full} controls the profile tails, normalizes the energy and estimates the additional observed energy, completing Theorem~\ref{thm:realization}.
\end{itemize}

\setcounter{subsection}{-1}
\subsection{From the disc to a general domain: the intuition}\label{subsec:curved-intuition}

We follow the construction of Section~\ref{subsec:gl-intuition} in seven steps. Stars mark the principal changes from the disc: the local geometry, the common boundary action, and the corrected phases and amplitudes. The remaining steps use these data in the same localization, matching, and approximation arguments.

\vspace{2mm} \noindent{\it $1^*$. A local positively curved boundary arc.} Notice that every smooth bounded planar domain has a boundary point of positive curvature with respect to the inward unit normal\footnote{Later we may just use positive curvature for short.}. We choose a small arc around it and a fixed short time so that all the moving packet regions remain near that arc. The rest of the boundary may be nonconvex. Outside the moving packet regions, the entire domain is observed. Section~\ref{subsec:curved-local} gives the local coordinates and the separation condition needed later for test rays.

Write $s$ for local boundary arclength, $y$ for inward distance and $\kappa(s)>0$ for curvature on the chosen arc.

\vspace{2mm} \noindent{\it $2^*$. A common boundary action and the same compatible profiles.} Keep the diamond profiles and their reference matching rectangles from Section~\ref{subsec:diamond}:
\begin{equation*}
g_*=\calH f_*,\quad f_*=\tau_{1/4}g_*,\text{ \; and \; } \mathcal R=(F_+\times J_+)\cup(F_-\times J_-).
\end{equation*}
To use the same boundary translation at every point, we require the phase difference at $y=0$ to remain a quarter-period. As in the disc, this is achieved by keeping the boundary action equal to $\lambda^3/3$ and coordinating the semiclassical parameter with it. Freezing the curvature at one point, a small normal momentum $\beta$ gives the leading action
\begin{equation*}
\int_0^{\beta^2/(2\kappa)}\sqrt{\beta^2-2\kappa y}\dd y =\frac{\beta^3}{3\kappa}.
\end{equation*}
To keep this equal to $\lambda^3/3$, choose $\beta=\lambda\kappa^{1/3}$. This predicts the depth $\lambda^2/(2\kappa^{1/3})$: larger curvature makes the moving packet region thinner.
\vspace{1.5mm}

The corrected normal action $R_\lambda$ will preserve the boundary value exactly, while its turning depth has this leading behavior:
\begin{equation*}
R_\lambda(s,0)=\frac{\lambda^3}{3},\text{ \; and \; } y_c^\lambda(s)=\frac{\lambda^2}{2\kappa(s)^{1/3}}+O(\lambda^4).
\end{equation*}
The common boundary action allows the same quarter-period relation along the arc. The physical matching cells can therefore deform without changing the reference symmetry; see Figure~\ref{fig:curved-intuition}. We construct finite-order phases, instead of considering an exact invariant billiard caustic.

\begin{figure}[htbp]
\centering
\includegraphics[width=\textwidth]{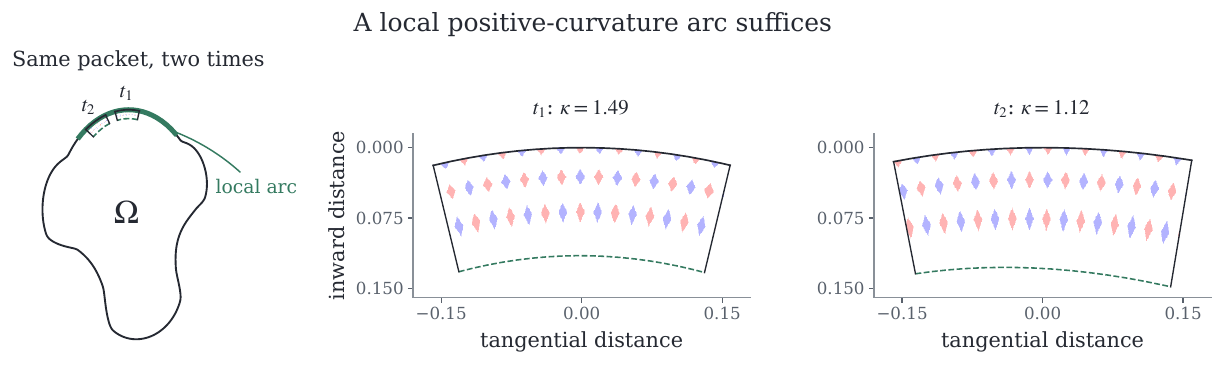}
\caption{Local construction on a nonconvex domain.}\label{fig:curved-intuition}
\end{figure}

\vspace{2mm} \noindent{\it $3^*$. Corrected phases, amplitudes, and generated approximate velocities.} Recover the two-branch travelling wave form as in the disc case. Varying curvature introduces a normal correction to the common tangential phase:
\begin{equation*}
\Phi_{0,\lambda}(s,y)=C_\lambda(s)+U(s,y), \quad U(s,y)=-\frac{\kappa'(s)}{6\kappa(s)}y^2.
\end{equation*}
This correction and the corrected normal action $R_\lambda$ control the phase errors in Section~\ref{subsec:curved-phases}. The common amplitude also acquires the factor $\kappa(s)^{1/6}$ to cancel a leading transport term, as explained in Section~\ref{subsec:curved-boundary}. The profile arguments remain
\begin{align*}
\phi_k=\frac{\Phi_{0,\lambda_k}-t+R_{\lambda_k}}{2\pi h_k}+\alpha_k,\;\psi_k=\frac{\Phi_{0,\lambda_k}-t-R_{\lambda_k}}{2\pi h_k}+\alpha_k.
\end{align*}
They generate the velocity $S_k=A_{\lambda_k}\left(f_*(\phi_k)-g_*(\psi_k)\right)$ with the common amplitude specified in the next step.

\Needspace{8\baselineskip} \vspace{2mm} \noindent{\it 4. Localization in a curved moving packet region.} As in Section~\ref{subsec:gl-intuition}, the Airy expression concentrates in the normal layer, and one moving envelope localizes both branches in $s-t$. For $I_k=z_k+(-w_k/2,w_k/2)$, take
\begin{gather*}
\mathcal U_k=\{s-t\in I_k,\quad 0<y<y_c^{\lambda_k}(s)\},\\
A_{\lambda_k}=c_k(\lambda_kw_k)^{-1/2} \chi\!\left(\frac{s-t-z_k}{w_k}\right)\kappa(s)^{1/6}(\zeta_{\lambda_k}(s,y))^{-1/4} \quad\text{on }\mathcal U_k.
\end{gather*}
Here $\zeta_\lambda$ is the smooth Airy coordinate: $R_\lambda=\frac23(\zeta_\lambda)^{3/2}$ on the oscillatory side, and $\zeta_\lambda=0$ defines the turning curve. The envelope $\chi\in C_c^\infty((-1/2,1/2))$ is real and nonzero, and $c_k$ is the energy-normalizing factor chosen in Section~\ref{subsec:curved-full}. Extending the amplitude by zero outside the moving packet region, including its turning curve, gives
\begin{equation*}
S_k=\1_{\mathcal U_k}S_k.
\end{equation*}
The tangential width is $w_k$ and the center is $s=t+z_k$. The normal depth changes with curvature as the moving packet region moves. The comparison with the smooth Airy velocity across the turning curve is proved in Proposition~\ref{prop:finite-airy}, following Lemma~\ref{lem:gl-finite}.

\vspace{2mm} \noindent{\it 5. Symmetry matching inside the moving packet region.} The full profiles and reference rectangles are unchanged. Pulling them back by the corrected phases gives
\begin{align*}
\Gamma_k&=\mathcal U_k\cap\{(\phi_k,\psi_k)\pmod1\in\mathcal R\},\\*
\1_{\Gamma_k}S_k&=\1_{\Gamma_k}A_{\lambda_k}(f_*(\phi_k)-g_*(\psi_k))=0.
\end{align*}
Thus the same reference symmetry gives matching regions whose physical shapes depend on curvature, as in Proposition~\ref{prop:curved-support}.

\vspace{2mm} \noindent{\it 6. Coordinated scales and one fixed observation region.} The semiclassical parameter $h_k$, depth of order $\lambda_k^2$, and tangential width $w_k$ retain their different roles from Section~\ref{subsec:gl-family}. The phase error, the discrepancy between Fourier modes, and the moving-envelope error require, respectively,
\begin{equation*}
\frac{\lambda_k^5}{h_k}\to0, \quad \frac{h_k}{\lambda_k^4}\to0 ,\text{ \; and \; }\frac{\lambda_k^2}{w_k}\to0.
\end{equation*}
The choice $h_k\sim\lambda_k^{9/2}$ already used in \eqref{eq:gl-disc-scales} leaves room for both of the first two requirements. An integer adjustment also enforces the boundary shift exactly. Section~\ref{subsubsec:curved-scales} explains this coordination; the Fourier cutoff is fixed before taking $\lambda_k$ sufficiently small.

Choose the intervals $I_k$ pairwise disjoint, so that the moving packet regions remain disjoint despite their varying depths. Using these phases, define $\calE_k$ by \eqref{eq:gl-extra} and put $\mathcal H_k=\mathcal U_k\setminus(\Gamma_k\cup\calE_k)$. As in Section~\ref{subsec:gl-observation}, set
\begin{gather*}
G^{(N)}=\left((0,T)\times\Omega\right)\setminus\bigcup_{k\le N}\mathcal H_k, \quad G^{(N)}\downarrow G,\\
G^{(N)}\cap\mathcal U_k=G\cap\mathcal U_k=\Gamma_k\cup\calE_k \quad\forall k\in\{1,\ldots,N\}.
\end{gather*}
In addition to the internal phase drift seen in the disc, each moving packet region changes depth as it travels through different curvatures; see Figure~\ref{fig:curved-intro-stages}.

\vspace{2mm} \noindent{\it 7. Velocity approximation and non-observability.} We now prove the general-domain counterpart of the disc realization stated in Proposition~\ref{prop:gl-disc-real}. We first construct an approximate wave that uses truncated profiles and satisfies the Dirichlet condition exactly. We then correct its residual and control the Fourier tails.

After energy normalization, Theorem~\ref{thm:realization} gives
\begin{equation*}
E[u_k](0)=1,\quad \partial_tu_k=S_k+r_k ,\text{ \; and \; }\|r_k\|_{L^2((0,T)\times\Omega)}\rightarrow0.
\end{equation*}
Support and matching then reduce the observation to two terms. In the following calculation all norms are on $(0,T)\times\Omega$:
\begin{align*}
\|\1_G\partial_tu_k\|_2\le\|\1_{G^{(k)}}\partial_tu_k\|_2\le\|\1_{G^{(k)}}r_k\|_2+\|\1_{\calE_k}S_k\|_2\le\|r_k\|_2+\|\1_{\calE_k}S_k\|_2\rightarrow0.
\end{align*}
Theorem~\ref{thm:realization} controls the additional observed energy despite concentration near the turning curve. This proves non-observability for any reference widths $\gamma_k\to0$. Section~\ref{sec:gcc} makes further parameter choices to obtain \hyperref[eq:gcc]{GCC} on the same fixed $G$.

\subsection{Geometric preliminaries}\label{subsec:curved-local}
We fix the local geometry for two purposes. The moving envelopes and normal cutoffs of the approximate waves must stay inside one boundary chart. For the ray comparison in Section~\ref{sec:gcc}, every test ray meeting the neighborhood containing the moving packet regions must stay near the same boundary arc throughout the observation interval. We therefore choose neighborhoods $V\Subset U$ and a fixed short time: $V$ contains the construction, and the distance from $V$ to $U^c$ keeps the test rays inside $U$. 

\begin{lemma}\label{lem:local-geometry}
Let $\Omega\subset\R^2$ be a smooth bounded domain. There are a boundary point $p$, bounded open neighborhoods $V\Subset U$ of $p$, a boundary arclength chart $s\in\mathcal I$, and constants $d_0,\kappa_*,T_{\rm loc}>0$ with the following properties.
\begin{enumerate}[label=(\roman*),leftmargin=*,itemsep=1mm]
\item  The boundary in $U$ belongs to this chart and has curvature $\kappa\ge\kappa_*$. The set $U\cap\overline\Omega$ lies in a regular collar chart $x(s,y)=x(s,0)+y\mathbf n(s)$, defined also on a neighborhood of the closed strip below.
\item  There is a bounded open interval $\mathcal I_{\rm in}$ such that
\begin{equation}\label{eq:local-buffer}
\mathcal I_{\rm in}+[0,T_{\rm loc}]\Subset\mathcal I, x\left(\left(\overline{\mathcal I_{\rm in}}+[0,T_{\rm loc}]\right)\times[0,d_0]\right)\subset V, 4T_{\rm loc}<\dist(\overline V,U^c).
\end{equation}
\item There is a smooth bounded strictly convex domain $\Omega_{\rm c}$ with strictly positive curvature, which can be chosen so that $\Omega_{\rm c}\cap U=\Omega\cap U$.
\end{enumerate}
\end{lemma}
\begin{remark}
The inclusions keep moving packet regions with $I_k\subset\mathcal I_{\rm in}$ and sufficiently small normal depths in $V$ for every $0\le t\le T<T_{\rm loc}$. Moreover, the distance bound keeps test rays meeting $V$ inside $U$. For every $0<T<T_{\rm loc}$, any regular unit-speed ray that meets $V$ during $[0,T]$ remains in $U$ throughout that interval and agrees there with the ray in $\Omega_{\rm c}$ having the same initial state. 
\end{remark}
\begin{proof}[Proof of Lemma \ref{lem:local-geometry}] We prove this lemma as follows.\\
\noindent\emph{A positive-curvature arc.} Fix $c\in\Omega$ and let $R=\max_{\overline\Omega}|x-c|$. The enclosing circle touches $\partial\Omega$ at a maximizing point $p$, and the two inward normals agree there; see Figure~\ref{fig:local-geometry}(a). For a unit-speed boundary parametrization $\gamma$ with $\gamma(0)=p$, we have
\begin{equation*}
\overline\Omega\subset\overline{B(c,R)},\quad p-c=-R\mathbf n(p),\text{ \; and \; }\gamma''(0)=\kappa(p)\mathbf n(p).
\end{equation*}
Since $|\gamma(s)-c|^2$ has a maximum at $s=0$, it follows that
\begin{equation*}
0\ge\left.\frac{d^2}{ds^2}|\gamma(s)-c|^2\right|_{s=0} =2-2R\kappa(p).
\end{equation*}
Hence $\kappa(p)\ge R^{-1}>0$. Continuity gives a positive-curvature arc and regular collar coordinates around it.

\vspace{2mm} \noindent\emph{A convex completion.} To compare the complete short ray segments used in Section~\ref{sec:gcc}, we extend the chosen arc to a closed strictly convex boundary without changing it near $p$. For this step we use the osculating circle at $p$.

Write the boundary arc and its osculating circle as graphs $\rho(\xi)$ and $\rho_{\rm c}(\xi)$ over their common tangent. Their derivatives agree through order two at $\xi=0$, so
\begin{equation*}
\partial_\xi^j(\rho-\rho_{\rm c})(\xi)=O(|\xi|^{3-j}), \quad \forall j\in\{0,1,2\}.
\end{equation*}
Choose $\eta\in C_c^\infty((-1,1))$ equal to one near zero. Replace the circular graph by
\begin{equation*}
\rho_\delta(\xi):=\rho_{\rm c}(\xi)+\eta(\xi/\delta)\left(\rho(\xi)-\rho_{\rm c}(\xi)\right).
\end{equation*}
In particular, the second derivative of the perturbation satisfies
\begin{gather*}
\partial_\xi^2(\rho_\delta-\rho_{\rm c}) =\delta^{-2}\eta''(\xi/\delta)(\rho-\rho_{\rm c}) +2\delta^{-1}\eta'(\xi/\delta)(\rho'-\rho_{\rm c}')+\eta(\xi/\delta)(\rho''-\rho_{\rm c}''),\\
\|\partial_\xi^2(\rho_\delta-\rho_{\rm c})\|_{L^{\infty}} \le C(\delta^{-2}\delta^3+\delta^{-1}\delta^2+\delta)\le C\delta.
\end{gather*}
Together with the lower derivatives, this gives
\begin{equation*}
\|\partial_\xi^j(\rho_\delta-\rho_{\rm c})\|_{L^{\infty}}\le C\delta^{3-j}, \quad \forall j\in\{0,1,2\}.
\end{equation*}
Thus the change is $O(\delta)$ in $C^2$. For small $\delta$ the resulting closed curve remains smooth, embedded and positively curved. It bounds a strictly convex domain $\Omega_{\rm c}$ agreeing with $\Omega$ in a neighborhood $U$ of $p$.

\vspace{2mm} \noindent\emph{The distance bound and ray comparison.} We choose $V\Subset U$, then $\mathcal I_{\rm in}$, $d_0$ and $T_{\rm loc}$ small enough for \eqref{eq:local-buffer}. The first two inclusions control the moving packet regions. Indeed, their tangential coordinates satisfy
\begin{equation*}
s-t\in I_k\subset\mathcal I_{\rm in} \quad\rightarrow\quad s\in I_k+t\subset\mathcal I_{\rm in}+[0,T_{\rm loc}], \quad \forall t\in[0,T].
\end{equation*}
Taking their normal depths below $d_0$ therefore keeps them in $V$. This support condition comes from the chosen envelope.

The distance bound controls test rays, independently of that envelope. If a regular unit-speed ray satisfies $x(t_*)\in V$, then
\begin{equation*}
|x(t)-x(t_*)|\le |t-t_*|\le T<\dist(\overline V,U^c), \quad \forall t\in[0,T].
\end{equation*}
Hence the ray remains in $U$ throughout $[0,T]$, regardless of its number of reflections. The boundaries and reflection laws of $\Omega$ and $\Omega_{\rm c}$ agree there, so the two rays with the same initial state coincide; see Figure~\ref{fig:local-geometry}(b).

The factor $4$ in \eqref{eq:local-buffer} also leaves room for extensions of length $T_{\rm loc}$ before time zero and after time $T$. On $[-T_{\rm loc},T+T_{\rm loc}]$, the distance traveled from $x(t_*)$ is at most $T+T_{\rm loc}<2T_{\rm loc}$. For a sufficiently small glancing threshold, these extensions contain the complete ray segments containing time zero or time $T$ and the neighboring short ray segments used in Section~\ref{sec:gcc}.
\end{proof}

\begin{figure}[htbp]
\centering
\includegraphics[width=\textwidth]{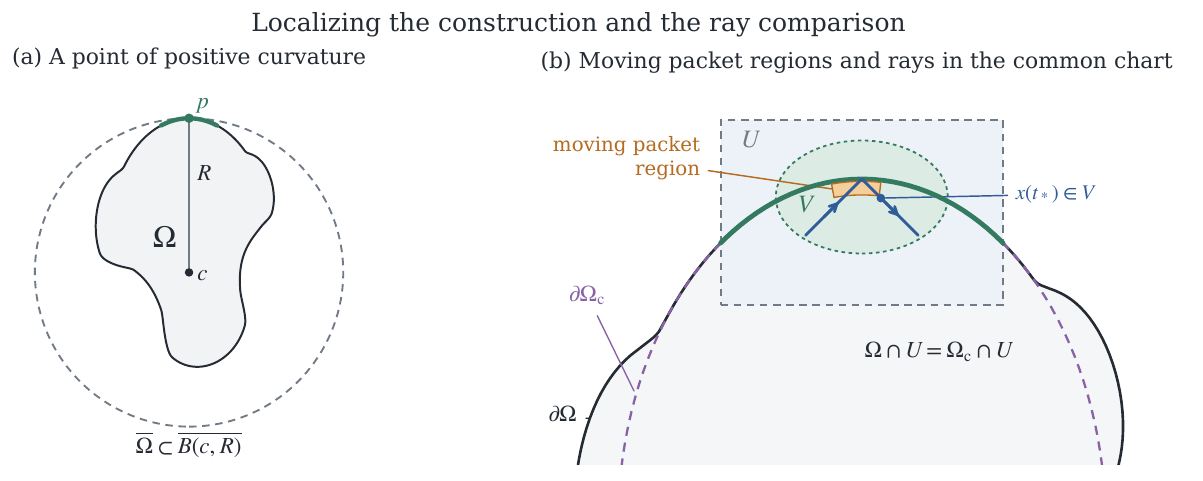}
\caption{Local geometry for the construction and ray comparison.}\label{fig:local-geometry}
\end{figure}

Fix this geometry and $0<T<T_{\rm loc}$ for the rest of Sections~\ref{sec:curved}--\ref{sec:gcc}. Choose $I_k\subset\mathcal I_{\rm in}$ before the normal scales. The normal cutoff of the approximate waves will also be supported in $0\le y<d_0$; taking the later $\lambda_k$ sufficiently small ensures this.

\Needspace{8\baselineskip}
\subsection{From the disc phases to variable curvature}\label{subsec:curved-phases}

The reference profiles $f_*,g_*$ are fixed. We construct phases with a common boundary action, controlled phase errors, and a turning curve that determines the normal depth of the moving packet region. Recall the coordinates in Lemma~\ref{lem:local-geometry}. Their metric is $\dd y^2+\mu(s,y)^2\dd s^2$, where $\mu(s,y):=1-\kappa(s)y$. We first state the resulting phase properties and define the moving packet region.

\begin{prop}\label{prop:curved-phases}
Fix the local geometry of Lemma~\ref{lem:local-geometry}. For all sufficiently small $\lambda>0$, the functions $\Phi_{0,\lambda}$ and $\zeta_\lambda$ constructed in Section~\ref{subsubsec:curved-phase-correction}, with the explicit formulas \eqref{eq:W0} and \eqref{eq:airy-coeff}, have the following properties.
\begin{enumerate}[label=(\roman*),leftmargin=*,itemsep=1mm]
\item \emph{Common boundary action.} On the oscillatory side, the normal action and its boundary value satisfy
\begin{equation}\label{eq:phase-data}
\zeta_\lambda(s,0)=2^{-\frac{2}{3}}\lambda^2, R_\lambda:=\frac{2}{3}(\zeta_\lambda)^{\frac{3}{2}},\text{ \; and \; }R_\lambda(s,0)=\frac{1}{3}\lambda^3.
\end{equation}
The two spatial phases are $\Phi_\lambda^\pm:=\Phi_{0,\lambda}\pm R_\lambda$.

\item \emph{Phase errors.} We have\footnote{Note that $|\dd f|_{g^{-1}}^2=|\nabla_g f|_g^2$. For simplicity, we abuse the notation $|\dd f|_{g}^2$.}
\begin{equation}\label{eq:eik}
\begin{aligned}
E_0&:=|\dd\Phi_{0,\lambda}|_g^2+\zeta_\lambda|\dd\zeta_\lambda|_g^2-1=O(\lambda^6),\\
E_1&:=\langle \dd\Phi_{0,\lambda},\dd\zeta_\lambda\rangle_g=O(\lambda^4).
\end{aligned}
\end{equation}
Norms and inner products are with respect to the metric $\dd y^2+\mu(s,y)^2\dd s^2$. On each fixed scaled collar $0\le y\le c\lambda^2$, over the compact part of the chart fixed in Lemma~\ref{lem:local-geometry}, the error estimates hold uniformly. %
In particular, on the oscillatory collar the two spatial phases satisfy
\begin{equation*}
|\dd\Phi_\lambda^\pm|^2=1+O(\lambda^5).
\end{equation*}

\item \emph{Turning curve.} The first turning point satisfies
\begin{equation}\label{eq:turning}
y_c^\lambda(s)=\frac{\lambda^2}{2\kappa(s)^{1/3}}+O(\lambda^4) ,\text{ \; and \; }\partial_y\zeta_\lambda=-(2\kappa)^{1/3}+O(\lambda^2).
\end{equation}
We use the real $R_\lambda$ only for $0<y<y_c^\lambda(s)$. The smooth Airy coordinate itself continues through the zero. For $\kappa=1$, these formulas are exactly \eqref{eq:gl-disc-phase}.
\end{enumerate}
\end{prop}
We defer the proof to Section~\ref{subsubsec:curved-phase-correction}.
\begin{defi}[Moving packet region]\label{def:curved-packet}
Fix the collar coordinates $x(s,y)$ and data $\mathcal I_{\rm in},d_0,T_{\rm loc}$ of Lemma~\ref{lem:local-geometry}, and take $0<T<T_{\rm loc}$. Choose $0<w<1$, a center $z\in\R$ with $z+[-w/2,w/2]\subset\mathcal I_{\rm in}$, and $\lambda>0$ sufficiently small that the turning depth in \eqref{eq:turning} is below $d_0$. Define
\begin{equation}\label{eq:curved-packet}
\mathcal U_{\lambda,w,z}:=\left\{\left(t,x(s,y)\right):
\begin{aligned}
&t\in(0,T),\quad s\in\mathcal I,\\
&|s-t-z|<w/2,\quad 0<y<y_c^\lambda(s)
\end{aligned}
\right\}.
\end{equation}
At time $t$, the tangential interval has width $w$ and center $s=t+z$. Its normal depth is
\begin{equation*}
y_c^\lambda(s)=\frac{\lambda^2}{2\kappa(s)^{1/3}}+O(\lambda^4).
\end{equation*}
The inclusions in \eqref{eq:local-buffer} give $\mathcal U_{\lambda,w,z}\subset(0,T)\times V$.
\end{defi}

We now derive these phase properties from the boundary-coordinate equation and the leading normal geometry. Section~\ref{subsubsec:curved-scales} then chooses the semiclassical parameter to control both the phase residual and the discrepancies between Fourier modes, while imposing the quarter-period boundary relation. The Airy construction and velocity approximation follow in Sections~\ref{subsec:curved-boundary}--\ref{subsec:curved-full}.

\subsubsection{The wave equation in boundary normal coordinates}\label{subsubsec:boundary-geometry}

In the disc, the coordinates $(\theta,y)$ with $y=1-r$ made the boundary $y=0$. We keep this description on the chosen boundary arc and introduce boundary normal coordinates $(s,y)$, where $s\in\mathcal I$ is boundary arclength, and $y$ is the inward distance. Let $\boldsymbol\tau(s)$ be the unit tangent, and $\mathbf n(s)$ be the inward unit normal.  The associated coordinate map is
\begin{equation*}
x(s,y)=x(s,0)+y\mathbf n(s).
\end{equation*}
The Frenet relation $\mathbf n'=-\kappa\boldsymbol\tau$ gives
\begin{equation*}
\partial_sx(s,y)=\left(1-\kappa(s)y\right)\boldsymbol\tau(s) ,\text{ \; and \; }\partial_yx(s,y)=\mathbf n(s).
\end{equation*}
The metric and area element are given by
\begin{equation*}
g=\dd y^2+\mu(s,y)^2\dd s^2 ,\text{ \; and \; }\dd x=\mu(s,y)\dd s\dd y.
\end{equation*}
For a sufficiently small fixed collar neighborhood, $\mu>0$ and these coordinates are regular. The corresponding wave equation is
\begin{equation}\label{eq:collar-wave}
\partial_t^2u-\frac{1}\mu\partial_y(\mu\partial_yu) -\frac{1}\mu\partial_s(\mu^{-1}\partial_su)=0, \quad\text{with }u(t,s,0)=0.
\end{equation}

To construct the travelling wave solutions, away from the turning point, we make the ansatz $u_h(t,s,y)=A_h(s,y)e^{\frac{i}{h}(\Phi(s,y)-t)}$. Then, by WKB expansion, the leading eikonal equation is
\begin{equation*}
|d\Phi|_{g}^2=(\partial_y\Phi)^2+\mu^{-2}(\partial_s\Phi)^2=1.
\end{equation*}
This relation determines the normal momentum from the tangential momentum. Unlike in the disc, the coefficients depend on $s$. A single angular Fourier mode no longer separates the equation. We therefore construct local oscillations whose phases and amplitudes vary along the boundary.

To make the context more compact, we use the following convention. For any smooth function $f$ defined in the collar neighborhood,
\begin{gather*}
\nabla_gf:=\partial_yf\partial_y+\mu^{-2}\partial_sf\partial_s, \text{ \; and \; }|\nabla_gf|^2_g=\langle\nabla_gf,\nabla_gf\rangle_g:=|\partial_yf|^2+\mu^{-2}|\partial_sf|^2,\\
-\Delta_gf:=-\frac{1}{\mu}\partial_y(\mu\partial_yf) -\frac{1}{\mu}\partial_s(\mu^{-1}\partial_sf), \text{ where }\mu(s,y)=1-\kappa(s)y>0.
\end{gather*}
Using the Airy coordinate $\zeta$, we set the ansatz $W_h(s,y)=\left( \frac{|\nabla_g\zeta|_g^2} {1-\zeta|\nabla_g\zeta|_g^2} \right)^{1/4}e^{\frac{i}{h}\Phi(s,y)}\Ai\left(-\frac{\zeta(s,y)}{h^{\frac{2}{3}}}\right)$, where $\zeta$ and $\Phi$ are defined via
\begin{equation}\label{eq:phase-eq-s}
|d\Phi|_g^2+\zeta|d\zeta|_g^2=1,\;\langle d\Phi,d\zeta\rangle_g=0.
\end{equation}
Then, we have the following lemma.
\begin{lemma}\label{lem:exact-Airy-solution}
Let $U$ be a coordinate neighborhood equipped with the metric $g=\mu^2\,ds^2+dy^2$. Assume also that $\zeta(s,y_c(s))=0$ and that $0<c\leq -\partial_y\zeta(s,y)\leq C$. Then, locally near the turning curve,
\begin{equation*}
|\partial_y\zeta|\sim 1,\;|\zeta(s,y)|\sim |y-y_c(s)|.
\end{equation*}
For $0<h\leq 1$, $W_h$ satisfies the following equation
\[
(-h^2\Delta_g-1)W_h =e^{i\Phi/h}\left(h^2r_0 \Ai\left(-\frac{\zeta}{h^{2/3}}\right)+h^{4/3}r_1\Ai'\left(-\frac{\zeta}{h^{2/3}}\right) \right),
\]
where $r_0$ and $r_1$ are defined by
\begin{gather*}
a=\left( \frac{|\nabla_g\zeta|_g^2} {1-\zeta|\nabla_g\zeta|_g^2} \right)^{1/4},\;r_0=-\Delta_ga,\;\text{ \; and \; }r_1=2\langle d\zeta,da\rangle_g+(\Delta_g\zeta)a.
\end{gather*}
Moreover, for every compact set $K\Subset U$, there exists a constant $C_K>0$, independent of $h$, such that
\[
\|r_0\|_{L^\infty(K)} +\|r_1\|_{L^\infty(K)} \leq C_K.
\]
\end{lemma}
\begin{proof}
Direct computation will give the proof. In fact, 
\begin{align*}
e^{-i\Phi/h}(-h^2\Delta_g-1)W_h ={}&a\left( |d\Phi|_g^2+\zeta|d\zeta|_g^2-1 \right)\Ai\left(-\frac{\zeta}{h^{2/3}}\right)\\
&+2ih^{1/3}a\langle d\Phi,d\zeta\rangle_g\Ai'\left(-\frac{\zeta}{h^{2/3}}\right)\\
&+\left(-ih\widetilde{r}_1+h^2r_0\right)\Ai\left(-\frac{\zeta}{h^{2/3}}\right) +h^{4/3}r_1\Ai'\left(-\frac{\zeta}{h^{2/3}}\right).
\end{align*}
Here $\widetilde{r}_1= 2\langle d\Phi,da\rangle_g+(\Delta_g\Phi)a$. The first two terms vanish by the phase equations \eqref{eq:phase-eq-s}. The coefficient bounds follow from their formulas, and the comparison $|\zeta(s,y)|\sim |y-y_c(s)|$ follows from the bounds on $\partial_y\zeta$. It remains to check that $\widetilde{r}_1\equiv0$. Indeed, the equation $\widetilde{r}_1=0$ is equivalent to $-a\widetilde{r}_1=0$, that is
\begin{equation*}
0=-\langle d\Phi,da^2\rangle_g-(\Delta_g\Phi)a^2=-\mathrm{div}_g(a^2\nabla_g\Phi).
\end{equation*}
On the other hand, for any smooth function $f$ defined in the local chart, we have $\nabla_gf=\partial_yf\partial_y+\mu^{-2}\partial_sf\partial_s$. Define the linear rotation operator $\mathcal{R}$ by $\mathcal{R}\partial_s=\mu\partial_y$ and $\mathcal{R}\partial_y=-\mu^{-1}\partial_s$. Then
\begin{gather*}
\mathcal{R}\nabla_gf=\mu^{-1}(-\partial_yf\partial_s+\partial_sf\partial_y),\\
\mathrm{div}_g(\mathcal{R}\nabla_gf)=\mu^{-1}\left(\partial_y(\mu\mu^{-1}\partial_sf)+\partial_s(-\mu\mu^{-1}\partial_yf)\right)\equiv0.
\end{gather*}
 By the second equation in \eqref{eq:phase-eq-s}, we know that $\nabla_g\Phi$ is orthogonal to $\nabla_g\zeta$. Since we work in dimension $2$, the orthogonal direction is uniquely determined (up to orientation). Hence, we deduce that
\begin{equation*}
\nabla_g\Phi=c\mathcal{R}\nabla_g\zeta.
\end{equation*}
Comparing the norms, we know that $c=\pm\frac{|\nabla_g\Phi|_g}{|\nabla_g\zeta|_g}$. So we obtain
\begin{equation*}
a^2\nabla_g\Phi=\pm a^2\frac{|\nabla_g\Phi|_g}{|\nabla_g\zeta|_g}\mathcal{R}\nabla_g\zeta.
\end{equation*}
Therefore, a natural choice of $a$ is such that $a^2\frac{|\nabla_g\Phi|_g}{|\nabla_g\zeta|_g}=\mathrm{constant}$. Applying the first equation in \eqref{eq:phase-eq-s}, we know that $a=\left( \frac{|\nabla_g\zeta|_g^2} {1-\zeta|\nabla_g\zeta|_g^2} \right)^{1/4}$ makes $\widetilde{r}_1\equiv0$.
\end{proof}
\begin{remark}
In the disc case, the second equation in \eqref{eq:phase-eq-s} is automatically satisfied,  since $\Phi(\theta,r)=b\theta$ depends only on the angular variable, while $\zeta=\zeta(r)$ depends only on the radial variable. However, in the general case, we cannot expect a simple clean formula like in the disc case. Moreover, in Lemma~\ref{lem:W-form}, we choose $a_b(r):=\left(r|\zeta'(r)|\right)^{-\frac{1}{2}}$ to make $r_1\equiv0$. 
\end{remark}
Since we may not get a simple clean form from the equations in \eqref{eq:phase-eq-s}, we perform an asymptotic analysis of the phases and solutions. As we showed in Section~\ref{sec:glancing-disc}, Taylor expansion can help construct approximate travelling wave solutions. Following this idea, we expand our phases near the chosen boundary arc $y=0$. 

Let us first fix a base point on the boundary, and take the frozen metric in the neighborhood, that is, for a certain $s_0$ and its associated curvature $\kappa_0=\kappa(s_0)>0$, we freeze the metric $g_0=(1-\kappa_0y)^2\dd s^2+\dd y^2$ and its associated Laplacian 
\begin{equation}\label{eq:freezing-lap-op}
\Delta_{g_0}=\partial_y^2-\frac{\kappa_0}{1-\kappa_0y}\partial_y +\frac{1}{(1-\kappa_0y)^2}\partial_s^2.
\end{equation}
Locally, this is just a disc $\D(0,1/\kappa_0)$. Therefore, heuristically, we expect we could benefit from the analysis of the disc case. Let us first do some \textit{formal} computation.

\vspace{2mm} For the frozen Laplacian $\Delta_{g_0}$, as in Section~\ref{sec:glancing-disc}, the phase equations \eqref{eq:phase-eq-s} reduce to
\begin{equation*}
\Phi(s,y)=\sqrt{1-\beta^2}s,\;\zeta(y)(\zeta'(y))^2=1-\frac{1-\beta^2}{(1-\kappa_0 y)^2}.
\end{equation*}
For a unit covector $p_s\dd s+p_y\dd y$ with inward normal component $\beta\in(0,1)$ at the boundary and positive tangential component, the conserved tangential momentum in this frozen model is $p_s=\sqrt{1-\beta^2}$, which gives the phase $\Phi$. Let $R_{\beta}(y):=\frac{2}{3}|\zeta(y)|^{3/2}$. Then, the phases are $p_ss\pm R_{\beta}(y)$. Now we do expansions of the normal momentum $p_y$:
\begin{align*}
1-\frac{1-\beta^2}{(1-\kappa_0 y)^2}=\beta^2-2\kappa_0 y+O(\beta^2y+y^2).
\end{align*}

Thus, the leading part of the Airy coordinate $\zeta_m$ is given by
\begin{equation*}
\zeta_m(y)(\zeta_m'(y))^2=\beta^2-2\kappa_0 y,
\end{equation*}
which implies that $\zeta_m(y):=(2\kappa_0)^{1/3}(\frac{\beta^2}{2\kappa_0}-y)$ with the turning point $y_c=\frac{\beta^2}{2\kappa_0}$ and $\zeta_m(0)=\frac{\beta^2}{(2\kappa_0)^{2/3}}$. For this leading model, set $R_{\beta,m}(y):=\frac{2}{3}(\zeta_m(y))^{3/2}$ for $0\le y\le y_c$. Then, we have
\begin{equation*}
R_{\beta,m}(0)=\frac{\beta^3}{3\kappa_0}.
\end{equation*}
Using the small parameter $\lambda$ from Section~\ref{sec:glancing-disc}, we impose $\lambda^3=3R_{\beta,m}(0)$ in this leading model. We aim to keep this boundary-action normalization along the chosen arc. 

\vspace{2mm} As in Lemma~\ref{lem:local-geometry}, for a certain boundary point $p$ with $\kappa_0>0$, fix its neighborhood $U$, such that along the arclength chart $s\in\mathcal{I}$, $\kappa(s)\geq\kappa_*>0$. Given a small parameter $\lambda$, we can define functions $\beta$, $y_c$ and $p_s$ (see Figure~\ref{fig:curved-collar}) via
\begin{equation}\label{lem:normal-geometry}
\begin{aligned}
\beta(s):=\lambda\kappa(s)^{\frac{1}{3}},\;y_c(s):=\frac{\lambda^2}{2}\kappa(s)^{-\frac{1}{3}},\; p_s(s):=\sqrt{1-\lambda^2\kappa(s)^{\frac{2}{3}}}, \\
\text{and }p_y(s,y)=\sqrt{1-\frac{1-\lambda^2\kappa(s)^{\frac{2}{3}}}{(1-\kappa(s) y)^{2}}}=\sqrt{\lambda^2\kappa(s)^{\frac{2}{3}}-2\kappa(s)y}+O(\lambda^2).
\end{aligned}
\end{equation}
Here the expansion is uniform for $0\le y\le y_c(s)$. We next correct these leading expressions.

\begin{figure}[htbp]
\centering
\includegraphics[width=.9\textwidth]{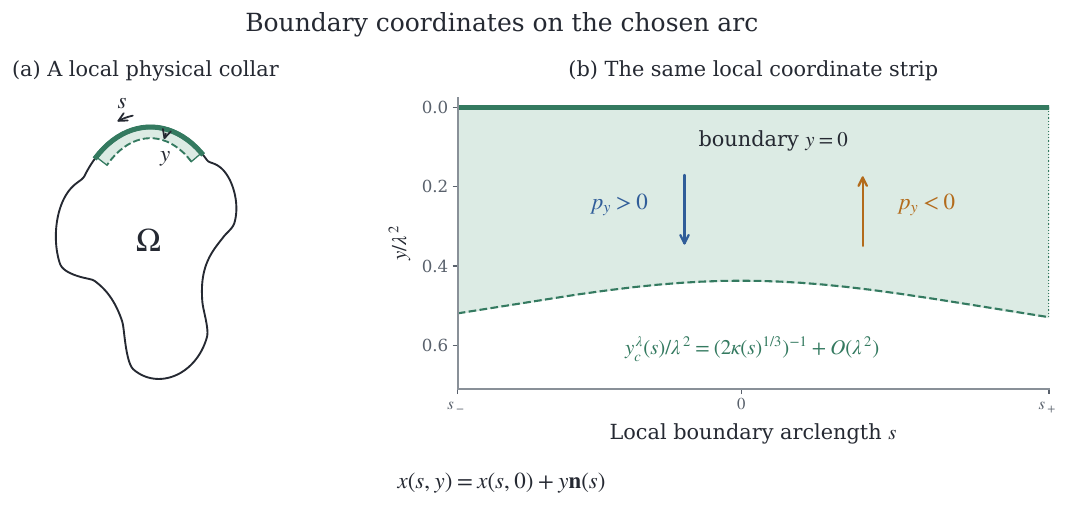}
\caption{Local geometry in boundary normal coordinates}\label{fig:curved-collar}
\end{figure}

\subsubsection{Correcting the phases without changing the boundary action}\label{subsubsec:curved-phase-correction}
We seek $\Phi_\lambda^\pm=\Phi_{0,\lambda}\pm R_\lambda$, while keeping $R_\lambda(s,0)=\frac{1}{3}\lambda^3$ exactly. We call this property a \emph{common boundary action}. It allows the same boundary--caustic compatibility to be used along the chosen arc as we observed in the frozen metric model.

In the disc, $\Phi_{0,\lambda}=b_\lambda \theta$ and $\zeta_\lambda$ depends only on $y$. With varying curvature, the turning coordinate also depends on $s$. This introduces a mixed phase error.

Based on the formal computation in the frozen metric, the appropriate turning coordinate from \eqref{lem:normal-geometry} is $y_c(s):=\frac{\lambda^2}{2}\kappa(s)^{-\frac{1}{3}}$. Then the leading term is 
\begin{equation*}
R_{0}(s,y)=\int_y^{y_c(s)}\sqrt{\lambda^2\kappa(s)^{\frac{2}{3}}-2\kappa(s)z}\dd z=\frac{\lambda^3}{3} \left(1-\frac{2\kappa(s)^{\frac{1}{3}}y}{\lambda^2}\right)^{\frac{3}{2}}.
\end{equation*}
Here $\zeta_0(s,y)=2^{-\frac{2}{3}}\lambda^2-(2\kappa(s))^{\frac{1}{3}}y$ is the leading part of the Airy coordinate. 
\vspace{1.5mm}

We go back to the phase equations \eqref{eq:phase-eq-s}. We fix the leading term $\zeta_0$, and seek a solution $\Phi$ of the second equation in \eqref{eq:phase-eq-s}:
\begin{equation*}
\partial_y\Phi+\frac{\kappa'(s)y}{3(1-\kappa(s)y)^2\kappa(s)}\partial_s\Phi=0,\;\Phi(s,0)=\int_0^s\sqrt{1-\lambda^2\kappa(\tau)^{\frac{2}{3}}}\dd\tau.
\end{equation*}
Thus, we obtain
\begin{equation*}
\Phi(s,y)=\int_0^s\sqrt{1-\lambda^2\kappa(\tau)^{\frac{2}{3}}}\dd\tau-\frac{\kappa'(s)\sqrt{1-\lambda^2\kappa(s)^{\frac{2}{3}}}}{6\kappa(s)}y^2+O(y^3).
\end{equation*}
Now we plug this expression into the first equation in \eqref{eq:phase-eq-s} and obtain
\begin{gather*}
|\nabla_g\Phi|_g^2=\mu^{-2}|\partial_s\Phi|^2+|\partial_y\Phi|^2=1-\lambda^2\kappa(s)^{\frac{2}{3}}+2\kappa(s)y +O(\lambda^4),\\
-\zeta_0|\nabla_g\zeta_0|_g^2=\left((2\kappa(s))^{\frac{1}{3}}y-2^{-\frac{2}{3}}\lambda^2\right)\left((2\kappa(s))^{\frac{2}{3}}+\mu^{-2}\left(\frac{2^{\frac{1}{3}}}{3}\kappa(s)^{-\frac{2}{3}}\kappa'(s)y\right)^2\right)\\
=2\kappa(s)y-\lambda^2\kappa(s)^{2/3}+O(\lambda^6).
\end{gather*}
Therefore, define 
\begin{equation*}
\Phi_{0}:=\int_0^s\sqrt{1-\lambda^2\kappa(\tau)^{\frac{2}{3}}}\dd\tau-\frac{\kappa'(s)\sqrt{1-\lambda^2\kappa(s)^{\frac{2}{3}}}}{6\kappa(s)}y^2,\text{ \; and \; }\zeta_{0}=2^{-\frac{2}{3}}\lambda^2-(2\kappa(s))^{\frac{1}{3}}y.
\end{equation*}
Then $\Phi_0$ and $\zeta_0$ will be the approximate solutions to \eqref{eq:phase-eq-s}:
\begin{equation*}
|\dd\Phi_0|_g^2+\zeta_0|\dd\zeta_0|_g^2-1=O(\lambda^4),\;\langle d\Phi_0,d\zeta_0\rangle_g=O(\lambda^4).
\end{equation*}
We expect that they are sufficient for our application. However, unfortunately, we need a better error estimate like $O(\lambda^6)$ for our later use in the first equation. We therefore seek a smooth correction $\zeta_\lambda$ of $\zeta_{0}$ and also the associated $\Phi_{0,\lambda}$, preserving this boundary value, and set $R_\lambda=\frac{2}{3}(\zeta_\lambda)^{\frac{3}{2}}$ on the oscillatory side. For $\Phi_{0,\lambda}$ and $\zeta_\lambda$, the error terms to be made small are
\begin{equation*}
E_0=|\dd\Phi_{0,\lambda}|^2+\zeta_\lambda|\dd\zeta_\lambda|^2-1 ,\text{ \; and \; }E_1=\langle \dd\Phi_{0,\lambda},\dd\zeta_\lambda\rangle.
\end{equation*}
We cancel the leading part of $E_1$ by writing $\Phi_{0,\lambda}=C_\lambda(s)+U(s,y)$. Using that $C_\lambda'=1+O(\lambda^2)$ and $\mu=1+O(\lambda^2)$ on $y=O(\lambda^2)$, the leading cancellation is
\begin{equation*}
0=\partial_s\zeta_{0}+\partial_yU\,\partial_y\zeta_{0}=-(2\kappa)^{1/3}\left(\frac{\kappa'}{3\kappa}y+\partial_yU\right).
\end{equation*}
Imposing $U(s,0)=0$, we obtain
\begin{equation*}
\partial_yU=-\frac{\kappa'}{3\kappa}y ,\text{ \; and \; }U(s,y)=-\frac{\kappa'(s)}{6\kappa(s)}y^2.
\end{equation*}
This correction recovers the preceding construction. We then add a quartic correction to $C_\lambda'$ and an order-$\lambda^4$ correction to $\zeta_{0}$ (as we presented in Section~\ref{sec:glancing-disc}). These cancel the remaining quartic term in $E_0$, with the latter correction zero at $y=0$. We now specify the family used in Proposition~\ref{prop:curved-phases}. For readers' convenience, the proof of Lemma~\ref{lem:airy-coord} derives the full coefficients.

\begin{proof}[Proof of Proposition~\ref{prop:curved-phases}]
Define
\begin{equation}\label{eq:W0}
\begin{aligned}
\Phi_{0,\lambda}(s,y):=C_\lambda(s)+U(s,y),&\text{ where  }U(s,y):=-\frac{\kappa'(s)}{6\kappa(s)}y^2,\\
C_\lambda'(s):=1-\frac12\lambda^2\kappa(s)^{2/3}+\lambda^4c_4(s),&\text{ \; with \; }c_4(s):=\frac{\kappa^{4/3}}{120}+\frac{\kappa''}{45\kappa^{5/3}}-\frac{4(\kappa')^2}{135\kappa^{8/3}}.
\end{aligned}
\end{equation}
Choose the origin of the local arclength coordinate at $p$ and take the primitive with $C_\lambda(0)=0$ for every $\lambda$. Use the polynomial Airy coordinate $\zeta_\lambda$ in \eqref{eq:airy-coeff}.

Lemma~\ref{lem:airy-coord} constructs $\zeta_\lambda$ as a polynomial in $y/\lambda^2$ and cancels the quadratic and quartic terms in $E_0$. The term $U$ cancels the leading mixed term $E_1$. The boundary value of the normal correction is zero, preserving the displayed boundary action.

The simple turning zero follows from $\partial_y\zeta_\lambda=-(2\kappa)^{1/3}+O(\lambda^2)$. When $\kappa=1$, $U=0$, $c_4=1/120$, and direct substitution into \eqref{eq:airy-coeff} gives \eqref{eq:gl-disc-phase}.
\end{proof}

The common choice of primitive matters when we compare different Fourier modes: a derivative estimate for $C_\lambda$ becomes an estimate for the phase itself on the fixed chart. We use this same choice throughout.

\subsubsection{Choosing the coordinated scales}\label{subsubsec:curved-scales}
We first read the geometric requirement from the fixed profile relation $f_*=\tau_{1/4}g_*$. For branch arguments $\frac{\Phi_{0,\lambda}-t\pm R_\lambda}{2\pi h}+\alpha$, the boundary phase difference is
\begin{equation*}
\psi-\phi=-\frac{R_\lambda(s,0)}{\pi h}.
\end{equation*}
To use the same quarter-period translation at every boundary point, we impose
\begin{equation*}
-\frac{R_\lambda(s,0)}{\pi h}=\frac{1}{4}\pmod1.
\end{equation*}
We therefore keep the disc boundary action $R_\lambda(s,0)=\lambda^3/3$ independent of $s$, and then coordinate $h$ with $\lambda$ as follows.
\vspace{1.5mm}

On the oscillatory collar, $\zeta_\lambda>0$ and $\zeta_\lambda=O(\lambda^2)$. Since $\dd R_\lambda=\sqrt{\zeta_\lambda}\dd\zeta_\lambda$, the estimates in \eqref{eq:eik} give
\begin{align*}
|\dd\Phi_\lambda^\pm|^2-1 =E_0\pm2\sqrt{\zeta_\lambda}E_1 =O(\lambda^6)+O(\lambda)O(\lambda^4)=O(\lambda^5).
\end{align*}
We later (in Sections~\ref{subsec:curved-boundary}--\ref{subsec:curved-correction}) need to require
\begin{equation}\label{eq:window}
\lambda^5\ll h\ll\lambda^4,\text{ \; and \; }\lambda^2\ll w,
\end{equation}
where $w$ is the tangential width (see Figure~\ref{fig:curved-scales}). 
\vspace{1.5mm}

As in the disc, choose $h$ of order $\lambda^{9/2}$ to satisfy both requirements, and adjust it to impose the quarter-period boundary relation exactly:
\begin{equation}\label{eq:scales}
j:=\lceil\lambda^{-3/2}\rceil, \quad h:=\frac{\lambda^3}{3\pi(j-1/4)}.
\end{equation}
This places $h$ between the competing powers:
\begin{equation*}
h\sim\lambda^{9/2},\; \frac{\lambda^5}{h}\sim\lambda^{1/2},\text{ \; and \; } \frac{h}{\lambda^4}\sim\lambda^{1/2}.
\end{equation*}
It also gives
\begin{equation*}
-\frac{R_\lambda(s,0)}{\pi h}=-j+1/4=1/4\pmod1, \; \frac{R_\lambda(s,0)}{2\pi h}=\tfrac12(j-1/4)\rightarrow\infty.
\end{equation*}

\begin{figure}[H]
\centering
\includegraphics[width=.9\textwidth]{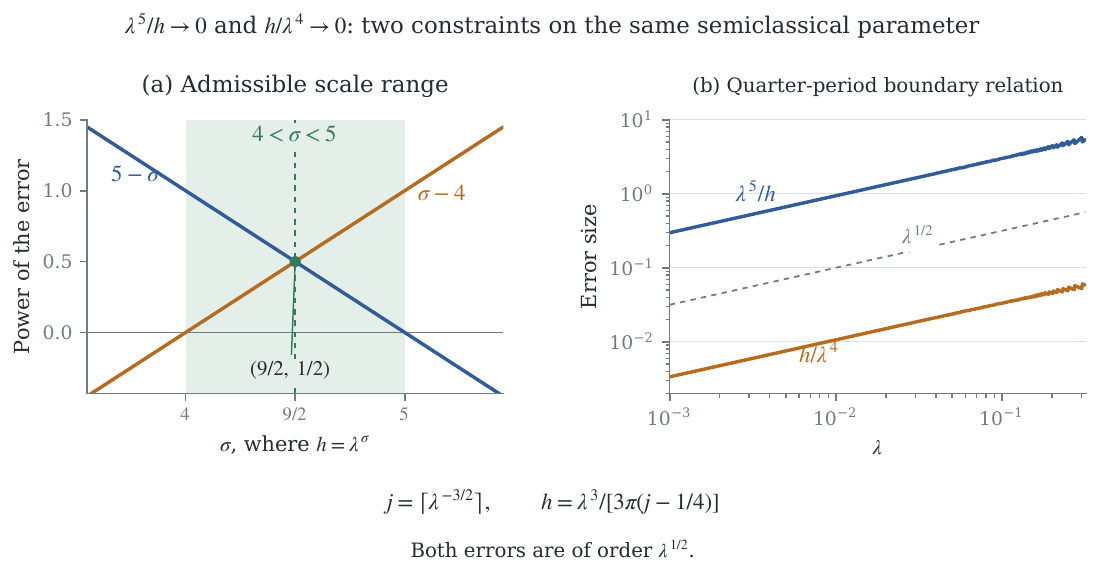}
\caption{The two requirements on $h$}\label{fig:curved-scales}
\end{figure}

\begin{remark}\label{rem:curved-disc-scales}
In the disc, $\Phi_{0,\lambda}=b_\lambda\theta$ and $\zeta_\lambda=\zeta_\lambda(y)$ give $E_1=0$. Thus the phase contribution to the normalized residual is $O(\lambda^6/h)$, as in Lemma~\ref{lem:gl-finite}. With varying curvature, the mixed term $E_1=O(\lambda^4)$ gives the stronger requirement $\lambda^5/h\to0$. The modewise phase comparison still requires $h/\lambda^4\to0$. Hence the same choice $h\sim\lambda^{9/2}$ works in both constructions.
\end{remark}

\subsection{Exact Dirichlet condition and a common amplitude}\label{subsec:curved-boundary}

We seek an approximate wave whose leading time derivative has the same full-profile form as in Section~\ref{subsubsec:gl-packet}. With the curved phases just constructed, the target and its finite Fourier approximation are
\begin{equation*}
S_{\lambda,w,z,\alpha}=A_{\lambda,w}(f_*(\phi)-g_*(\psi)),\; S_{\lambda,w,\ell}=A_{\lambda,w}(f_*^{(\ell)}(\phi)-g_*^{(\ell)}(\psi)).
\end{equation*}
Both expressions are extended by zero outside the moving packet region; their phases and common amplitude are specified in Proposition~\ref{prop:curved-support} below. The full profiles define exact matching. The finite profiles $f_*^{(\ell)},g_*^{(\ell)}$ instead allow us to construct a finite Airy sum $v_\ell$ and control its errors before letting $\ell\to\infty$.

We follow the Airy realization in Lemma~\ref{lem:gl-finite}, now with variable curvature and support in one boundary chart. Choosing Airy zeros gives an exact Dirichlet condition, but the localized Airy sum is not yet an exact solution of the wave equation. Section~\ref{subsec:curved-correction} corrects its residual, and Section~\ref{subsec:curved-full} recovers the full profile.
\vspace{1.5mm}

There are three tasks before returning to the observation geometry.
\begin{itemize}[leftmargin=*,itemsep=3pt]
\item Choose the Airy-zero indices $j_n$ so that each mode vanishes at the boundary and its phase remains close to the common phase.
\item Choose the mode amplitudes so that the leading time derivative of the approximate wave $v_\ell$ has a common real factor $A_{\lambda,w}$.
\item Read the generated approximate velocity from these branches, using the full profiles $f_*,g_*$ to define $S_{\lambda,w,z,\alpha}$ and $\Gamma_{\lambda,w,z,\alpha}$.
\end{itemize}

\subsubsection{Choosing the Airy-zero indices}

Fix a Fourier truncation $H_*^{(\ell)}(z)=\sum_{n\le \ell,\ n\equiv1(\mathrm{mod}4)}d_{*,n}z^n$ of the diamond profile. Its travelling-wave profiles are
\begin{align*}
f_*^{(\ell)}(\xi):=\Real\left(e^{i\pi/4}H_*^{(\ell)}(e^{2\pi i\xi})\right), \text{ \; and \; } g_*^{(\ell)}(\xi):=\Real\left(e^{-i\pi/4}H_*^{(\ell)}(e^{2\pi i\xi})\right).
\end{align*}
Here $g_*^{(\ell)}=\calH f_*^{(\ell)}$. The same restriction $n\equiv1\pmod4$ that gives profile compatibility will make the Airy-zero indices below integers, as in Lemma~\ref{lem:gl-finite}.

\smallskip\noindent\emph{Exact Dirichlet condition.} Each Fourier mode must vanish at an Airy zero. For the positive zeros $z_j$ of $\Ai(-z)$, the classical expansion is \cite{Olver}
\begin{equation*}
z_j^{3/2}=\frac{3\pi}2(j-1/4)+O(j^{-1}).
\end{equation*}
For $h_n=h/n$ and $n\leq \ell$, the leading actions agree if
\begin{equation*}
h_n(j_n-1/4)=h(j-1/4),\text{ \; and \; }j_n-1/4=n(j-1/4).
\end{equation*}
This determines
\begin{equation}\label{eq:airy-lock}
h_n:=h/n, \quad j_n:=nj-(n-1)/4 ,\text{ \; and \; }\lambda_n^3:=2h_nz_{j_n}^{3/2}.
\end{equation}
Since $n\equiv1\pmod4$, the index $j_n$ is an integer. The definition of $\lambda_n$ gives
\begin{align*}
\frac{\zeta_{\lambda_n}(s,0)}{h_n^{2/3}} =\frac{2^{-2/3}\lambda_n^2}{h_n^{2/3}} =\left(\frac{\lambda_n^3}{2h_n}\right)^{2/3}=z_{j_n}.
\end{align*}
For the scale comparison, the Airy-zero expansion yields
\begin{align*}
\lambda_n^3=3\pi h_n(j_n-1/4)+O_\ell(h/j)=3\pi h(j-1/4)+O_\ell(h/j)=\lambda^3\left(1+O_\ell(j^{-2})\right).
\end{align*}
Consequently,
\begin{equation}\label{eq:boundary-zero}
\Ai\left(-\zeta_{\lambda_n}(s,0)/h_n^{2/3}\right)=\Ai(-z_{j_n})=0 ,\text{ \; and \; }\lambda_n/\lambda=1+O_\ell(j^{-2}).
\end{equation}
The boundary condition is therefore exact for each Fourier mode. It remains to show that these different modes can be read using one common pair of propagation phases.

\vspace{2mm} \noindent\emph{Recovery of the common profile.} The normal scales $\lambda_n$ are slightly different. A common profile requires more than $\lambda_n/\lambda\to1$, since phases are divided by $h_n$. We first compare the unscaled tangential phases. Using \eqref{eq:boundary-zero} and the formula for $C_\lambda'$, we obtain
\begin{gather*}
|\lambda_n-\lambda|\le C_\ell\lambda j^{-2},\\
|C_{\lambda_n}'-C_\lambda'| \le C\left(|\lambda_n^2-\lambda^2|+|\lambda_n^4-\lambda^4|\right) \le C_\ell\lambda^2j^{-2}.
\end{gather*}
Since $C_{\lambda_n}(0)=C_\lambda(0)=0$, integration gives, on the fixed arclength chart,
\begin{equation*}
|C_{\lambda_n}(s)-C_\lambda(s)| =\left|\int_0^s(C_{\lambda_n}'-C_\lambda')(\sigma)\dd\sigma\right| \le C_{\ell,T}\lambda^2j^{-2}.
\end{equation*}
This implies that
\begin{equation}\label{eq:phase-coherence}
\frac{|C_{\lambda_n}-C_\lambda|}{h_n} \le C_{\ell,T}\frac{\lambda^2j^{-2}}h \le C_{\ell,T}\frac {h}{\lambda^4}\rightarrow0.
\end{equation}
The last comparison uses $j\sim\lambda^3/h$. Since $h_n=h/n$, this controls the oscillatory factors themselves:
\begin{equation*}
\left|e^{iC_{\lambda_n}/h_n}-e^{inC_\lambda/h}\right| \le\frac{|C_{\lambda_n}-C_\lambda|}{h_n}\rightarrow0.
\end{equation*}
On a fixed regular part of the normal collar, the same scale comparison gives
\begin{equation*}
\frac{|R_{\lambda_n}-R_\lambda|}{h_n} \le C_\ell\frac{\lambda^3j^{-2}}h \le C_\ell j^{-1}.
\end{equation*}
Here and below the constants for this comparison may depend on the chosen regular band. Since $U$ in \eqref{eq:W0} is independent of $\lambda$, the full branch factors therefore satisfy
\begin{align*}
\left|e^{i(\Phi_{0,\lambda_n}\pm R_{\lambda_n}-t)/h_n}-e^{in(\Phi_{0,\lambda}\pm R_\lambda-t)/h}\right|\quad\le\frac{|C_{\lambda_n}-C_\lambda|+|R_{\lambda_n}-R_\lambda|}{h_n}\le C_{\ell,T}\left(\frac{h}{\lambda^4}+j^{-1}\right)\rightarrow0.
\end{align*}
These phase estimates explain the upper restriction on $h$ in \eqref{eq:window}.

The amplitudes also agree asymptotically. On a fixed oscillatory band with $y_c^\lambda(s)-y\ge\delta\lambda^2$, the explicit formula for $\zeta_\lambda$ and \eqref{eq:boundary-zero} give
\begin{equation*}
\frac{\zeta_{\lambda_n}}{\zeta_\lambda}=1+O_{\ell,\delta}(j^{-2}), \text{ hence }(\zeta_{\lambda_n})^{-1/4}=(\zeta_\lambda)^{-1/4}\left(1+O_{\ell,\delta}(j^{-2})\right).
\end{equation*}

\subsubsection{Constructing the approximate wave}
We first choose the curvature-dependent factor $A_0(s)$ in the Airy amplitude. We may not have the exact equation $\widetilde{r}_1=0$ as in Lemma~\ref{lem:exact-Airy-solution}. The transport expression to be made small is
\begin{equation*}
2\langle \dd\Phi_{0,\lambda},\dd A_0\rangle_g+A_0\Delta_g \Phi_{0,\lambda}.
\end{equation*}
In the boundary neighborhood, the corrected phase satisfies
\begin{equation*}
\Delta_g\Phi_{0,\lambda}=-\frac{\kappa'}{3\kappa}+O(\lambda^2),\; \mu^{-2}\partial_s\Phi_{0,\lambda}=1+O(\lambda^2).
\end{equation*}
For an amplitude factor $A_0(s)$, cancellation of the leading transport term therefore gives
\begin{equation*}
2\frac{A_0'}{A_0}-\frac{\kappa'}{3\kappa}=0, \quad\text{so we take }A_0(s)=\kappa(s)^{1/6}.
\end{equation*}
Indeed, $2\mu^{-2}\partial_s\Phi_{0,\lambda}\frac{A_0'}{A_0}+\Delta_g \Phi_{0,\lambda}=O(\lambda^2)$. The remaining transport terms are estimated in Section~\ref{subsec:curved-correction}.

For a common leading amplitude, multiply by $h_n^{5/6}$ to cancel the Airy factor $h_n^{1/6}$ and the time derivative $h_n^{-1}$:
\begin{equation*}
h_n^{5/6}\cdot h_n^{1/6}\cdot h_n^{-1}=1.
\end{equation*}
This is the same Airy superposition as in Lemma~\ref{lem:gl-finite}, with the curvature factor just chosen. See also the general discussion in \cite[Appendix~A.5.2]{AOSI}.

Fix a moving packet region $\mathcal U_{\lambda,w,z}$ as in Definition~\ref{def:curved-packet}, a real nonzero cutoff $\chi\in C_c^\infty((-1/2,1/2))$, and a common phase shift $\alpha\in\R/\Z$. Choose fixed $0<Y_0<Y_1$ above all scaled turning depths on the chart, and a smooth cutoff $\rho\in C^{\infty}$ such that $\rho=1$ on $[0,Y_0]$ and $\rho=0$ on $[Y_1,\infty)$. On this local chart, take the real part of
\begin{equation}\label{eq:airy-wave}
v_\ell^{\mathbb C}:=-2\sqrt{\frac{\pi}{\lambda w}} \chi\left(\frac{\theta-t-z}{w}\right)\rho\left(\frac{y}{\lambda^2}\right)\kappa(s)^{1/6} \sum_{n\le \ell,\ n\equiv1(\mathrm{mod}4)}d_{*,n}h_n^{5/6} \Ai\left(-\frac{\zeta_{\lambda_n}}{h_n^{2/3}}\right)e^{i\frac{\Phi_{0,\lambda_n}-t}{h_n}+2\pi in\alpha}.
\end{equation}
For fixed $\ell$ and sufficiently small $\lambda$, each retained mode satisfies
\begin{equation*}
\zeta_{\lambda_n}(s,y)\le -c\lambda^2, \quad \forall y\in[Y_0\lambda^2,Y_1\lambda^2].
\end{equation*}
Thus the cutoff errors are exponentially small. This cutoff lies beyond the turning curve; the Airy wave remains smooth across that curve. The target $S_{\lambda,w,\ell}$ instead keeps only the oscillatory-side expression and is zero on the other side. Their comparison across the turning curve will use the integral Airy estimates in Section~\ref{subsec:curved-correction}.

The tangential envelope has a different role, and its transport error is controlled by $\lambda^2/w$. Here the tangential envelope is built directly into the approximate wave, whereas the construction in \cite[Section~3.4]{AOSI} achieves angular localization by replacing the angular factors and removing small tails.

\vspace{2mm} We take $\lambda$ smaller so that $Y_1\lambda^2<d_0$. By \eqref{eq:local-buffer}, the resulting approximation has support
\begin{equation*}
\operatorname{supp}v_\ell(t) \subset x\!\left(\left(t+z+[-w/2,w/2]\right)\times[0,Y_1\lambda^2]\right) \subset V, \quad \forall t\in[0,T].
\end{equation*}
The moving envelope stays away from the tangential edges of the chart, and the normal cutoff vanishes before $y=d_0$. We therefore extend by zero to the rest of $\Omega$. The extension is smooth across the artificial edges and creates no additional error there. We note that $v_\ell|_{\partial\Omega}=0$ exactly due to the Airy zeros on the chosen arc, and zero extension elsewhere. 
\subsubsection{The generated approximate velocity and its matching region}
The preceding mode calculation determines the common phases and amplitude. As in Section~\ref{subsubsec:gl-packet}, we separate the factors of the generated velocity amplitude:
\begin{equation*}
A_{\lambda,w} =\underbrace{(\lambda w)^{-1/2}}_{\text{mass scaling}} \quad \underbrace{\chi\!\left(\frac{s-t-z}{w}\right)}_{\text{tangential localization}} \quad \underbrace{\kappa(s)^{1/6}}_{\text{curvature transport}} \quad \underbrace{(\zeta_\lambda)^{-1/4}}_{\text{Airy branch amplitude}} \quad\text{on }\mathcal U_{\lambda,w,z}.
\end{equation*}
We extend this amplitude by zero outside the moving packet region and on its turning curve. The next proposition records the full-profile function, its support and its matching region. Proposition~\ref{prop:finite-airy} proves velocity approximation for its finite-profile version, and Theorem~\ref{thm:realization} recovers the full profile. On this chart define the common phases and amplitude by
\begin{equation}\label{eq:actual-phases}
\begin{split}
\phi&:=(2\pi h)^{-1}(\Phi_{0,\lambda}-t+R_\lambda)+\alpha,\;\psi:=(2\pi h)^{-1}(\Phi_{0,\lambda}-t-R_\lambda)+\alpha\\
\text{and }&  A_{\lambda,w}:=(\lambda w)^{-1/2}\chi\left((s-t-z)/w\right)\kappa(s)^{1/6}(\zeta_\lambda)^{-1/4}.
\end{split}
\end{equation}
The amplitude is extended by zero outside the moving packet region and on its turning curve.

On $\mathcal U_{\lambda,w,z}$, define
\begin{equation}\label{eq:curved-matching}
\begin{aligned}
S_{\lambda,w,z,\alpha}(t,s,y) &:=A_{\lambda,w}(t,s,y)(f_*(\phi)-g_*(\psi)),\\
\Gamma_{\lambda,w,z,\alpha} &:=\mathcal U_{\lambda,w,z}\cap\{(\phi,\psi)\pmod1\in\mathcal R\}.
\end{aligned}
\end{equation}
Extend $S_{\lambda,w,z,\alpha}$ by zero outside the moving packet region and on its turning curve.%

\begin{prop}\label{prop:curved-support}
Let $\Omega\subset\R^2$ be a smooth bounded domain. Fix the local geometry of Lemma~\ref{lem:local-geometry}, with collar coordinates $x(s,y)$, arclength interval $\mathcal I_{\rm in}$, collar depth $d_0$ and time bound $T_{\rm loc}$. Fix $0<T<T_{\rm loc}$, the diamond profiles $f_*,g_*$ and rectangles $\mathcal R$ from Proposition~\ref{prop:diamond}, and a real nonzero $\chi\in C_c^\infty((-1/2,1/2))$. For sufficiently small $\lambda>0$, use the phases of Proposition~\ref{prop:curved-phases} and the scale $h$ in \eqref{eq:scales}. Fix $w,z,\alpha$ with
\begin{equation*}
0<w<1,\quad z+[-w/2,w/2]\subset\mathcal I_{\rm in} ,\text{ \; and \; }\alpha\in\R/\Z.
\end{equation*}

Then $S_{\lambda,w,z,\alpha}\in L^2((0,T)\times\Omega)$, and the following properties hold.
\begin{enumerate}[label=(\roman*),leftmargin=*,itemsep=1mm]
\item \emph{Support in the moving packet region.} We have $S_{\lambda,w,z,\alpha}=0$ outside $\mathcal U_{\lambda,w,z}$.%
\item \emph{Additional cancellation through profile matching.} The full diamond profiles give
\begin{equation*}
\1_{\Gamma_{\lambda,w,z,\alpha}}S_{\lambda,w,z,\alpha}=0.
\end{equation*}
\end{enumerate}
\end{prop}
\begin{proof}
\noindent\emph{The common amplitude and support in the moving packet region.} %

After the phase replacement justified by \eqref{eq:phase-coherence}, the finite profile sum reads
\begin{align*}
&\Real\sum_{\substack{n\le\ell\\n\equiv1\ (4)}}d_{*,n} \left(e^{i\pi/4}e^{2\pi in\phi}-e^{-i\pi/4}e^{2\pi in\psi}\right)\\
&\quad=\Real\left(e^{i\pi/4}H_*^{(\ell)}(e^{2\pi i\phi})\right) -\Real\left(e^{-i\pi/4}H_*^{(\ell)}(e^{2\pi i\psi})\right)\\
&\quad=f_*^{(\ell)}(\phi)-g_*^{(\ell)}(\psi).
\end{align*}
Replacing the finite profiles by $f_*,g_*$ gives the target \eqref{eq:curved-matching}; Proposition~\ref{prop:finite-airy} and Theorem~\ref{thm:realization} will justify the approximation. Here $\Phi_{0,\lambda}=C_\lambda+U$ retains the normal correction in \eqref{eq:W0}.

\vspace{2mm} Support follows directly from the definition. For the $L^2$ bound, the simple turning zero and the curvature bounds give
\begin{equation*}
\int_0^{y_c^\lambda(s)}(\zeta_\lambda(s,y))^{-1/2}\mu(s,y)\dd y\le C\lambda
\end{equation*}
uniformly on the fixed chart. Hence, for every $t\in[0,T]$, $\|A_{\lambda,w}(t,\cdot)\|_{L^2(\Omega)}^2\le C$. Since $f_*,g_*$ are bounded, $S_{\lambda,w,z,\alpha}\in L^2((0,T)\times\Omega)$.

\smallskip\noindent\emph{Additional cancellation through profile matching.} On each component of $\mathcal R$, the two full profiles take the same plateau value. Pulling back by $(\phi,\psi)$ therefore gives $\1_{\Gamma_{\lambda,w,z,\alpha}}S_{\lambda,w,z,\alpha}=0$.

The same phases also recover the boundary relation
\begin{equation*}
\psi=\phi+1/4\pmod1, \quad\text{hence }f_*(\phi)-g_*(\psi)=0,
\end{equation*}
at $y=0$, because $R_\lambda(s,0)=\lambda^3/3$ and \eqref{eq:scales} hold for every $s$.
\end{proof}

Inside the collar, the same reference rectangles give matching regions whose physical shape depends on curvature; see Figure~\ref{fig:curved-matching}.

\begin{figure}[htbp]
\centering
\includegraphics[width=\textwidth]{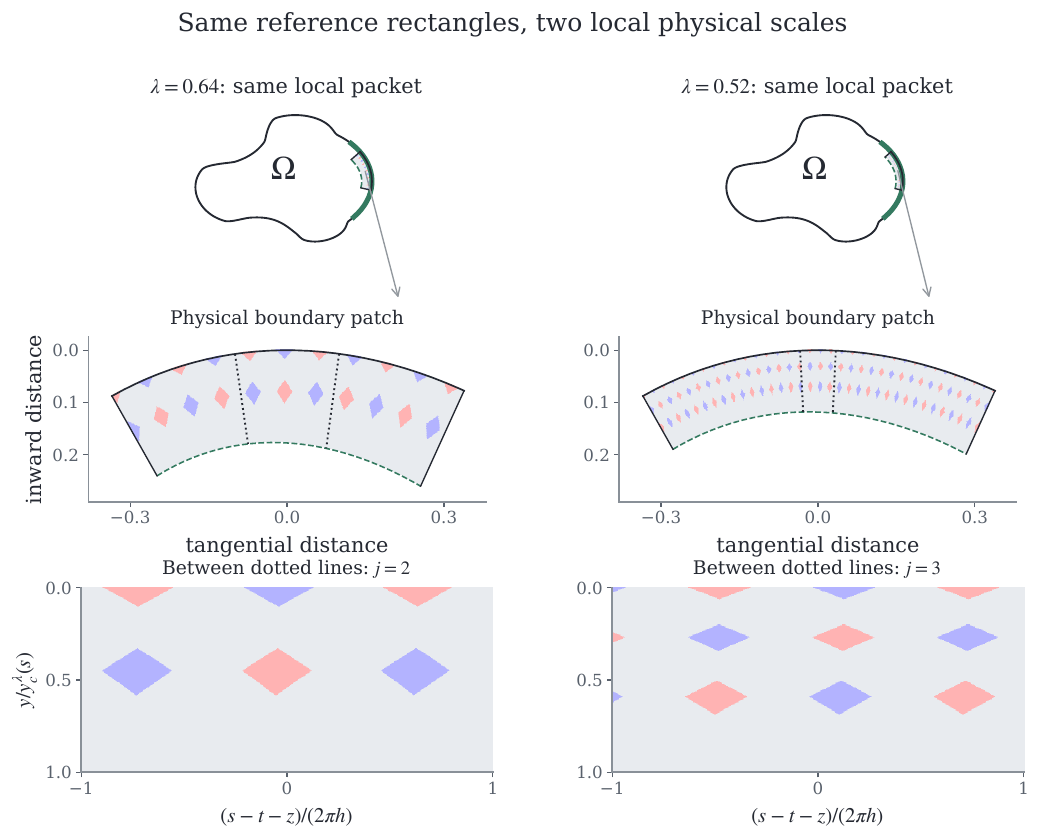}
\caption{The same matching rectangles near one boundary arc at two scales, with the center and tangential width of the moving packet region fixed.}\label{fig:curved-matching}
\end{figure}

\begin{remark}\label{rem:curved-S-params}
For a family $S_k:=c_kS_{\lambda_k,w_k,z_k,\alpha_k}$, we record the dependence as in the disc:
\begin{equation*}
S_k=c_k\,S\!\left(\Omega,q,\chi;z_k,\alpha_k;\lambda_k,w_k\right),\quad\text{where }q=f_*\quad\text{and}\quad g_*=\calH q.
\end{equation*}
Here $S$ is given by \eqref{eq:curved-matching}. We fix $\Omega,q,\chi$, the boundary coordinates and the primitive in \eqref{eq:W0}; $z_k,\alpha_k$ have the roles stated in Remark~\ref{rem:gl-S-params}. Equation \eqref{eq:scales} determines $h_k$ from $\lambda_k$. The normal depth varies with $s$ according to \eqref{eq:turning}, but remains of order $\lambda_k^2$; $w_k$ is the tangential width.
\end{remark}

\Needspace{8\baselineskip}
\subsubsection{The same observation rule on the curved domain}\label{subsubsec:curved-observation}

The matching cells now depend on curvature, but the assembly of the observation region uses exactly the support and matching properties from Proposition~\ref{prop:curved-support}. Before we make the statements. Define $S_k:=c_kS_{\lambda_k,w_k,z_k,\alpha_k}$ and $\Gamma_k:=\Gamma_{\lambda_k,w_k,z_k,\alpha_k}$ by Proposition~\ref{prop:curved-support}. The moving packet regions of Definition~\ref{def:curved-packet} are
\begin{equation}\label{eq:packet}
\mathcal U_k:=\mathcal U_{\lambda_k,w_k,z_k}=\{(t,x)\in(0,T)\times\Omega:s(x)-t\in I_k,\quad 0<y(x)<y_c^{\lambda_k}\left(s(x)\right)\}.
\end{equation}
All points in the moving packet regions lie in the fixed chart by \eqref{eq:local-buffer}.

\smallskip
With the phases $\phi_k,\psi_k$ from \eqref{eq:actual-phases}, the additional observation region \eqref{eq:gl-extra} is
\begin{equation*}
\calE_k:=\{(t,x)\in\mathcal U_k:\dist(\phi_k,\mathcal K)\le4\gamma_k \ \text{or}\ \dist(\psi_k,\mathcal K)\le4\gamma_k\},
\end{equation*}
where $\mathcal K=\{0,1/4,1/2,3/4\}$ and distances are taken on $\R/\Z$. Put
\begin{equation*}
\mathcal H_k:=\mathcal U_k\setminus(\Gamma_k\cup\calE_k).
\end{equation*}
The finite-stage and fixed observation regions are
\begin{equation}\label{eq:fixed-G}
G^{(N)}:=\left((0,T)\times\Omega\right)\setminus\bigcup_{k\le N}\mathcal H_k,\;G:=\bigcap_{N\ge1}G^{(N)}=\left((0,T)\times\Omega\right)\setminus\bigcup_{k\ge1}\mathcal H_k.
\end{equation}

\begin{prop}\label{prop:curved-observation}
Let $\Omega\subset\R^2$ be a smooth bounded domain. Fix the normal coordinates $x(s,y)$ and data $\mathcal I_{\rm in},d_0,T_{\rm loc}$ of Lemma~\ref{lem:local-geometry}. Fix $0<T<T_{\rm loc}$, the diamond profiles $f_*,g_*$ of Proposition~\ref{prop:diamond}, and a real nonzero envelope $\chi\in C_c^\infty((-1/2,1/2))$. Choose widths $0<w_k<1$ with $w_k\to0^+$ and centers $z_k\in\R$ such that the intervals
\begin{equation*}
I_k=z_k+(-w_k/2,w_k/2),\quad\text{with }\overline I_k\subset\mathcal I_{\rm in},
\end{equation*}
are pairwise disjoint. Fix common phase shifts $\alpha_k\in\R/\Z$, sufficiently small scales $\lambda_k\to0^+$ with turning depths below $d_0$, factors $c_k>0$, and reference widths $\gamma_k\to0^+$. Then, these Borel sets satisfy
\begin{equation}\label{eq:curved-stages}
\begin{gathered}
G^{(N)}\downarrow G,\\
G^{(N)}\cap\mathcal U_k=G\cap\mathcal U_k=\Gamma_k\cup\calE_k, \quad \forall k\in\{1,\ldots,N\}.
\end{gathered}
\end{equation}
Moreover,
\begin{equation}\label{eq:curved-observed}
\begin{aligned}
\1_{G^{(N)}}S_k &=\1_{G^{(N)}\cap\mathcal U_k}S_k =\1_{\Gamma_k\cup\calE_k}S_k\\
&=\1_{\calE_k}S_k=\1_GS_k, \quad \forall k\in\{1,\ldots,N\}.
\end{aligned}
\end{equation}
\end{prop}
\begin{remark}\label{rem:curved-params}
The set depends on the fixed domain and full profile through
\begin{equation}\label{eq:curved-G-param}
G=G\!\left(\Omega,q,T;\{z_k,\alpha_k\}_{k\ge1};\{\lambda_k,w_k,\gamma_k\}_{k\ge1}\right),\text{where }q=f_*\quad\text{and}\quad g_*=\calH q.
\end{equation}
The positions and propagation scales are as in Remark~\ref{rem:curved-S-params}. The reference width $\gamma_k$ changes $\calE_k$ and $G$ without changing $S_k$.

The widths $w_k$, centers $z_k$ and common phase shifts $\alpha_k$ may be prescribed first. For realization, we then fix a Fourier cutoff and take $\lambda_k$ sufficiently small. Theorem~\ref{thm:realization} gives non-observability for any $\gamma_k\to0$. For GCC, Section~\ref{sec:gcc} imposes stronger interval separation and further upper bounds on $\lambda_k$, allowed by that theorem, and chooses $\gamma_k$ as a function of $\lambda_k$.

A change of Fourier cutoff may require a smaller $\lambda_k$ and hence a different physical set. Once the entire sequence is chosen, $G$ is fixed for all the exact waves.
\end{remark}

\begin{proof}[Proof of Proposition~\ref{prop:curved-observation}]
Disjointness of $I_k$ gives disjointness of $\mathcal U_k$, independently of the varying turning depth. Proposition~\ref{prop:curved-support} supplies support in $\mathcal U_k$ and cancellation on $\Gamma_k$. The set-theoretic argument in the proof of Proposition~\ref{prop:gl-observation} therefore applies unchanged, giving \eqref{eq:curved-stages} and \eqref{eq:curved-observed}.
\end{proof}

\subsection{From approximate waves to exact waves}\label{subsec:curved-correction}

Let $u_\ell$ denote the exact wave with initial state $v_\ell[0]$. For fixed $\ell$, two errors remain: the exact-wave correction $\partial_t(u_\ell-v_\ell)$ and the error for the truncated profiles $\partial_tv_\ell-S_{\lambda,w,\ell}$. We control the first through the residual of $v_\ell$ and the second through the common-phase comparison and the Airy bounds.

Starting from the target, these are the two spacetime $L^2$ comparisons
\begin{equation*}
S_{\lambda,w,\ell} \underset{\substack{\text{Airy expansion}\\\text{and phase comparison}}}{\approx} \partial_tv_\ell \underset{\substack{\text{energy}\\\text{estimate}}}{\approx} \partial_tu_\ell.
\end{equation*}
The first comparison includes the turning layer through an integral estimate.

The common amplitude $A_0(s)=\kappa(s)^{1/6}$ cancels the leading transport term. The envelope was chosen to move at speed one, so
\begin{equation*}
(\partial_t+\partial_s)\chi\left((s-t-z)/w\right)=0.
\end{equation*}
The actual tangential transport coefficient is $\mu^{-2}\partial_s\Phi_{0,\lambda}=1+O(\lambda^2)$ in the collar. Since the envelope has no normal dependence, its remaining transport error is
\begin{equation*}
\left(\partial_t+\mu^{-2}\partial_s\Phi_{0,\lambda}\,\partial_s\right) \chi\left((s-t-z)/w\right) =\frac{\mu^{-2}\partial_s\Phi_{0,\lambda}-1}{w}\chi'\left((s-t-z)/w\right).
\end{equation*}
Thus this error has size $O(\lambda^2/w)$, not $O(1/w)$. Together with the phase-equation estimate, these cancellations give the residual bound below.
\begin{prop}\label{prop:finite-airy}
Let $\Omega\subset\R^2$ be a smooth bounded domain. Fix the collar coordinates $x(s,y)$ and data $\mathcal I_{\rm in},d_0,T_{\rm loc}$ of Lemma~\ref{lem:local-geometry}. Fix $0<T<T_{\rm loc}$, a real nonzero $\chi\in C_c^\infty((-1/2,1/2))$, a nonzero Fourier truncation $H_*^{(\ell)}$ of the diamond profile, and a center parameter $z$ and a common phase shift $\alpha$.

Use the scales \eqref{eq:scales} and choose $w=w(\lambda)$ with
\begin{equation*}
0<w<1,\quad z+[-w/2,w/2]\subset\mathcal I_{\rm in} ,\text{ \; and \; }\lambda^2/w\to0.
\end{equation*}
Let $v_\ell$ be the real approximation \eqref{eq:airy-wave}, with its normal cutoff below $d_0$ and zero extension to $\Omega$, and let $u_\ell$ solve \eqref{eq:wave} on $\Omega$ with $u_\ell[0]=v_\ell[0]$. Then
\begin{equation}\label{eq:forcing}
\|(\partial_t^2-\Delta)v_\ell\|_{L^1(0,T;L^2(\Omega))}\le C_{\ell,T}\left(\frac{\lambda^5}{h}+\lambda+\frac{\lambda^2}{w}+h(1+w^{-1}+\lambda^2w^{-2})\right)+o(1).
\end{equation}

\smallskip
Its leading time derivative, in full spacetime $L^2$, is
\begin{equation}\label{eq:airy-S}
S_{\lambda,w,\ell} =A_{\lambda,w}\left(f_*^{(\ell)}(\phi)-g_*^{(\ell)}(\psi)\right),
\end{equation}
with the phases and amplitude in \eqref{eq:actual-phases}, extended by zero outside \eqref{eq:curved-packet}. More precisely,
\begin{equation}\label{eq:finite-S-error}
\|\partial_tv_\ell-S_{\lambda,w,\ell}\|_{L^2((0,T)\times\Omega)}\rightarrow0
\end{equation}
as $\lambda\to0$ with $\ell$ fixed.

\smallskip
The exact wave also satisfies
\begin{equation}\label{eq:finite-exact}
\|\partial_tu_\ell-S_{\lambda,w,\ell}\|_{L^2((0,T)\times\Omega)}\rightarrow0.
\end{equation}
\end{prop}

\begin{proof}
\noindent\emph{Step 1. The wave-equation error.} The proof of Lemma~\ref{lem:airy-residual} gives the exact differentiation identity \eqref{eq:residual-id}. The integral Airy bounds are
\begin{equation}\label{eq:airy-norms}
\begin{aligned}
\int|\Ai(-\zeta_\lambda/h^{2/3})|^2\dd y\sim h^{1/3}\lambda, \text{ \; and \; } \int|\Ai'(-\zeta_\lambda/h^{2/3})|^2\dd y\le Ch^{-1/3}\lambda^3.
\end{aligned}
\end{equation}
They include the Airy zeros and the turning layer. In particular,
\begin{equation*}
\frac{\|\Ai'(-\zeta_\lambda/h^{2/3})\|_{L^2_y}} {\|\Ai(-\zeta_\lambda/h^{2/3})\|_{L^2_y}} \le C\lambda h^{-1/3}.
\end{equation*}
Relative to the leading time-derivative norm, the two phase-equation terms in \eqref{eq:residual-id} are therefore bounded by
\begin{align*}
\frac{\|E_0\|_{L^{\infty}}}{h}&\le C\frac{\lambda^6}{h},\\
\frac{\|E_1\|_{L^{\infty}}}{h^{2/3}} \frac{\|\Ai'(-\zeta_\lambda/h^{2/3})\|_{L^2_y}} {\|\Ai(-\zeta_\lambda/h^{2/3})\|_{L^2_y}} &\le C\frac{\lambda^4}{h^{2/3}}\frac{\lambda}{h^{1/3}} =C\frac{\lambda^5}{h}.
\end{align*}
Thus the mixed term $E_1$, rather than $E_0$, gives the stronger lower restriction on $h$. These estimates are uniform on the fixed chart; replacing $(\lambda,h)$ by $(\lambda_n,h_n)$ only changes the constant for fixed $\ell$.

The transport and envelope bounds in the proof of Lemma~\ref{lem:airy-residual} combine as
\begin{align*}
&\lambda^2(1+w^{-1})+\lambda(1+\lambda^2w^{-1}) +h(1+w^{-1}+\lambda^2w^{-2})\\
&\quad\le C\left(\lambda+\frac{\lambda^2}{w} +h(1+w^{-1}+\lambda^2w^{-2})\right).
\end{align*}
In the last term, the leading second derivatives of the moving envelope cancel: the coefficient of $\chi''/w^2$ is $1-\mu^{-2}=O(\lambda^2)$. This leaves $h\lambda^2/w^2$. The normal cutoff contributes only exponentially small errors. Combining these bounds gives \eqref{eq:forcing}.

All terms tend to zero under \eqref{eq:scales} and $\lambda^2/w\to0$. For the remaining envelope terms, notice that
\begin{align*}
\frac hw=\frac h{\lambda^2}\frac{\lambda^2}{w}\rightarrow0, \text{ \; and \; } \frac{h\lambda^2}{w^2}=\frac h{\lambda^2}\left(\frac{\lambda^2}{w}\right)^2\rightarrow0,
\end{align*}
since $h/\lambda^2\sim\lambda^{5/2}\to0$.

\vspace{2mm} \noindent\emph{Step 2. One common propagated profile.}

\noindent\emph{Spacetime approximation.} On a fixed regular band, we use \eqref{eq:phase-coherence} and the amplitude comparison. The normal phase discrepancy is $O_\ell(j^{-1})$.

Differentiating the moving envelope contributes an $L^2$ error $O_{\ell,T}(h/w)$, which tends to zero since $h/w=(h/\lambda^2)(\lambda^2/w)\to0$. Equation \eqref{eq:airy-osc} then yields \eqref{eq:airy-S} on that band.

To pass through the turning region in $L^2$, we use the integral Airy bounds. The relative squared mass in the Airy transition layer is $O\left((h/\lambda^3)^{1/3}\right)$.

More explicitly, for fixed $\delta>0$ let $\mathcal V_{\lambda,\delta}$ be the spacetime band in the collar given by $|y-y_c^\lambda(s)|\le\delta\lambda^2$. The Airy bounds and the integrable squared fold amplitude give
\begin{equation*}
\limsup_{\lambda\to0^+}\left( \|\1_{\mathcal V_{\lambda,\delta}}\partial_tv_\ell\|_{L^2}^2 +\|\1_{\mathcal V_{\lambda,\delta}}S_{\lambda,w,\ell}\|_{L^2}^2\right) \le C_{\ell,T}\sqrt\delta.
\end{equation*}
Here the norms are on $(0,T)\times\Omega$, with $\ell$ fixed and the coordinated scale conditions imposed.

Outside this band, the regular expansion on the oscillatory side and Airy decay on the other side imply
\begin{equation*}
\|\1_{\mathcal V_{\lambda,\delta}^c} (\partial_tv_\ell-S_{\lambda,w,\ell})\|_{L^2((0,T)\times\Omega)}\rightarrow0.
\end{equation*}
Combining the two parts first at fixed $\delta$, we obtain
\begin{equation*}
\limsup_{\lambda\to0^+} \|\partial_tv_\ell-S_{\lambda,w,\ell}\|_{L^2((0,T)\times\Omega)}^2 \le C_{\ell,T}\sqrt\delta.
\end{equation*}
Letting $\delta\to0^+$ proves \eqref{eq:finite-S-error} on the full spacetime region.

\Needspace{8\baselineskip} \vspace{2mm} \noindent\emph{Energy near the turning curve at fixed time.} For the initial-energy calculation below, we also need a fixed-time estimate. Let $B_{\lambda,\delta}\subset\Omega$ be the spatial collar strip $|y-y_c^\lambda(s)|\le\delta\lambda^2$, so that $\mathcal V_{\lambda,\delta}=(0,T)\times B_{\lambda,\delta}$. We claim that
\begin{equation}\label{eq:turning-energy}
\limsup_{\lambda\to0^+}\sup_{0\le t\le T} \int_{B_{\lambda,\delta}}\left(|\partial_tv_\ell(t,x)|^2+|\nabla v_\ell(t,x)|^2\right)\dd x \le C_\ell\sqrt\delta.
\end{equation}
Here $\ell$ and $\delta>0$ are fixed before taking $\lambda\to0^+$. Indeed, the turning depths for $\lambda_n$ differ from $y_c^\lambda(s)$ by $O_\ell(\lambda^2j^{-2})$. For small $\lambda$, a fixed enlargement of the strip therefore contains all these turning curves. The Airy integral bound gives, uniformly in $s$,
\begin{equation*}
\int_{|y-y_c^\lambda(s)|\le\delta\lambda^2} |\Ai(-\zeta_{\lambda_n}/h_n^{2/3})|^2\dd y \le Ch_n^{1/3}\lambda\left(\sqrt\delta+o_\ell(1)\right).
\end{equation*}
The leading time derivative has squared prefactor $h_n^{-1/3}$, and the envelope has squared $L^2$ norm $w\|\chi\|_2^2$ in $s$. The normalization factors therefore combine as
\begin{equation*}
\underbrace{(\lambda w)^{-1}}_{b_{\lambda,w}^2}\,w\,h_n^{-1/3} \cdot h_n^{1/3}\lambda\left(\sqrt\delta+o_\ell(1)\right) =\sqrt\delta+o_\ell(1).
\end{equation*}
Summing over the fixed modes gives \eqref{eq:turning-energy} for $\partial_tv_\ell$. The same envelope identity for $\chi'$ controls its differentiated term by $O_\ell(h/w)$ in $L^2(\Omega)$.

\Needspace{7\baselineskip} To control the spatial derivatives, let $V_n$ be the real $n$-th summand of \eqref{eq:airy-wave}, including its coefficient and cutoffs. Using $\partial_y\Phi_{0,\lambda_n}=O(\lambda^2)$, $\zeta_{\lambda_n,s}=O(\lambda^2)$, and \eqref{eq:airy-norms}, we obtain, uniformly for $t\in[0,T]$,
\begin{align*}
\|\partial_yV_n(t)\|_{L^2(\Omega)} &\le C|d_{*,n}|\lambda+O_\ell(e^{-c\lambda^3/h}),\\*
\|\partial_sV_n(t)+\partial_s\Phi_{0,\lambda_n}\partial_tV_n(t)\|_{L^2(\Omega)} &\le C|d_{*,n}|\left(h_n(1+\lambda^2/w)+\lambda^3\right)+O_\ell(e^{-c\lambda^3/h}).
\end{align*}
In the second line, the envelope derivative is multiplied by $1-\partial_s\Phi_{0,\lambda_n}=O(\lambda^2)$; the $\Ai'$ term contributes $O(|d_{*,n}|\lambda^3)$. The normal-cutoff terms are exponentially small. Since the collar metric is uniformly comparable to the Euclidean coordinate metric, summing over the fixed finite set of modes proves \eqref{eq:turning-energy}.

\vspace{2mm} \noindent\emph{Step 3. The exact-wave correction.} Recall that $u_\ell$ is the exact wave with initial state $v_\ell[0]$. Their difference solves
\begin{equation*}
\begin{cases}
(\partial_t^2-\Delta)(u_\ell-v_\ell)=-(\partial_t^2-\Delta)v_\ell,&(t,x)\in(0,T)\times\Omega,\\
(u_\ell-v_\ell)|_{\partial\Omega}=0,\\
(u_\ell-v_\ell)[0]=0.
\end{cases}
\end{equation*}
The energy estimate therefore gives
\begin{equation}\label{eq:energy-correction}
\sup_{t\in[0,T]}\|(u_\ell-v_\ell)[t]\|_{H_0^1\times L^2} \le C\|(\partial_t^2-\Delta)v_\ell\|_{L^1_tL^2_x}\to0.
\end{equation}

We now compare the exact wave with the generated approximate velocity:
\begin{equation*}
\partial_tu_\ell-S_{\lambda,w,\ell} =\partial_t(u_\ell-v_\ell)+(\partial_tv_\ell-S_{\lambda,w,\ell}).
\end{equation*}
Combining \eqref{eq:energy-correction} and \eqref{eq:finite-S-error}, we obtain
\begin{align*}
&\|\partial_tu_\ell-S_{\lambda,w,\ell}\|_{L^2((0,T)\times\Omega)}\\
&\quad\le \sqrt T\sup_{t\in[0,T]}\|\partial_t(u_\ell-v_\ell)(t)\|_{L^2(\Omega)} +\|\partial_tv_\ell-S_{\lambda,w,\ell}\|_{L^2((0,T)\times\Omega)}\\
&\quad\le C_T\|(\partial_t^2-\Delta)v_\ell\|_{L^1(0,T;L^2(\Omega))} +\|\partial_tv_\ell-S_{\lambda,w,\ell}\|_{L^2((0,T)\times\Omega)} \rightarrow0.
\end{align*}
This proves \eqref{eq:finite-exact} and finishes the proof.
\end{proof}

\Needspace{12\baselineskip}
\subsection{Realizing the full profile family}\label{subsec:curved-full}

The finite profiles $f_*^{(\ell)},g_*^{(\ell)}$ converge to $f_*,g_*$ in $L^2(\R/\Z)$, but do not retain their exact plateau regions. We use them to construct waves while keeping the full-profile function \eqref{eq:curved-matching}. With the common normalization factors $c_k$ chosen below, the realization proceeds through
\begin{equation*}
S_k \underset{\substack{\text{profile}\\\text{tail}}}{\approx} c_kS_{\lambda_k,w_k,\ell_k} \underset{\substack{\text{Airy expansion}\\\text{and phase comparison}}}{\approx} c_k\partial_tv_{\ell_k} \underset{\substack{\text{energy}\\\text{estimate}}}{\approx} \partial_tu_k.
\end{equation*}
All three comparisons are in spacetime $L^2$. Section~\ref{subsec:curved-correction} controls the last two for a fixed Fourier cutoff. The remaining error is the propagated profile tail; the same $c_k$ multiplies the finite and full profiles. Three estimates complete the realization.
\begin{itemize}[leftmargin=*,itemsep=3pt]
\item Energy bounds for $v_\ell$ must remain uniform when the Fourier cutoff increases.
\item The operators $\calT_\varepsilon$ must control $f_*^{(\ell)}-f_*$ and $g_*^{(\ell)}-g_*$ uniformly in the propagation scale.
\item A diagonal choice of the Fourier cutoff and scale must make all errors small before energy normalization.
\end{itemize}
\Needspace{8\baselineskip}
\begin{theorem}\label{thm:realization}
Let $\Omega\subset\R^2$ be a smooth bounded domain. Fix the local geometry of Lemma~\ref{lem:local-geometry}, with collar coordinates $x(s,y)$, interval $\mathcal I_{\rm in}$, depth $d_0$, time bound $T_{\rm loc}$ and fixed neighborhood $V$. Fix $0<T<T_{\rm loc}$, the diamond profiles $f_*,g_*$ from Proposition~\ref{prop:diamond}, and a real nonzero envelope $\chi\in C_c^\infty((-1/2,1/2))$.

Prescribe widths $w_k$ and centers $z_k\in\R$ with
\begin{equation*}
0<w_k<1,\quad w_k\to0 ,\text{ \; and \; }z_k+[-w_k/2,w_k/2]\subset\mathcal I_{\rm in}.
\end{equation*}
Fix common phase shifts $\alpha_k\in\R/\Z$ and any positive upper bounds $\overline\lambda_k$. Realization does not require the moving packet regions to be disjoint.

\smallskip
One can choose Fourier cutoffs $\ell_k\to\infty$, then sufficiently small scales satisfying
\begin{equation*}
0<\lambda_k\le\overline\lambda_k,\quad \lambda_k\to0 ,\text{ \; and \; }\lambda_k^2/w_k\to0.
\end{equation*}
Determine $h_k$ by \eqref{eq:scales}, form the corresponding approximate waves $v_{\ell_k}$, and set $c_k:=E[v_{\ell_k}](0)^{-1/2}$. There are constants $0<c_*\le c_k\le C_*$ such that the real solutions $u_k$ of \eqref{eq:wave} with $u_k[0]=c_kv_{\ell_k}[0]$ have the following properties.

\begin{enumerate}[label=(\roman*),leftmargin=*,itemsep=1.5mm]
\item \emph{(Asymptotic realization).} We have
\begin{equation}\label{eq:full-realization}
\begin{gathered}
\partial_tu_k=S_k+r_k,\\
\text{where }S_k:=c_kS_{\lambda_k,w_k,z_k,\alpha_k}=A_{\lambda_k}\left(f_*(\phi_k)-g_*(\psi_k)\right),\\
E[u_k](0)=1 ,\text{ \; and \; }\|r_k\|_{L^2((0,T)\times\Omega)}\to0.
\end{gathered}
\end{equation}
Here $A_{\lambda_k}=c_kA_{\lambda_k,w_k}$, and $S_k,\Gamma_k$ use the full profiles and local coordinates of Proposition~\ref{prop:curved-support}.

All approximate waves are supported in $V$ and extended by zero to $\Omega$; the exact waves solve the equation on the whole original domain. In particular,
\begin{equation*}
\1_{\Gamma_k}S_k=0,\text{ \; and \; } \liminf_{k\to\infty}\|S_k\|_{L^2((0,T)\times\Omega)}>0.
\end{equation*}
\item \emph{(Observation on the fixed set and non-observability).} If the intervals $I_k=z_k+(-w_k/2,w_k/2)$ are pairwise disjoint, take any $\gamma_k\to0^+$ and use these same scales in Proposition~\ref{prop:curved-observation} to define $G^{(N)}\downarrow G$. Then
\begin{equation}\label{eq:curved-wave-obs}
\1_{G^{(N)}}\partial_tu_k =\1_{G^{(N)}}r_k+\1_{\calE_k}S_k, \quad \forall k\in\{1,\ldots,N\},
\end{equation}
and the same identity holds on $G$. Moreover,
\begin{equation*}
\|\1_{\calE_k}S_k\|_{L^2((0,T)\times\Omega)}^2 \le C_T(\gamma_k+h_k/w_k)\rightarrow0.
\end{equation*}
Consequently,
\begin{equation*}
\|\1_G\partial_tu_k\|_{L^2((0,T)\times\Omega)}^2 \le2\|r_k\|_{L^2((0,T)\times\Omega)}^2+C_T(\gamma_k+h_k/w_k) \rightarrow0,
\end{equation*}
so observability fails on $G$. Here $C_T$ is independent of $k$ and the reference widths $\gamma_k$.
\end{enumerate}
\end{theorem}

\subsubsection{Energy bounds uniform in the Fourier cutoff}\label{subsubsec:curved-energy}

We separate the individual energies from their interactions. Let $V_n$ denote the real $n$-th summand of \eqref{eq:airy-wave}, including its coefficient and cutoffs. For a measurable set $D\subset\Omega$, put
\begin{equation*}
I_{nm}(D):=\int_D\left(\partial_tV_n\partial_tV_m+\nabla V_n\cdot\nabla V_m\right)(0,x)\dd x.
\end{equation*}
Then
\begin{equation*}
v_\ell=\sum_nV_n ,\text{ \; and \; } E[v_\ell](0)=\sum_n E[V_n](0)+\sum_{n<m}I_{nm}(\Omega).
\end{equation*}
All sums here run over the positive frequencies $n\le\ell$ with $n\equiv1\pmod4$. 

For one mode, \eqref{eq:airy-norms} gives a normal integral of order $h_n^{1/3}\lambda_n$. The squared leading time-derivative factor is $h_n^{-1/3}$, so the normalization gives
\begin{equation*}
|d_{*,n}|^2b_{\lambda,w}^2\,w\,h_n^{-1/3}\left(h_n^{1/3}\lambda_n\right) =|d_{*,n}|^2\frac{\lambda_n}{\lambda}\rightarrow |d_{*,n}|^2.
\end{equation*}
Together with the spatial-derivative bounds in Step~2 of the proof of Proposition~\ref{prop:finite-airy}, this yields, for each fixed $\ell$ and sufficiently small $\lambda$,
\begin{equation*}
c_0|d_{*,n}|^2\le E[V_n](0)\le C_0|d_{*,n}|^2.
\end{equation*}
The lower bound follows by averaging the oscillations on a regular subcollar; the upper bound includes the turning layer. The constants are independent of $\ell$, although how small $\lambda$ must be may depend on it.

To identify the limiting energy, we first compute the normal mass of the squared common amplitude. On $0\le y\le y_c^\lambda(s)$, the phase formulas give
\begin{gather*}
\zeta_\lambda(s,0)=2^{-2/3}\lambda^2,\quad \partial_y\zeta_\lambda=-\left(2\kappa(s)\right)^{1/3}\left(1+O(\lambda^2)\right), \text{ \; and \; }1-\kappa(s)y=1+O(\lambda^2).
\end{gather*}
Changing variables to $q=\zeta_\lambda(s,y)$, we obtain, uniformly in $s$,
\begin{align*}
\frac1\lambda\int_0^{y_c^\lambda(s)}\kappa(s)^{1/3}(\zeta_\lambda(s,y))^{-1/2}\left(1-\kappa(s)y\right)\dd y&=\frac{2^{-1/3}}\lambda\left(1+O(\lambda^2)\right)\int_0^{2^{-2/3}\lambda^2}q^{-1/2}\dd q\\
&=2^{1/3}+O(\lambda^2).
\end{align*}
The curvature factors cancel in this integral. The same calculation at $\lambda_n$, with normalization $1/\lambda$, has the additional factor $\lambda_n/\lambda\to1$.

Each real branch square has mean $|d_{*,n}|^2/2$, so the two branches contribute $|d_{*,n}|^2$ once their oscillatory terms have been removed. The leading spatial and time-derivative energies agree: the tangential derivative is the negative time derivative up to an $o(1)$ error in $L^2(\Omega)$, and the normal derivative is $o(1)$ there. We now justify the averaging, including the interactions between the branches.

\vspace{2mm} To estimate the interactions, write the two phases at $t=0$ as $\Psi_n^\pm(s,y):=\frac{\Phi_{0,\lambda_n}(s,y)\pm R_{\lambda_n}(s,y)}{h_n}$. Fix $\delta>0$ sufficiently small that $y_c^\lambda(s)/\lambda^2-\delta$ stays uniformly positive. On the regular band $0\le y\le y_c^\lambda(s)-\delta\lambda^2$, for $n\ne m$ and $\sigma,\tau\in\{+,-\}$, we have
\begin{equation*}
\partial_s(\Psi_n^\sigma-\Psi_m^\tau)=\frac{n-m}{h}+O_{\ell,\delta}(\lambda^2/h),\; \partial_s(\Psi_n^\sigma+\Psi_m^\tau)=\frac{n+m}{h}+O_{\ell,\delta}(\lambda^2/h).
\end{equation*}
For the two branches of one mode, with $Y=y/\lambda^2$, the normal separation satisfies
\begin{gather*}
\Psi_n^+-\Psi_n^-=\frac{2R_{\lambda_n}}{h_n},\text{ \; and \; }\left|\partial_Y\left((\Psi_n^+-\Psi_n^-)(s,\lambda^2Y)\right)\right| \sim_{\ell,\delta}\frac{\lambda^3}{h}.
\end{gather*}
Both separations increase on the respective tangential and normal localization scales, since
\begin{equation*}
\frac hw=\frac h{\lambda^2}\frac{\lambda^2}{w}\rightarrow0 ,\text{ \; and \; }\frac h{\lambda^3}\sim\lambda^{3/2}\rightarrow0.
\end{equation*}
The upper boundary of the regular band depends on $s$. To integrate over a fixed domain, set
\begin{equation*}
Y_c^\lambda(s):=y_c^\lambda(s)/\lambda^2 ,\text{ \; and \; }y=\lambda^2\eta\left(Y_c^\lambda(s)-\delta\right),\quad 0\le\eta\le1.
\end{equation*}
At fixed $\eta$, the extra tangential phase derivative is
\begin{equation*}
\lambda^2\eta\,\partial_sY_c^\lambda\,\partial_y(\Psi_n^\sigma\pm\Psi_m^\tau) =O_{\ell,\delta}(\lambda^3/h).
\end{equation*}
It is smaller than the $(n\pm m)/h$ term for distinct modes. After also setting $s=z+w\xi$, integration by parts bounds the leading oscillatory integrals by
\begin{equation*}
C_{\ell,\delta}(h/w+h)\quad\text{for distinct modes}, \quad C_{\ell,\delta}h/\lambda^3\quad\text{for the two branches of one mode}.
\end{equation*}
The doubled phases in each real branch square are controlled by the same tangential averaging. These bounds concern the leading oscillatory integrals. The Airy approximation and differentiated-envelope errors also tend to zero for fixed $\ell,\delta$, by the estimates in the proof of Proposition~\ref{prop:finite-airy}.

Fix $\ell$ and recall the turning strip $B_{\lambda,\delta}$ from \eqref{eq:turning-energy}. Split
\begin{equation*}
I_{nm}(\Omega)=I_{nm}(\Omega\setminus B_{\lambda,\delta})+I_{nm}(B_{\lambda,\delta}).
\end{equation*}
Nonstationary phase on the oscillatory side and Airy decay on the other side give
\begin{equation*}
\sum_{n<m}|I_{nm}(\Omega\setminus B_{\lambda,\delta})|\rightarrow0, \quad \lambda\to0^+\quad\text{ \; with \; }\ell,\delta\text{ fixed}.
\end{equation*}
Inside the strip, the Cauchy--Schwarz inequality gives
\begin{equation*}
|I_{nm}(B_{\lambda,\delta})| \le\left(\int_{B_{\lambda,\delta}}\left(|\partial_tV_n|^2+|\nabla V_n|^2\right)(0,x)\dd x\right)^{1/2}\left(\int_{B_{\lambda,\delta}}\left(|\partial_tV_m|^2+|\nabla V_m|^2\right)(0,x)\dd x\right)^{1/2}.
\end{equation*}
Applying the fixed-time turning-layer energy estimate \eqref{eq:turning-energy} to the individual modes, we obtain
\begin{equation*}
\limsup_{\lambda\to0^+}\sum_{n<m}|I_{nm}(\Omega)| \le C_\ell\sqrt\delta.
\end{equation*}
Letting $\delta\to0^+$ proves that the cross terms vanish.

The same turning-layer estimate completes the averaging within each mode. Combining it with the normal-mass calculation and $w^{-1}\int\chi\left((s-z)/w\right)^2\dd s=\|\chi\|_{L^2(\R)}^2$, we obtain, for every fixed $\ell$,
\begin{equation*}
E[v_\ell](0)\rightarrow 2^{1/3}\|\chi\|_{L^2(\R)}^2\sum_{\substack{n\le\ell\\n\equiv1\ (4)}}|d_{*,n}|^2, \quad \lambda\to0^+.
\end{equation*}
Here we use the scales \eqref{eq:scales} and $\lambda^2/w\to0$; $\ell,\delta$ are fixed before $\lambda\to0$, and then $\delta\to0$.

\vspace{2mm} We may therefore choose $\lambda$ sufficiently small that
\begin{equation*}
\sum_{n<m}|I_{nm}(\Omega)|\le\frac{c_0}{2}|d_{*,1}|^2.
\end{equation*}
Combining this with the individual energy bounds, we obtain
\begin{equation*}
\frac{c_0}{2}|d_{*,1}|^2\le E[v_\ell](0) \le C_0\sum_{n\le\ell,\ n\equiv1(4)}|d_{*,n}|^2+\frac{c_0}{2}|d_{*,1}|^2.
\end{equation*}
Since $d_{*,1}\ne0$ and
\begin{equation*}
0<|d_{*,1}|^2\le\sum_{n\le\ell,\ n\equiv1(4)}|d_{*,n}|^2 \le\sum_{n\ge1}|d_{*,n}|^2<\infty,
\end{equation*}
the bounds are uniform in the Fourier cutoff. Thus, along the diagonal,
\begin{equation}\label{eq:uniform-energy}
0<c_*\le E[v_{\ell_k}](0)\le C_*<\infty
\end{equation}
with constants independent of $k$.

\vspace{2mm} To obtain the spacetime lower bound, choose fixed $Y_-,Y_+$ with
\begin{equation*}
0<Y_-<Y_+<\inf_{s\in\mathcal I}\frac1{2\kappa(s)^{1/3}}
\end{equation*}
and restrict to the regular subcollar
\begin{equation*}
\mathcal C_\lambda:=\{x(s,y):s\in\mathcal I,\quad Y_-<y/\lambda^2<Y_+\}.
\end{equation*}
There $\zeta_\lambda\sim\lambda^2$ and $\mu\sim1$. By \eqref{eq:actual-phases}, the common amplitude satisfies
\begin{equation*}
\begin{aligned}
\int_{\mathcal C_\lambda}|A_{\lambda,w}(t,x)|^2\dd x &\sim\frac1{\lambda w}\int\chi\left((s-t-z)/w\right)^2\dd s \int_{\lambda^2Y_-}^{\lambda^2Y_+}\lambda^{-1}\dd y\\
&\sim\frac1{\lambda w}\,w\lambda=1, \quad\forall t\in[0,T].
\end{aligned}
\end{equation*}
Tangential oscillation averages each real branch square and removes interactions between distinct modes. Normal oscillation removes the interaction between the two branches of one mode. Thus, for each fixed $\ell$,
\begin{equation*}
\int_0^T\!\int_{\mathcal C_\lambda}|S_{\lambda,w,\ell}|^2\dd x\dd t =\left(\sum_n|d_{*,n}|^2\right)\int_0^T\!\int_{\mathcal C_\lambda}|A_{\lambda,w}|^2\dd x\dd t+o_\ell(1).
\end{equation*}
Consequently,
\begin{equation*}
\liminf_{\lambda\to0^+}\|S_{\lambda,w,\ell}\|_{L^2((0,T)\times\Omega)}^2 \ge c_T\sum_{\substack{n\le\ell\\n\equiv1\ (4)}}|d_{*,n}|^2 \ge c_T|d_{*,1}|^2>0.
\end{equation*}
Here $c_T>0$ is independent of $\ell$. Taking $\lambda$ sufficiently small for each fixed cutoff makes the squared norm at least $c_T|d_{*,1}|^2/2$. We include this requirement in the diagonal choice below.

\subsubsection{Controlling the profile tail after propagation}\label{subsubsec:curved-tail}

To compare finite and full profiles, we need an estimate uniform in the scale. At a fixed spatial point in the chart, either phase has time slope $-(2\pi h)^{-1}$. Period counting against the smooth moving envelope gives, for any $F\in L^2(\R/\Z)$,
\begin{align*}
\int\chi\left((s-t-z)/w\right)^2|F(\xi_0-\frac{t}{2\pi h})|^2\dd t &\le2\pi h\|\chi\|_{L^{\infty}}^2 \left(\lfloor\frac{w}{2\pi h}\rfloor+1\right)\|F\|_{L^2(\R/\Z)}^2\\
&\le C(w+h)\|F\|_{L^2(\R/\Z)}^2.
\end{align*}
Here the envelope restricts time to an interval of length $w$, and each phase period has duration $2\pi h$.

The squared normal amplitude has integral $O(\lambda)$, so the normalized squared norm is bounded by
\begin{equation*}
C_T\frac{\lambda(w+h)}{\lambda w}\|F\|_{L^2(\R/\Z)}^2 =C_T(1+h/w)\|F\|_{L^2(\R/\Z)}^2.
\end{equation*}
Hence
\begin{equation}\label{eq:pullback}
\|A_{\lambda,w}F(\phi)\|_{L^2((0,T)\times\Omega)} +\|A_{\lambda,w}F(\psi)\|_{L^2((0,T)\times\Omega)} \le C_T\|F\|_{L^2(\R/\Z)}
\end{equation}
when $h/w\le1$. Outside the fixed chart the amplitude is zero.

\vspace{2mm} For the propagated tail, we subtract the two expressions with the same amplitude and phases:
\begin{equation*}
S_{\lambda,w,\ell}-S_{\lambda,w,z,\alpha} =A_{\lambda,w}\left((f_*^{(\ell)}-f_*)(\phi)-(g_*^{(\ell)}-g_*)(\psi)\right).
\end{equation*}
Equation \eqref{eq:pullback} then gives
\begin{align*}
\|S_{\lambda,w,\ell}-S_{\lambda,w,z,\alpha}\|_{L^2((0,T)\times\Omega)} &\le\|A_{\lambda,w}(f_*^{(\ell)}-f_*)(\phi)\|_{L^2((0,T)\times\Omega)}+\|A_{\lambda,w}(g_*^{(\ell)}-g_*)(\psi)\|_{L^2((0,T)\times\Omega)}\\
&\le C_T\left(\|f_*^{(\ell)}-f_*\|_{L^2(\R/\Z)}+\|g_*^{(\ell)}-g_*\|_{L^2(\R/\Z)}\right).
\end{align*}
This bounds the propagated tail uniformly in the scale, as required by the bound for $\calT_\varepsilon$ in \cite[Section~2]{AOSI}.

\subsubsection{The diagonal choice and normalization}\label{subsubsec:curved-diagonal}

\begin{proof}[Proof of Theorem~\ref{thm:realization}]
Choose $\ell_k\to\infty$ first. For each fixed $\ell_k$, take $\lambda_k$ sufficiently small that
\begin{equation*}
0<\lambda_k\le\min\{\overline\lambda_k,2^{-k}\} ,\text{ \; and \; }\lambda_k^2/w_k\le2^{-k}.
\end{equation*}
Impose the energy bounds \eqref{eq:uniform-energy} and the spacetime lower bound from Section~\ref{subsubsec:curved-energy}:
\begin{gather*}
c_*\le E[v_{\ell_k}](0)\le C_*,\; \|S_{\lambda_k,w_k,\ell_k}\|_{L^2((0,T)\times\Omega)}^2\ge\tfrac12c_T|d_{*,1}|^2.
\end{gather*}
The fixed-cutoff energy calculation in Section~\ref{subsubsec:curved-energy} permits the first requirement after decreasing $\lambda_k$: the constants $c_*,C_*$ are independent of $\ell_k$, while the required scale threshold may depend on it. For the exact wave $u_{\ell_k}$ with initial state $v_{\ell_k}[0]$, require
\begin{align*}
\|\partial_t(u_{\ell_k}-v_{\ell_k})\|_{L^2((0,T)\times\Omega)}\le2^{-k},\;\|\partial_tv_{\ell_k}-S_{\lambda_k,w_k,\ell_k}\|_{L^2((0,T)\times\Omega)}\le2^{-k}.
\end{align*}
Also require that every retained mode's normal cutoff lie below $d_0$. The fixed-cutoff estimates of Sections~\ref{subsec:curved-correction} and \ref{subsubsec:curved-energy} allow all these requirements after decreasing $\lambda_k$ further.

The moving envelopes remain inside the chart by \eqref{eq:local-buffer}, independently of $k$, so the zero extensions introduce no further error.
We normalize the energy by setting
\begin{equation*}
c_k:=E[v_{\ell_k}](0)^{-1/2},\quad u_k[0]:=c_kv_{\ell_k}[0],\text{ \; and \; }A_{\lambda_k}:=c_kA_{\lambda_k,w_k},
\end{equation*}
and let $u_k=c_ku_{\ell_k}$, which solves \eqref{eq:wave}. Then
\begin{equation*}
E[u_k](0)=c_k^2E[v_{\ell_k}](0)=1 ,\text{ \; and \; }C_*^{-1/2}\le c_k\le c_*^{-1/2},
\end{equation*}
where the bounds follow from \eqref{eq:uniform-energy}. The factors $c_k$ are fixed after choosing the diamond and are used unchanged when applying the bound for $\calT_\varepsilon$ to a profile tail.

\vspace{2mm} The three sources of error are now explicit:
\begin{equation*}
r_k=\partial_t(u_k-c_kv_{\ell_k})+c_k(\partial_tv_{\ell_k}-S_{\lambda_k,w_k,\ell_k})+c_k(S_{\lambda_k,w_k,\ell_k}-S_{\lambda_k,w_k,z_k,\alpha_k}).
\end{equation*}
These are, in reverse order to the introductory comparison, the energy correction \eqref{eq:energy-correction}, the Airy and common-phase error \eqref{eq:finite-S-error}, and the propagated profile tail controlled in Section~\ref{subsubsec:curved-tail}. Hence
\begin{equation*}
\|r_k\|_{L^2((0,T)\times\Omega)}\le C_T\left(2^{-k}+\|f_*^{(\ell_k)}-f_*\|_{L^2(\R/\Z)}+\|g_*^{(\ell_k)}-g_*\|_{L^2(\R/\Z)}\right)\rightarrow0.
\end{equation*}
This proves \eqref{eq:full-realization}. Proposition~\ref{prop:curved-support}(ii) gives the matching equality.

For nonvanishing, the spacetime lower bound from Section~\ref{subsubsec:curved-energy} and the tail estimate from Section~\ref{subsubsec:curved-tail} give
\begin{align*}
\|S_k\|_{L^2((0,T)\times\Omega)}&\ge c_k\left(\|S_{\lambda_k,w_k,\ell_k}\|_{L^2((0,T)\times\Omega)}-\|S_{\lambda_k,w_k,\ell_k}-S_{\lambda_k,w_k,z_k,\alpha_k}\|_{L^2((0,T)\times\Omega)}\right)\\
&\ge c_k\left(\sqrt{c_T/2}\,|d_{*,1}|-o(1)\right).
\end{align*}
Since $c_k\ge C_*^{-1/2}$, we obtain $\liminf_{k\to\infty}\|S_k\|_{L^2((0,T)\times\Omega)} \ge C_*^{-1/2}\sqrt{c_T/2}\,|d_{*,1}|>0$. For disjoint intervals, Proposition~\ref{prop:curved-observation} gives $\1_{G^{(N)}}S_k=\1_GS_k=\1_{\calE_k}S_k$ for $k\le N$. Multiplying the realization formula by the corresponding indicators proves \eqref{eq:curved-wave-obs} and its fixed-set version.

\vspace{2mm} \noindent\emph{Additional observed energy.} We now estimate the remaining term for any $\gamma_k\to0^+$. Since $|f_*|,|g_*|\le1$ and the normalization factors $c_k$ are uniformly bounded, it suffices to integrate the squared common amplitude over $\calE_k$.

At a fixed spatial point, each phase in \eqref{eq:actual-phases} is affine in time with slope $-(2\pi h_k)^{-1}$. The moving envelope is supported in a time interval of length $w_k$. Counting complete phase periods and at most two partial periods gives
\begin{equation*}
\int_0^T\chi\left((s-t-z_k)/w_k\right)^2\1_{\{\dist(\phi_k,\mathcal K)\le4\gamma_k\}}\dd t \le C(\gamma_kw_k+h_k).
\end{equation*}
The same bound holds for $\psi_k$, uniformly in the common phase shift and in $\gamma_k$.

Recall that $\dd x=\mu(s,y)\dd s\dd y$. The squared normal amplitude satisfies
\begin{equation*}
\int_0^{y_c^{\lambda_k}(s)} \kappa(s)^{1/3}(\zeta_{\lambda_k}(s,y))^{-1/2}\mu(s,y)\dd y \le C\lambda_k.
\end{equation*}
By \eqref{eq:turning}, $\zeta_{\lambda_k}(s,y)\sim y_c^{\lambda_k}(s)-y$ on the moving packet region, and $y_c^{\lambda_k}(s)=O(\lambda_k^2)$. Integrating over the fixed arclength chart and using the normalization $(\lambda_kw_k)^{-1}$, we obtain
\begin{equation}\label{eq:extra-energy}
\|\1_GS_k\|_{L^2((0,T)\times\Omega)}^2=\|\1_{\calE_k}S_k\|_{L^2((0,T)\times\Omega)}^2\le\frac{C_T}{\lambda_kw_k}\,\lambda_k(\gamma_kw_k+h_k)=C_T(\gamma_k+h_k/w_k)\rightarrow0.
\end{equation}
The limit follows from $\gamma_k\to0$ and $\frac{h_k}{w_k}=\frac{h_k}{\lambda_k^2}\frac{\lambda_k^2}{w_k}\rightarrow0,\text{ since }h_k\sim\lambda_k^{9/2}$. Combining the observation identity \eqref{eq:curved-wave-obs} with the additional observed-energy estimate \eqref{eq:extra-energy}, we conclude that
\begin{align*}
\|\1_G\partial_tu_k\|_{L^2((0,T)\times\Omega)}^2 &\le2\|r_k\|_{L^2((0,T)\times\Omega)}^2 +2\|\1_{\calE_k}S_k\|_{L^2((0,T)\times\Omega)}^2\\
&\le2\|r_k\|_{L^2((0,T)\times\Omega)}^2 +C_T(\gamma_k+h_k/w_k)\rightarrow0.
\end{align*}
Since $E[u_k](0)=1$, observability fails on $G$.

\vspace{2mm} \noindent\emph{Choice of parameters.} For the prescribed widths, centers, common phase shifts and scale bounds, the dependencies are
\begin{equation*}
(w_k,z_k,\alpha_k,\overline\lambda_k) \rightarrow\ell_k\rightarrow\lambda_k \rightarrow(h_k,v_{\ell_k})\rightarrow(c_k,u_k,S_k).
\end{equation*}
The cutoff $\ell_k$ controls the profile tail. Once it is fixed, $\lambda_k$ can be taken sufficiently small for all finite-profile errors and support requirements. Their constants may depend on $\ell_k$, whereas the energy and propagated-tail bounds are uniform. We point out that \eqref{eq:scales} determines $h_k$ from $\lambda_k$, and $c_k$ is fixed after constructing $v_{\ell_k}$.

In particular, $h_k\sim\lambda_k^{9/2}$ and $\lambda_k^2/w_k\to0$ coordinate the oscillation, normal depth and tangential width. Any further upper bounds needed in Section~\ref{sec:gcc} are imposed when selecting $\lambda_k$, before fixing the common observation region $G$. This finishes the proof of the theorem.
\end{proof}

In the notation of \cite[Section~2]{AOSI}, \cite[equation~(2.7)]{AOSI} holds with $w_1=1$, $w_2=-1$, $\Theta_{\lambda_k}=(\phi_k,\psi_k)$, and common real amplitude $A_{\lambda_k}$.

\section{GCC realization on smooth bounded domains}\label{sec:gcc}

Theorem~\ref{thm:realization} gives normalized exact waves with asymptotically vanishing observation on the fixed set $G$. We now make compatible parameter choices in \eqref{eq:curved-G-param} so that this same $G$ also satisfies \hyperref[eq:gcc]{GCC}.

There are two difficulties. At one scale, a ray may be much shallower than the moving packet region or nearly follow one of the two wave branches. Across scales, combining the hidden regions may destroy the time bound. We resolve the first difficulty by profile matching and the quarter-turn relation, and the second by strengthening the separation of moving packet regions. The following theorem records the complete construction and implies the \hyperref[thm:main]{Main Theorem}.

\begin{theorem}\label{thm:curved-gcc}
Let $\Omega$ be a smooth bounded planar domain and fix the local geometry of Lemma~\ref{lem:local-geometry}. Fix $0<T<T_{\rm loc}$, the diamond profiles $f_*,g_*$ and rectangles $\mathcal R$ of Proposition~\ref{prop:diamond}, and a real nonzero envelope $\chi\in C_c^\infty((-1/2,1/2))$. For any prescribed common phase shifts $\alpha_k\in\R/\Z$, there are constants $R,\lambda_*,\eta>0$, depending only on this geometry, $T$ and $\chi$, and coordinated choices with the following properties.

\begin{enumerate}[label=(\roman*),leftmargin=*]
\item \emph{(Geometry and scales).} There are centers $z_k\in\R$ and widths $0<w_k<1$, with $w_k\to0$, such that the intervals $I_k=z_k+(-w_k/2,w_k/2)$ have pairwise disjoint concentric triple enlargements contained in $\mathcal I_{\rm in}$ and
\begin{equation*}
\sum_k w_k\le T/2.
\end{equation*}

One can then choose Fourier cutoffs $\ell_k\to\infty$ and scales $\lambda_k\to0^+$ with
\begin{gather*}
\lambda_k\le\lambda_*,\quad w_k\ge R\lambda_k^2,\text{ \; and \; } \lambda_k^2/w_k\rightarrow0,\\
j_k:=\lceil\lambda_k^{-3/2}\rceil,\quad h_k:=\frac{\lambda_k^3}{3\pi(j_k-1/4)},\quad \text{and }\gamma_k:=\lambda_k^{1/4}.
\end{gather*}
The scales can be taken small enough for the realization in Theorem~\ref{thm:realization}, with normalization factors $0<c_*\le c_k\le C_*$. Take $\mathcal U_k:=\mathcal U_{\lambda_k,w_k,z_k}$ from Definition~\ref{def:curved-packet}. Define $S_k,\Gamma_k$ by Proposition~\ref{prop:curved-support}, with $S_k=c_kS_{\lambda_k,w_k,z_k,\alpha_k}$, and $\calE_k$ by Proposition~\ref{prop:curved-observation}, using these same parameters and local coordinates. All moving packet regions lie in $(0,T)\times V$. Thus one fixed Borel set is given by
\begin{equation*}
\begin{gathered}
\mathcal H_k:=\mathcal U_k\setminus(\Gamma_k\cup\calE_k),\quad G^{(N)}:=\left((0,T)\times\Omega\right)\setminus\bigcup_{k\le N}\mathcal H_k,\\
G:=\bigcap_{N\ge1}G^{(N)},\quad \1_{G^{(N)}}S_k=\1_GS_k=\1_{\calE_k}S_k\quad\forall k\in\{1,\ldots,N\}.
\end{gathered}
\end{equation*}

\item \emph{(Asymptotic realization and non-observability).} By Theorem~\ref{thm:realization}, the same choices give real solutions $u_k$ of \eqref{eq:wave} with
\begin{equation*}
E[u_k](0)=1,\quad \partial_tu_k=S_k+r_k,\text{ \; and \; } \|r_k\|_{L^2((0,T)\times\Omega)}\rightarrow0.
\end{equation*}
Their observation satisfies
\begin{equation}\label{eq:curved-nonobs}
\begin{aligned}
\|\1_G\partial_tu_k\|_{L^2((0,T)\times\Omega)}^2 &\le2\|r_k\|_{L^2((0,T)\times\Omega)}^2 +2\|\1_{\calE_k}S_k\|_{L^2((0,T)\times\Omega)}^2\\
&\le2\|r_k\|_{L^2((0,T)\times\Omega)}^2 +C_T(\gamma_k+h_k/w_k)\rightarrow0.
\end{aligned}
\end{equation}
Consequently, observability fails on $G$.

\item \emph{(Essential GCC and boundary completion).} For this same set,
\begin{equation}\label{eq:curved-gcc}
\essinf_{\rho\in S^*\Omega}\int_0^T\1_G(t,\pi\Phi_t\rho)\dd t\ge\eta>0,
\end{equation}
where the essential infimum is relative to Liouville measure. Moreover, for the specified Borel set $G$, define
\begin{equation*}
G^\sharp:=G\cup\left((0,T)\times\partial\Omega\right).
\end{equation*}
Every unit-speed generalized ray $x_\gamma:[0,T]\to\overline\Omega$ satisfies
\begin{equation*}
\int_0^T\1_{G^\sharp}\left(t,x_\gamma(t)\right)\dd t\ge\eta.
\end{equation*}
Both lower bounds hold with the same constant for every $G^{(N)}\supset G$, using its boundary completion in the second bound. The wave observation remains on the interior set $G$.
\end{enumerate}
\end{theorem}

\setcounter{subsection}{-1}
\subsection{Sketch of the proof}\label{subsec:gcc-sketch}

We first estimate the observed time at one scale, then combine the scales, and finally use the exact waves on the same set.

\smallskip\noindent\emph{Step 1. A uniform one-scale estimate.} Section~\ref{subsec:gcc-local} defines $\mathcal O_{\lambda,\alpha}$ by observing outside the oscillatory collar, on the matching regions inside it, and on additional strips near their corners. We prove
\begin{equation*}
\int_0^T\1_{\mathcal O_{\lambda,\alpha}}\left(t,x(t)\right)\dd t\ge\eta_0T
\end{equation*}
for positive regular glancing rays meeting $V$, uniformly in sufficiently small values of the scale parameter $\lambda$ and the reference angle, and in the common phase shift.

Section~\ref{subsec:gcc-action} introduces the normalized action ratio
\begin{equation*}
r_j:=\frac{\beta_j}{\lambda\kappa(s_j)^{1/3}},
\end{equation*}
where $s_j,\beta_j$ are the impact point and inward normal component. Finite-time action stability preserves its limiting regime, independently of the number of reflections. Section~\ref{subsec:gcc-action} treats limits in $(0,1)$ by phase averaging and those in $(1,\infty]$ by observing central chord portions outside the collar.

For the limit $0$, Section~\ref{subsec:gcc-shallow} uses the boundary relation to count a continuous common phase over $[0,T]$. For the limit $1$, Section~\ref{subsec:gcc-matched} combines two successive chords using the quarter-period return and
\begin{equation*}
J_0\cup(J_0-1/4)=\R/\Z,
\end{equation*}
where $J_0$ is the union of the two $g_*$ plateau intervals. The additional strips absorb corner errors. Section~\ref{subsec:gcc-one-scale} combines these ranges by contradiction to obtain uniform constants.

\smallskip\noindent\emph{Step 2. Combining the scales.} In Section~\ref{subsec:gcc-fixed}, choose disjoint concentric triple enlargements of the intervals $I_k$ and require $w_k\ge R\lambda_k^2$. For a positive glancing ray with reference normal component $\beta$, the $O(\beta^2)$ drift of $s-t$ allows at most one encountered moving packet region with $\lambda_k\ge\delta\beta$, for a fixed small $\delta>0$ and sufficiently large $R$. All smaller moving packet regions lie in one collar of depth $D\delta^2\beta^2$.

If $\mathcal U_k$ is the unique moving packet region above the threshold met by the ray, then along the ray, we have
\begin{align*}
\1_G\ge\1_{\mathcal O_{\lambda_k,\alpha_k}}-\1_{\{y\le D\delta^2\beta^2\}},
\end{align*}
which implies that $D_G\ge\eta_0T-C\delta^2T-C\beta\ge\eta_0T/2$ by choosing $\delta$, then the glancing threshold, sufficiently small. If no larger moving packet region is met, the thin-collar estimate alone suffices.

Negative glancing rays have hidden time at most $\sum_kw_k\le T/2$. Non-glancing rays spend uniformly little time in a sufficiently thin boundary collar; rays avoiding $V$ are fully observed. These bounds apply to each regular ray. Section~\ref{subsubsec:gcc-completion} extends the bound to every generalized ray after adjoining $(0,T)\times\partial\Omega$: a ray meeting $V$ stays in the positive-curvature chart, where it is either regular or entirely on the boundary.

\smallskip\noindent\emph{Step 3. The exact waves on the same set.} Section~\ref{subsec:gcc-energy} coordinates the construction with Theorem~\ref{thm:realization}. The geometric requirements impose only upper bounds:
\begin{equation*}
\lambda_k\le\min\{\lambda_*,(w_k/R)^{1/2}\}.
\end{equation*}
After fixing each Fourier cutoff, we may shrink $\lambda_k$ further to ensure $\lambda_k^2/w_k\to0$ and vanishing velocity approximation error.

These same choices give a fixed $G$ satisfying essential \hyperref[eq:gcc]{GCC} and normalized exact waves with $\int_G|\partial_tu_k|^2\dd x\dd t\to0$. This proves Theorem~\ref{thm:curved-gcc}. For arbitrary observation times, we construct on a shorter admissible interval, continue the exact waves, and retain observation only on that interval, proving the \hyperref[thm:main]{Main Theorem}.

\subsection{The observation region at one scale}\label{subsec:gcc-local}
Recall the scales $\lambda,h$ from \eqref{eq:scales}, and we further introduce
\begin{equation}\label{eq:M-eps}
\gamma:=\lambda^{1/4},\; M:=(2\pi h)^{-1}\lambda^3\sim\lambda^{-3/2},\text{ \; and \; } \epsilon:=(2\pi h)^{-1}\lambda^6\sim\lambda^{3/2}=o(\gamma).
\end{equation}
Here $\gamma$ is the width in \eqref{eq:gl-extra}, $M$ measures the normal phase variation, and $\epsilon$ bounds the scaled return error. Since $\epsilon/\gamma\sim\lambda^{5/4}\to0$, the additional reference strips are wider than this return error near a matching corner. The quarter-period identity $-2(2\pi h)^{-1}\lambda^3/3=1/4\pmod1$ still holds exactly.

Recall the plateau intervals and matching rectangles from \eqref{eq:plateaus}--\eqref{eq:rectangles}:
\begin{gather*}
F_+=[3/4,1],\text{ \; and \; }F_-=[1/4,1/2],\; J_+=[0,1/4],\text{ \; and \; }J_-=[1/2,3/4],\\
\mathcal R=(F_+\times J_+)\cup(F_-\times J_-).
\end{gather*}
The intervals are understood modulo $1$. As in \eqref{eq:gl-extra}, write $\mathcal K=\{0,1/4,1/2,3/4\}$, and $J_0:=J_+\cup J_-$. Let $\mathcal C_\lambda:=\{x(s,y):s\in\mathcal I,\ 0<y<y_c^\lambda(s)\}$ be the local collar. %
On $(0,T)\times\mathcal C_\lambda$, the phases \eqref{eq:actual-phases} are
\begin{equation*}
\phi=(2\pi h)^{-1}(\Phi_{0,\lambda}-t+R_\lambda)+\alpha,\text{ \; and \; }\psi=(2\pi h)^{-1}(\Phi_{0,\lambda}-t-R_\lambda)+\alpha,
\end{equation*}
with a common phase shift $\alpha\in\R/\Z$. Define the one-scale observation region by
\begin{equation}\label{eq:local-obs}
\mathcal O_{\lambda,\alpha}:=\left((0,T)\times(\Omega\setminus\mathcal C_\lambda)\right)\cup\calE_\lambda\cup\{(t,x)\in(0,T)\times\mathcal C_\lambda:(\phi,\psi)\pmod1\in\mathcal R\},
\end{equation}
where the set $\calE_\lambda:=\{(t,x)\in(0,T)\times\mathcal C_\lambda:\dist(\phi,\mathcal K)\le4\gamma \ \text{or}\ \dist(\psi,\mathcal K)\le4\gamma\}$ is defined by the same phase-strip rule as in \eqref{eq:gl-extra}. Distances here are taken on $\R/\Z$. Profile matching gives zero observed energy on the rectangles. Moreover, the additional strips $\calE_\lambda$ accommodate the phase-comparison error near their corners. Thus the one-scale rule and observation inside each moving packet region satisfy
\begin{equation*}
\mathcal U_k\cap\mathcal O_{\lambda_k,\alpha_k}=\Gamma_k\cup\calE_k,\; \mathcal H_k=\mathcal U_k\setminus\mathcal O_{\lambda_k,\alpha_k}.
\end{equation*}
Theorem~\ref{thm:realization} gives $\|\1_{\calE_k}S_k\|_2\to0$, as required in \cite[equation~(2.12)]{AOSI}. Both the test rays and the physical packets use the same real arclength coordinate and the same primitive $C_\lambda$ (see Eq. \eqref{eq:W0} for its definition).

For the one-scale argument we consider positively oriented regular glancing rays that meet $V$ during $[0,T]$. Negative glancing rays and rays separated from glancing will be treated in Section~\ref{subsec:gcc-fixed}. Here \emph{regular glancing rays} means near-tangent rays with transverse reflections, not boundary-gliding trajectories; see Appendix~\ref{subsec:meas-boundary}.

Lemma~\ref{lem:local-geometry} confines them to $U$, where they agree with rays in the auxiliary strictly convex domain $\Omega_{\rm c}$. For sufficiently small glancing angles, the complete short ray segments containing time zero or time $T$, and the neighboring ray segments used below, also stay in this chart. All boundary-point uniformity below refers to this fixed arc.

\begin{prop}[Uniform one-scale estimate]\label{prop:one-scale}
Fix the local geometry and $0<T<T_{\rm loc}$. There are $\eta_0,\beta_*,\lambda_*>0$ such that every positively oriented regular glancing ray meeting $V$ during $[0,T]$, with reference inward normal component at most $\beta_*$, satisfies
\begin{equation}\label{eq:one-scale}
\int_0^T\1_{\mathcal O_{\lambda,\alpha}}\left(t,x(t)\right)\dd t\ge\eta_0T
\end{equation}
for all $\lambda\in(0,\lambda_*]$ chosen by \eqref{eq:scales} and all $\alpha\in\R/\Z$. The reference impact may precede time zero on the chord containing the initial point.
\end{prop}

\subsection{One-scale estimate 1: finite-time action stability and regular ranges}\label{subsec:gcc-action}

We first need to know whether a ray stays in the same scale regime during $[0,T]$. In the disc this follows from conserved angular momentum. For variable curvature, a normalized action is only approximately preserved, and its accumulated error must be controlled independently of the number of reflections.

Let $s_j,\beta_j$ be successive boundary points and inward normal components, and let $l_j$ be the intervening segment lengths. The frozen normal action in Section~\ref{sec:curved} was $\beta_j^3/\left(3\kappa(s_j)\right)$. It suggests the normalized action and the action ratio
\begin{equation*}
\delta_j:=\beta_j\kappa(s_j)^{-1/3} ,\text{ \; and \; }r_j:=\delta_j/\lambda.
\end{equation*}
The second quantity compares the test ray's normal scale with the scale parameter $\lambda$. The reference point with $j=0$ is the initial endpoint of the chord containing time zero.
\begin{lemma}\label{lem:action-stability}
Fix the local geometry and $0<T<T_{\rm loc}$. There are $\delta_*,C_T>0$ such that every positively oriented regular glancing ray meeting $V$ during $[0,T]$, with $0<\delta_0\le\delta_*$, satisfies
\begin{equation}\label{eq:action-stable}
\delta_{j+1}-\delta_j=O(\delta_j^2l_j), \quad |\delta_{j+1}^{-1}-\delta_j^{-1}|\le Cl_j, \quad \sup_j|\delta_j/\delta_0-1|\le C_T\delta_0.
\end{equation}
The indices range over the boundary-touch points bounding the ray segments that meet $[0,T]$. The constants depend only on the fixed geometry and time.%
\end{lemma}

The relevant geometries are
\begin{equation*}
r_j\ll1:\ \text{shallower ray},\quad r_j\simeq1:\ \text{nearly matched ray},\quad r_j>1:\ \text{deeper ray}.
\end{equation*}
Figure~\ref{fig:gcc-rays} shows why different estimates are needed. As the reference normal component and the scale parameter $\lambda$ tend to zero, pass to a subsequence with limiting initial ratio $r_*\in[0,\infty]$. Lemma~\ref{lem:action-stability} preserves this limiting regime throughout $[0,T]$.

Table~\ref{tab:gcc-ranges} records the corresponding estimates; Section~\ref{subsec:gcc-one-scale} will combine them into a uniform bound. The values $0$ and $\infty$ describe limiting regimes. The limits $0$ and $1$ require the separate arguments of Sections~\ref{subsec:gcc-shallow} and~\ref{subsec:gcc-matched}.

\begin{table}[H]
\centering
\caption{Limiting action ratios and the one-scale observation estimates.}\label{tab:gcc-ranges}
\begin{tabular}{@{}lll@{}}
\toprule
Limit $r_*$ & Source of a positive observed fraction & Location\\
\midrule
$0$ & Boundary relation & Lemma~\ref{lem:shallow}\\
$(0,1)$ & Regular-range averaging & \eqref{eq:compact-average}\\
$1$ & Two-segment matching & Lemma~\ref{lem:matched}\\
$(1,\infty]$ & Fully observed central chord portions & Section~\ref{subsec:gcc-action}\\
\bottomrule
\end{tabular}
\end{table}

\begin{proof}[Proof of Lemma~\ref{lem:action-stability}]

At a boundary point $s$, let $\beta>0$ be the inward normal component of a positively oriented ray. Denote the next boundary-touch point by $s^+(s,\beta)$, the next normal component by $B(s,\beta)$, and the segment length by $l(s,\beta)$. In boundary normal coordinates, the boundary points are denoted by $x(s,0)$ and $x(s^+,0)$. Frenet expansion gives
\begin{equation*}
x(s+d,0)-x(s,0)=\left(d-\frac{\kappa^2}{6}d^3\right)\boldsymbol\tau(s)+\left(\frac{\kappa}{2}d^2+\frac{\kappa'}{6}d^3\right)\mathbf{n}(s)+O(d^4).
\end{equation*}

On the other hand, we know that
\begin{equation*}
\left(\sqrt{1-\beta^2}\boldsymbol\tau(s)+\beta\mathbf{n}(s)\right)l=x(s+d,0)-x(s,0).
\end{equation*}
Comparing the tangent and normal components, respectively, we obtain
\begin{gather*}
l\sqrt{1-\beta^2}=d-\frac{\kappa^2}{6}d^3+O(d^4),\;l\beta=\frac{\kappa}{2}d^2+\frac{\kappa'}{6}d^3+O(d^4).
\end{gather*}
Then, we deduce that
\begin{equation*}
\frac{\beta}{\sqrt{1-\beta^2}}=\frac{\frac{\kappa}{2}d^2+\frac{\kappa'}{6}d^3+O(d^4)}{d-\frac{\kappa^2}{6}d^3+O(d^4)}=\frac{\kappa}{2}d+\frac{\kappa'}{6}d^2+O(d^3).
\end{equation*}

Set $d=c_1\beta+c_2\beta^2+O(\beta^3)$. Comparing the coefficients of the series, we obtain
\begin{equation*}
s^+-s=\frac{2\beta}{\kappa}-\frac{4\kappa'}{3\kappa^3}\beta^2+O(\beta^3).
\end{equation*}
By reflection, we know that
\begin{equation*}
B=-\langle\sqrt{1-\beta^2}\boldsymbol\tau(s)+\beta\mathbf{n}(s),\mathbf{n}(s+d)\rangle.
\end{equation*}
Again, by Frenet expansion, we have
\begin{equation*}
\mathbf{n}(s+d)=\left(-\kappa d-\frac{\kappa'}{2}d^2\right)\boldsymbol\tau(s) + \left(1-\frac{\kappa^2}{2} d^2\right)\mathbf{n}(s) +O(d^3).
\end{equation*}
Therefore, we conclude that $B=\beta+\frac{2\kappa'}{3\kappa^2}\beta^2+O(\beta^3)$. As for the length $l$, we have
\begin{equation*}
l^2=\left(d-\frac{\kappa^2}{6}d^3+O(d^4)\right)^2+\left(\frac{\kappa}{2}d^2+\frac{\kappa'}{6}d^3+O(d^4)\right)^2.
\end{equation*}
This easily implies that $l=d-\frac{\kappa^2}{24}d^3+O(d^4)=\frac{2}{\kappa}\beta+O(\beta^2)$. In summary, we have
\begin{equation}\label{eq:short-chord}
\begin{aligned}
s^+-s&=\frac{2\beta}{\kappa}-\frac{4\kappa'}{3\kappa^3}\beta^2+O(\beta^3),\\
B&=\beta+\frac{2\kappa'}{3\kappa^2}\beta^2+O(\beta^3) ,\text{ \; and \; }l\sim\beta.
\end{aligned}
\end{equation}
Using Taylor expansion, we have
\begin{equation*}
\kappa(s^+)^{-1/3}=\kappa(s)^{-1/3} -\frac{\kappa'(s)}{3\kappa(s)^{4/3}}(s^+-s) +O((s^+-s)^2)
\end{equation*}
Therefore \eqref{eq:short-chord} gives
\begin{gather*}
\kappa(s^+)^{-1/3}=\kappa^{-1/3}\left(1-\frac{2\kappa'}{3\kappa^2}\beta+O(\beta^2)\right),\\
B\kappa(s^+)^{-1/3} =\left(\beta+\frac{2\kappa'}{3\kappa^2}\beta^2+O(\beta^3)\right) \kappa^{-1/3}\left(1-\frac{2\kappa'}{3\kappa^2}\beta+O(\beta^2)\right)=\beta\kappa^{-1/3}+O(\beta^3).
\end{gather*}
Since $l_j\sim\beta_j\sim\delta_j$, this proves the one-step estimates in \eqref{eq:action-stable}. It remains to control their accumulation over the fixed time interval. Summing the actual ray segment lengths gives
\begin{align*}
|\delta_j^{-1}-\delta_0^{-1}|\le\sum_{i<j}|\delta_{i+1}^{-1}-\delta_i^{-1}|\le C\sum_{i<j}l_i\le C_T.
\end{align*}
For sufficiently small $\delta_0$, we can invert this estimate:
\begin{equation*}
\frac{\delta_j}{\delta_0} =\frac1{1+\delta_0(\delta_j^{-1}-\delta_0^{-1})} =1+O_T(\delta_0).
\end{equation*}
This closes the small-action bound independently of the number of reflections and proves the lemma.
\end{proof}

\smallskip\noindent\emph{Regular-range averaging.} We now treat limiting ratios in $(0,1)$ and aim to prove that 
\begin{equation*}
\int_{\text{segment }j}\1_{\mathcal O_{\lambda,\alpha}}\left(t,x(t)\right)\dd t \geq c_*l_j.
\end{equation*}
We accomplish this by averaging the two phases $\phi$ and $\psi$ in \eqref{eq:actual-phases}. Indeed, we apply the Van der Corput Lemma (see \cite[Chapter~VIII, Proposition~2, pp.~332--334]{Stein1993})
\begin{lemma}\label{lem:van-der-corput}
Let $I=[a,b]$ be a bounded interval, let $k\geq 1$ be an integer, and let $f\in C^k(I;\mathbb R)$. Suppose that $|f^{(k)}(x)|\geq \sigma>0$, for all $x\in I$. If $k=1$, assume additionally that $f'$ is monotone on $I$. Then, for every $\Lambda>0$,
\begin{equation*}
\left|\int_I e^{i\Lambda f(x)}\,dx\right| \leq C_k(\Lambda\sigma)^{-1/k},
\end{equation*}
where $C_k$ depends only on $k$.
\end{lemma}
We want to verify the conditions of Lemma~\ref{lem:van-der-corput} on the ray segment away from the turning point. We use the large parameter $M=(2\pi h)^{-1}\lambda^3\sim\lambda^{-3/2}$ in \eqref{eq:M-eps} to replace the $\Lambda$ in Lemma~\ref{lem:van-der-corput}.

On the first half of a ray segment, from the boundary toward the turning region, reparametrize the segment with the parameter $t=t_j+(1-z)\frac{l_j}{2}$, $z\in[0,1]$. Our action ratio is $r=\beta\kappa^{-\frac{1}{3}}\lambda^{-1}$. Thus, we note that $\frac{\dd t}{\dd z}=-r\lambda\kappa^{-\frac{2}{3}}+O(\lambda^2)$. We write the ray in the boundary normal coordinates as $(s(t),y(t))$. Then, using the proof of Lemma~\ref{lem:action-stability},
\begin{equation*}
y(z)=\frac{\beta^2}{2\kappa}(1-z^2)+O(\lambda^3),\text{ \; and \; }\frac{\dd}{\dd t}y=\beta z+O(\lambda^2).
\end{equation*}
For the Airy coordinate $\zeta_{\lambda}$, we have
\begin{equation*}
\zeta_{\lambda}=2^{-\frac{2}{3}}\lambda^2-(2\kappa)^{\frac{1}{3}}y+O(\lambda^4)=2^{-\frac{2}{3}}\lambda^2 \left(1-r^2+r^2z^2\right)+O(\lambda^3).
\end{equation*}
If $r$ stays in a compact subset of $(0,1)$, by definition of $R_{\lambda}$, we obtain
\begin{equation*}
R_{\lambda}=\frac{2}{3}|\zeta_{\lambda}|^{\frac{3}{2}}=\frac{\lambda^3}{3}\left(1-r^2+r^2z^2\right)^{\frac{3}{2}}+O(\lambda^4   ).
\end{equation*}
For $z\in[1/2,1]$, the derivatives of $\phi,\psi$ satisfy
\begin{equation}\label{eq:compact-jets}
\begin{aligned}
\frac{\partial_z\phi}{M}&=\frac{r}{2}(\sqrt{1-r^2+r^2z^2}+rz)^2+O_{C^2}(\lambda),\\
\frac{\partial_z\psi}{M}&=\frac{r}{2}(\sqrt{1-r^2+r^2z^2}-rz)^2+O_{C^2}(\lambda),
\end{aligned}
\end{equation}
Indeed, we take $\phi$ for example. Note that $\frac{\dd}{\dd t}s=\frac{\sqrt{1-(\frac{\dd}{\dd t}y)^2}}{1-\kappa y}=1+\frac{\beta^2}{2}-\beta^2z^2 +O_{C^2}(\lambda^3)$. By the construction of $\Phi_{0,\lambda}$, we have
\begin{equation*}
\partial_s\Phi_{0,\lambda} =1-\frac12\lambda^2\kappa^{\frac{2}{3}}+O(\lambda^3), \;\partial_y\Phi_{0,\lambda}=O(\lambda^2).
\end{equation*}
Thus, we derive that
\begin{align*}
\frac{\dd}{\dd t}(\Phi_{0,\lambda}-t)=\partial_s\Phi_{0,\lambda}\frac{\dd}{\dd t}s+\partial_y\Phi_{0,\lambda}\frac{\dd}{\dd t}y-1=-\frac{\lambda^2\kappa^{\frac{2}{3}}}2 (1-r^2+2r^2z^2)+O(\lambda^3).
\end{align*}
Combining this with $\frac{\dd t}{\dd z}=-r\lambda\kappa^{-\frac{2}{3}}+O(\lambda^2)$, we conclude \eqref{eq:compact-jets}. Taking a Fourier approximation, we find that, for every nonzero integer pair $(p,n)$, the limiting derivative of $(p\phi+n\psi)/M$ is a nonzero smooth factor times
\begin{equation*}
(p+n)(w^2+1)+2(p-n)w,\text{ where  }w:=\frac{\sqrt{1-r^2+r^2z^2}}{rz}.
\end{equation*}
The function $w$ is strictly monotone with derivative bounded away from zero on the compact parameter set. This nonzero polynomial has degree at most two. Hence one of the first three derivatives of the combined phase $(p\phi+n\psi)/M$ is uniformly nonzero locally. For each fixed $(p,n)\ne(0,0)$, compactness and \eqref{eq:compact-jets} give a partition into uniformly finitely many subintervals on each of which one of these derivatives is bounded away from zero. On pieces where the first derivative is bounded away from zero, integration by parts and the uniform bound on the second derivative give $O(M^{-1})$; on the remaining pieces, apply Lemma~\ref{lem:van-der-corput} with $k=2$ or $3$. The constants may depend on $(p,n)$ and the fixed compact range of $r$, but are uniform in the boundary point, the common phase shift and sufficiently small $\lambda$. Thus we obtain
\begin{equation*}
\left|\int_{1/2}^1e^{2\pi i(p\phi(z)+n\psi(z))}\dd z\right|\leq CM^{-\frac{1}{3}}\rightarrow0, \;\forall(p,n)\in\Z^2\setminus\{(0,0)\}.
\end{equation*}
For a smooth periodic function $F$ on $(\R/\Z)^2$, Fourier approximation therefore yields
\begin{equation*}
\int_{1/2}^1F\left(\phi(z),\psi(z)\right)\dd z \rightarrow\frac12\int_{(\R/\Z)^2}F.
\end{equation*}
Since $|\mathcal R|=1/8$, choose a fixed smooth function $0\le F\le\1_{\mathcal R}$ with $\int_{(\R/\Z)^2}F\ge3/32$. For sufficiently small $\lambda$, the convergence above is uniform in the chosen compact range of action ratios, the boundary point and $\alpha$, with error at most $1/64$. Hence
\begin{equation}\label{eq:compact-average}
\int_{1/2}^1\1_{\mathcal R}\left(\phi(z),\psi(z)\right)\dd z\ge\frac{1}{32}.
\end{equation}

To recover the original time $t$, use $|dt/dz|=l_j/2$. For sufficiently small $\lambda$, uniformly in the chosen compact range of action ratios,
\begin{align*}
\int_{\text{segment }j}\1_{\mathcal O_{\lambda,\alpha}}\left(t,x(t)\right)\dd t \ge\frac{l_j}{2}\int_{1/2}^1\1_{\mathcal R}\left(\phi(z),\psi(z)\right)\dd z \ge\frac{l_j}{64}.
\end{align*}
Thus the phase average supplies a fixed observed fraction of each complete ray segment.

\smallskip\noindent\emph{Deeper rays.} If $r_j$ stays in a fixed compact subset $[r_-,r_+]$ of $(1,\infty)$, a central portion of each ray lies beyond the turning curve. To compare the depths, use $t=t_j+(1-z)l_j/2$ on the full ray segment, with $z\in[-1,1]$. Combining \eqref{eq:short-chord} and \eqref{eq:turning}, we have
\begin{gather*}
y(t)=\frac{\beta_j^2}{2\kappa(s_j)}(1-z^2)+O(\beta_j^3),\;y_c^\lambda(s(t))=\frac{\lambda^2}{2\kappa(s_j)^{1/3}}+O(\lambda^3),\\
\frac{y(t)}{y_c^\lambda(s(t))}=r_j^2(1-z^2)+O(\lambda).
\end{gather*}
Choose $z_*\in(0,1)$ with $r_-^2(1-z_*^2)>1$. For sufficiently small $\lambda$, the portion $|z|\le z_*$ is outside the oscillatory part. Hence $\1_{\mathcal O_{\lambda,\alpha}}\left(t,x(t)\right)=1$ for $|z|\le z_*$, and 
\begin{equation*}
\int_{\text{segment }j}\1_{\mathcal O_{\lambda,\alpha}}\left(t,x(t)\right)\dd t \ge\frac{l_j}{2}\int_{-z_*}^{z_*}\dd z=z_*l_j.
\end{equation*}
The choice of $z_*$ and the smallness threshold are uniform in the fixed compact range.

If $r_j\to\infty$ while $\beta_j\to0^+$, then $\lambda/\beta_j\to0$. On the central half of the ray, 
\begin{equation*}
y(t)\ge c\beta_j^2>C\lambda^2\ge y_c^\lambda(s(t)),\;\forall z\in[-1/2,1/2],
\end{equation*}
once the ratio is sufficiently large and $\beta_j$ sufficiently small. Thus the observed time on this ray segment is at least $l_j/2$.

\begin{figure}[tbp]
\centering
\includegraphics[width=.9\textwidth]{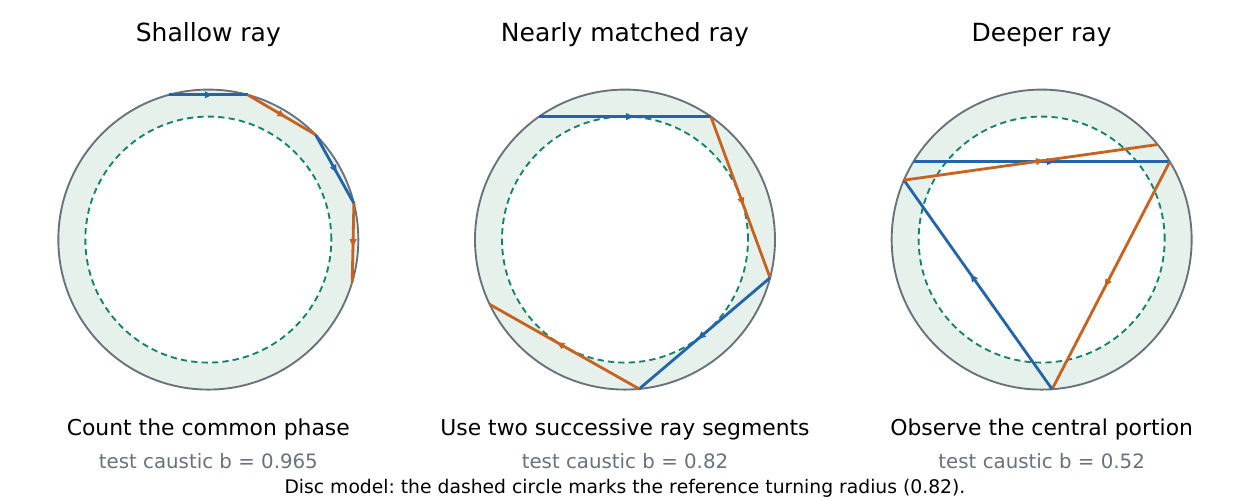}
\caption{Test rays relative to the normal depth of the moving packet region.}\label{fig:gcc-rays}%
\end{figure}

\subsection{One-scale estimate 2: shallow rays and the boundary phase relation}\label{subsec:gcc-shallow}

For $r$ small, the whole ray lies inside the oscillatory collar. Counting periods separately on each ray segment would lose a period at every reflection, and the number of reflections is unbounded. The boundary relation instead lets us count a continuous common phase over $[0,T]$. The matching rectangles alone give the following lower bound.

\begin{lemma}\label{lem:shallow}
Fix the local geometry and $0<T<T_{\rm loc}$. There are $r_{\rm sh},\lambda_{\rm sh},c_{\rm sh}>0$ such that every positively oriented regular glancing ray meeting $V$ during $[0,T]$, with initial action ratio $0<r_0\le r_{\rm sh}$ and scale $0<\lambda\le\lambda_{\rm sh}$ chosen by \eqref{eq:scales}, satisfies
\begin{equation*}
\int_0^T\1_{\mathcal R}\left(\phi\left(t,x(t)\right),\psi\left(t,x(t)\right)\right)\dd t\ge c_{\rm sh}T.
\end{equation*}
The constants depend only on the fixed geometry and time, and the bound is uniform in the common phase shift $\alpha$.
\end{lemma}
\begin{proof}
Define 
\begin{equation}\label{eq:shallow-sizes}
\nu(t):=(\pi h)^{-1}\left(\frac{1}{3}\lambda^3-R_\lambda(s(t),y(t))\right).
\end{equation}
$\nu$ measures the difference between normal phases from the initial boundary point to the segment point. The quarter-period boundary relation gives the exact identities
\begin{equation}\label{eq:anchor}
\phi=\psi-1/4-\nu\pmod1, \text{ \; with \; }\nu=0\text{ at each reflection}.
\end{equation}
Thus, we focus on $\psi$ and $\nu$ from now on. For $z\in[-1,1]$, by Lemma~\ref{lem:airy-coord} and the preceding section, we have
\begin{equation*}
\zeta_{\lambda}(s,y)=2^{-\frac{2}{3}}\lambda^2-(2\kappa)^{1/3}y+O(\lambda^2y+y^2),\;y(z)=\frac{\beta^2}{2\kappa}(1-z^2)+O(\beta^3).
\end{equation*}
This implies that
\begin{equation*}
R_{\lambda}(z)=\frac{\lambda^3}{3}\left(1-r^2(1-z^2)+O(\lambda^2r^2+\lambda r^3)\right)^{\frac{3}{2}}.
\end{equation*}
For $\psi$, we have
\begin{equation*}
2\pi h\frac{\dd}{\dd t}\psi=-\frac{\lambda^2\kappa^{\frac{2}{3}}}2 +O(\lambda^4+\beta^2+\lambda\beta).
\end{equation*}
Therefore, we conclude 
\begin{equation}\label{eq:shallow-jets}
\frac{\nu}{Mr_j^2}=1-z^2+O_{C^2}(\lambda^2+r_j^2), \text{ \; and \; }m_\psi(2\pi h)^{-1}\lambda^2\le-\frac{\dd}{\dd t}\psi\left(t,x(t)\right)\le M_\psi(2\pi h)^{-1}\lambda^2.
\end{equation}
We point out that $\psi$ is continuous at reflection in the fixed chart. %
On $I=[1/2,1]$, their derivatives can be bounded by
\begin{equation*}
\frac{Mr_j^2}{2}\le-\partial_z\nu\le4Mr_j^2,\;|\partial_z^2\nu|\le CMr_j^2,\;\partial_z\psi\sim Mr_j,\;|\partial_z^2\psi|\le CMr_j.
\end{equation*}

We can now separate the two motions. For a fixed $\nu$, define $W(\nu):=\int_0^1\1_{\mathcal R}(\theta-1/4-\nu,\theta)\dd\theta =2\left(1/4-\dist(\nu,\Z)\right)_+$. Then
\begin{equation}\label{eq:Wq}
 \int_0^A W(\nu)\dd \nu\ge A/16.
\end{equation}
Indeed, for $0<A\le1$, it follows from
\begin{equation*}
\int_0^A W(\nu)\dd\nu \ge
\begin{cases}
A/2-A^2\ge A/4,&0<A\le1/4,\\
1/16\ge A/16,&1/4\le A\le1.
\end{cases}
\end{equation*}
For $A\ge1$, each complete period contributes $\int_0^1W=1/8$, and $\lfloor A\rfloor/8\ge A/16$. Thus the averaged matching remains positive at every size of $\nu$.
\begin{enumerate}
    \item \emph{Small $\nu$.} Fix $\nu_0=1/64$. First suppose $Mr_j^2\le\nu_0$ on every ray segment meeting the time interval. Then $0\le\nu(t)\le2\nu_0$ by \eqref{eq:shallow-jets}. Let 
\begin{equation*}
J=[2\nu_0,1/4]\cup[1/2+2\nu_0,3/4],\text{ \; and \; }|J|=\frac{7}{16}.
\end{equation*}
The matching region contains every time for which $\psi(t)\in J$. In fact, if $\psi\in [2\nu_0,1/4]$, we know that $0\leq\psi-\nu\leq \frac{1}{4}$. Thus, $\phi=\psi-\frac{1}{4}-\nu\in [-\frac{1}{4},0]=[\frac{3}{4},1]\mod{1}$. Hence, $(\phi,\psi)\in\mathcal{R}$ by its definition and other cases are the same. By \eqref{eq:shallow-jets}, counting complete periods of the continuous phase gives
\begin{equation*}
|\{t\in[0,T]:\psi(t)\in J\}| \ge\frac{7}{16}\frac{m_\psi(2\pi h)^{-1}\lambda^2 T-2}{M_\psi(2\pi h)^{-1}\lambda^2}.
\end{equation*}
Since $(2\pi h)^{-1}\lambda^2\sim\lambda^{-\frac{5}{2}}\to\infty$ as $\lambda\to0^+$, for sufficiently small $\lambda$ we have
\begin{align*}
\frac{7}{16}\frac{m_\psi(2\pi h)^{-1}\lambda^2 T-2}{M_\psi(2\pi h)^{-1}\lambda^2} &=\frac7{16}\frac{m_\psi}{M_\psi} \left(T-\frac{2}{m_\psi(2\pi h)^{-1}\lambda^2}\right)\ge\frac{7T}{32C},
\end{align*}
with $\frac{M_\psi}{m_\psi}\leq C$.
\item \emph{Larger $\nu$.} Otherwise, there exists an index $j_0$ such that $Mr_{j_0}^2>\nu_0$. By Lemma~\ref{lem:action-stability}, for $\delta_*>0$ sufficiently small that $C_T\delta_0<\frac{1}{2}$, we have $\frac{\delta_j}{\delta_k}\geq\frac{1}{3}$, for any $j,k$. Then, for any $j$, we have
\begin{equation*}
Mr_j^2=Mr_{j_0}^2\frac{r_j^2}{r_{j_0}^2}=Mr_{j_0}^2\frac{\delta_j^2}{\delta_{j_0}^2}\geq \frac{1}{9}\nu_0.
\end{equation*}
Then, for $\nu_*=\frac{1}{18}\nu_0$, $Mr_j^2\ge\nu_*>0$ holds for every $j$. For $z\in[1/2,1]$, the boundary value $\nu(1)=0$ and the bounds for $-\partial_z\nu$ give $\nu(1/2)=\int_{1/2}^1\left(-\nu'(z)\right)\dd z\ge \frac{Mr_j^2}{4}$. By change of variables and \eqref{eq:Wq},  
\begin{align*}
\int_I W\left(\nu(z)\right)\dd z \ge\frac1{4Mr_j^2}\int_0^{\nu(1/2)}W(\nu)\dd\nu \ge\frac{\nu(1/2)}{64Mr_j^2}\ge\frac1{256}.
\end{align*}
Fix $0<\varepsilon\le1/3072$. By smoothing the rectangle indicators from inside, with sufficiently narrow boundary strips, we find a smooth periodic lower approximation $W_{\varepsilon}$ such that
\begin{gather*}
0\le W_{\varepsilon}(\psi,\nu)\le\1_{\mathcal R}(\psi-1/4-\nu,\psi),\\
0\leq W(\nu)-\int_0^1W_{\varepsilon}(\psi,\nu)\dd\psi\leq\varepsilon.
\end{gather*}
The second estimate is uniform in $\nu$, since each discarded boundary strip has $\psi$-length bounded by a constant times its width. Then we take a real finite Fourier approximation in $(\psi,\nu)$, such that 
\begin{equation*}
F=\sum_{(n,p)\text{ finite}}c_{p,n}e^{2\pi i(p\psi+n\nu)},\text{ \; and \; }\|F-W_{\varepsilon}\|_{L^{\infty}}\leq \varepsilon.
\end{equation*}
Moreover, for $F_0:=\int_0^1F(\psi,\nu)\dd\psi$, we have $F_0\geq W(\nu)-2\varepsilon$. For its modes with $p\ne0$, if $r_{\rm sh}$ is sufficiently small,
\begin{equation*}
|p\psi'+n\nu'|\ge cMr_j,|p\psi''+n\nu''|\leq CMr_j.
\end{equation*}
Integration by parts therefore gives
\begin{align*}
\left|\int_I e^{2\pi i(p\psi+n\nu)}\dd z\right| \le C\left(\frac1{\inf_I|p\psi'+n\nu'|} +\int_I\frac{|p\psi''+n\nu''|}{|p\psi'+n\nu'|^2}\dd z\right)\le C (Mr_j)^{-1}.
\end{align*}
Combining the smoothing, Fourier approximation and oscillatory errors, we obtain
\begin{align*}
\int_I\1_{\mathcal R}(\phi,\psi)\dd z \ge \int_IF\dd z-\varepsilon|I| \ge\int_I W\left(\nu(z)\right)\dd z-3\varepsilon|I|-C (Mr_j)^{-1}.
\end{align*}
Since $Mr_j=\sqrt{M}\sqrt{Mr_j^2}\ge\sqrt{\nu_*}\sqrt{M}\to\infty$ uniformly in $j$, for sufficiently small $\lambda$ we obtain $\int_I\1_{\mathcal R}(\phi,\psi)\dd z\ge \frac{1}{512}$. Therefore,
\begin{align*}
\int_0^T\1_{\mathcal R}\left(\phi(t),\psi(t)\right)\dd t \ge\sum_{\text{complete segments }j}\frac{l_j}{2}\frac1{512} \ge\frac{T-C\lambda}{1024}.
\end{align*}
Only the first and last partial ray segments have been dropped, and their total length is $O(\lambda)$, which is harmless.
\end{enumerate}

\end{proof}

\subsection{One-scale estimate 3: nearly matched rays and the quarter-turn relation}
\label{subsec:gcc-matched}

Near $r=1$, a ray almost follows one propagation branch. Its slow phase may remain outside the required plateau intervals throughout a ray segment. The next reflection supplies a quarter-period shift, and
\[
J_0\cup(J_0-1/4)=\R/\Z,
\]
where $J_0=J_+\cup J_-$. Thus one of two successive slow phases should lie in a plateau interval. We must quantify the error in this shift and control phases near the corners. Recall that $\mathcal K$ is the fixed finite corner set and that
\[
\partial J_+\cup\partial J_-\subset\mathcal K.
\]
Define the matched boundary momentum and the mismatch by
\begin{equation}\label{eq:mismatch}
a_\lambda(s):=\sqrt{1-C_\lambda'(s)^2}, \quad \rho_j:=\frac{\beta_j-a_\lambda(s_j)}{\lambda}, \quad e_j:=M\rho_j^2.
\end{equation}
\begin{lemma}\label{lem:matched}
There are $\rho_*,\lambda_*,c>0$ such that the following holds. Let two successive complete segments of a positively oriented regular ray lie in the fixed chart, with both segments contained in $[0,T]$. If $|\rho_0|\le\rho_*$ and $0<\lambda\le\lambda_*$, then the observation region $\mathcal O_{\lambda,\alpha}$ in \eqref{eq:local-obs} occupies at least $c(l_0+l_1)$ time on their union. The bound is uniform in the boundary point and common phase origin.
\end{lemma}
\begin{proof}
We first record the phase comparison and the averaging estimate used on the two segments. For two segments in the mismatch neighborhood of Lemma~\ref{lem:phase-jets}, that lemma gives, on each fixed compact regular incoming band $I\subset(0,1]$ bounded away from zero,
\begin{gather}
\begin{gathered}
\psi_j(z) = b_j+\epsilon D_\lambda(s_j,\rho_j,z) +e_jH_\lambda(s_j,\rho_j,z), \\
\partial_zH_0(s,0,z) =\frac{\kappa(s)^{-2/3}}{2z^2},
\end{gathered}
\label{eq:slow-jets}
\\
\begin{aligned}
\phi_j'/M &\rightarrow2z^2 \quad\text{as }(\lambda,\rho_j)\to(0,0), \\
\rho_1 &=\left(1+O(\lambda)\right)\rho_0+O(\lambda^4), \quad |e_1-e_0|\le C\lambda e_0+C\epsilon.
\end{aligned}
\label{eq:next-jets}
\end{gather}
Here $b_j=\psi_j(1)$ and $\epsilon=(2\pi h)^{-1}\lambda^6$. The coefficients have bounded fixed derivatives on $I$, and the phase limit holds with every fixed number of $z$-derivatives, uniformly in the boundary point on the fixed arc. We fix the regular bands first, then shrink the mismatch neighborhood, and finally take $\lambda$ small.

Centering each slow phase at its boundary value and comparing successive ray segments yields
\begin{equation}\label{eq:quarter-return}
\begin{gathered}
\psi_1(z) = \psi_0(z)+1/4+\delta_{\rm ret}(z)\pmod1, \\
\|\delta_{\rm ret}\|_{L^\infty(I)} \le(C_0+\lambda C_I)e_0+D_I\epsilon.
\end{gathered}
\end{equation}
The constant $C_0$ comes from the boundary return and is independent of $I$. The derivative convergence in \eqref{eq:next-jets} also gives
\[
\phi_j'\ge c_I M, \quad |\phi_j''|\le C_I M
\]
on each fixed band, after shrinking the mismatch and scale bounds. In particular,
\[
\int_I\frac{|\phi_j''|}{|\phi_j'|^2}\dd z \le\frac{C_I}{M}.
\]
Once a slow phase lies in $J_0$, the fast phase should spend one quarter of its time in the corresponding interval $F_+$ or $F_-$. We quantify this while allowing the slow phase to vary. If
\[
\phi'\ge c_I M, \quad \int_I\frac{|\phi''|}{|\phi'|^2}\dd z \le\frac{C_I}{M},
\]
then for any bounded periodic $p$ and $w\in BV(I)$,
\begin{equation}\label{eq:BV-average}
\left| \int_I w(z)\left(p\left(\phi(z)\right)-\overline p\right)\dd z \right| \le \frac{C_{I,p}}{M} \left(\|w\|_{L^\infty}+\Var_Iw\right).
\end{equation}
Indeed, let $P$ be a bounded periodic primitive of $p-\overline p$. Stieltjes integration by parts gives
\[
\int_I w(z)\left(p\left(\phi(z)\right)-\overline p\right)\dd z = \left[\frac{wP(\phi)}{\phi'}\right]_{\partial I} - \int_I P(\phi)\dd\!\left(\frac{w}{\phi'}\right).
\]
The assumptions on $\phi$ imply
\begin{align*}
\left\|\frac{w}{\phi'}\right\|_{L^\infty} + \Var_I\!\left(\frac{w}{\phi'}\right) &\le \frac{C_I}{M} \left(\|w\|_{L^\infty}+\Var_Iw\right) + \|w\|_{L^\infty} \int_I\frac{|\phi''|}{|\phi'|^2}\dd z \\
&\le \frac{C_I}{M} \left(\|w\|_{L^\infty}+\Var_Iw\right).
\end{align*}
This proves \eqref{eq:BV-average}.

To apply it to the matching rectangles, write
\[
\1_{\mathcal R}(\phi,\psi) = \1_{F_+}(\phi)\1_{J_+}(\psi) + \1_{F_-}(\phi)\1_{J_-}(\psi).
\]
If the slow phase is monotone with variation $O_I(e)$, then
\[
\|\1_{J_\pm}\circ\psi\|_{L^\infty}\le1, \quad \Var_I(\1_{J_\pm}\circ\psi) \le C(1+\Var_I\psi) \le C_I(1+e).
\]
Since $|F_+|=|F_-|=1/4$, applying \eqref{eq:BV-average} to the two terms gives
\begin{equation}\label{eq:conditional-average}
\int_I\1_{\mathcal R}(\phi,\psi)\dd z = \frac14\int_I\1_{J_0}(\psi)\dd z + O_I\left((1+e)/M\right).
\end{equation}
Thus the remaining issue is how much time the slow phase spends in $J_0$.

\smallskip
\noindent\emph{Estimating the observed time.} We use two regular incoming bands: the first for the quarter-turn relation, and the second, closer to the caustic, for a larger slow-phase sweep. The estimates have three roles.
\begin{itemize}[leftmargin=*,itemsep=3pt]
\item For the smallest mismatches, additional observation near $\mathcal K$ absorbs the return error, without requiring monotonicity of $\psi$.
\item For intermediate mismatches, the quarter-turn relation compares two slow phases; monotonicity bounds the time lost near a corner.
\item For larger mismatches, the second band lets the slow phase cross several periods of $J_0$.
\end{itemize}

First fix a preliminary mismatch neighborhood and a boundary-return constant $C_0$ valid throughout it in \eqref{eq:quarter-return}. This constant is independent of the regular band. We retain the same upper bound when shrinking the neighborhood below. Choose $m_s,C_s>0$ with slack relative to the fixed curvature bounds in \eqref{eq:slow-jets}.

Fix $d_1$, then $e_c$, then $d_2$, subject to
\begin{equation*}
\begin{array}{cl}
     0<d_1<1/8,&\frac{8\left(C_0+2\right)(\#\mathcal K)d_1^2}{m_s}\le\frac{d_1}{4}  \\
     e_c>0,& \frac{C_se_c}{2d_1}\le\frac1{64},\;\left(C_0+2\right)e_c\le\frac1{64},\\
     0<d_2<d_1/4,&\frac{m_se_c}{4d_2}\ge8.
\end{array}
\end{equation*}
The first two choices control corner losses on $I_1=[d_1,2d_1]$; the last ensures a slow-phase sweep on $I_2=[d_2,2d_2]$. The choices of $d_1,d_2$ are independent of the common phase shift $\alpha$ in $\mathcal O_{\lambda,\alpha}$.

With $I_1,I_2$ fixed, shrink the preliminary mismatch neighborhood so that both bands are regular, and fix their constants $C_I,D_I$ in \eqref{eq:quarter-return}. We then choose $\rho_*$ strictly inside this neighborhood and finally choose $\lambda_*$.

To ensure that the phase comparison applies to both segments, we first use \eqref{eq:short-chord} and $a_\lambda'=O(\lambda)$ to obtain $\beta_1-\beta_0=O(\lambda^2)$ and $a_\lambda(s_1)-a_\lambda(s_0)=O(\lambda^2)$.  Consequently,
\[
\rho_1=\rho_0+O(\lambda).
\]
Taking $\lambda_*$ sufficiently small therefore keeps both segments in the fixed mismatch neighborhood. The sharper estimates in \eqref{eq:next-jets} then apply. Since $h\sim\lambda^{9/2}$ and $\gamma=\lambda^{1/4}$, we deduce that $M\sim\lambda^{-3/2}$, $\epsilon\sim\lambda^{3/2}$, and $\frac{\epsilon}{\gamma}\sim\lambda^{5/4}\rightarrow0$. In particular, we may take $\lambda_*$ small enough that $c_g\gamma\le e_c$, where $c_g=\frac1{4(C_0+2)}$. Taking $\lambda C_{I_1}\le1$ and $D_{I_1}\epsilon\le c_g\gamma$, we obtain from \eqref{eq:quarter-return}
\[
\|\delta_{\rm ret}\|_{L^\infty(I_1)} \le
\begin{cases}
\gamma/2,& e_0\le c_g\gamma,\\
\left(C_0+2\right)e_0,& e_0\ge c_g\gamma.
\end{cases}
\]
Indeed, the right-hand side of \eqref{eq:quarter-return} is at most $\left(C_0+1\right)e_0+D_{I_1}\epsilon$.

In the second range, the additive error in the slow-phase derivative is small relative to $e_0$, since
\[
e_0\ge c_g\gamma\rightarrow\frac{\epsilon}{e_0}\le\frac{\epsilon}{c_g\gamma}\rightarrow0.
\]
Moreover, \eqref{eq:next-jets} gives $\frac{e_1}{e_0}=1+O(\lambda)+O(\epsilon/e_0)$. Combining this with \eqref{eq:slow-jets}, and decreasing $\rho_*,\lambda_*$ if necessary, we obtain 
\begin{equation}\label{eq:slow-speed}
m_se_0/z^2\le\psi_j'(z)\le C_se_0/z^2,\quad j\in\{0,1\},
\end{equation}
on either band whenever $e_0\ge c_g\gamma$.

\smallskip
\noindent\emph{Smallest mismatches: $e_0\le c_g\gamma$.} Choose a smooth periodic $g_\gamma\in[0,1]$, equal to one within $2\gamma$ of $\mathcal K$, zero at distance at least $3\gamma$, and with $|g_\gamma'|\le C/\gamma$. Put
\[
w_\pm=(1-g_\gamma)\1_{J_\pm}, \quad B_\gamma=g_\gamma+w_++w_-.
\]
These weights are smooth at the interval boundaries, since $\partial J_\pm\subset\mathcal K$.

The support of $g_\gamma$ lies in the additional $4\gamma$ observation strips from \eqref{eq:gl-extra}. Consequently, the observed indicator is at least
\[
g_\gamma(\psi)+w_+(\psi)\1_{F_+}(\phi)+w_-(\psi)\1_{F_-}(\phi).
\]
\eqref{eq:slow-jets} and \eqref{eq:next-jets} give $e_0,e_1=O(\gamma)$, $|\psi_j'|\le C_{I_1}\gamma$. Hence
\[
\Var_{I_1}(w_\pm\circ\psi_j) \le \frac C\gamma \int_{I_1}|\psi_j'(z)|\dd z \le C_{I_1}.
\]
Equation~\eqref{eq:BV-average} therefore applies with error $O_{I_1}(M^{-1})$. The perturbed quarter-turn relation implies $B_\gamma(\psi_0(z))+B_\gamma(\psi_1(z))\ge1$, for $z\in I_1$. Indeed, if the first term is less than one, then $\psi_0(z)\notin J_0$ and its distance from $\mathcal K$ is greater than $2\gamma$. Since $|\delta_{\rm ret}(z)|\le\gamma/2$, we still have
\[
\psi_0(z)+\delta_{\rm ret}(z)\notin J_0.
\]
Applying $J_0\cup(J_0-1/4)=\R/\Z$ to this perturbed phase gives $\psi_1(z)\in J_0$, so the second term equals one.

Let $t_j(z)=t_j+(1-z)l_j/2$ on the incoming band. Summing the two fast averages, we obtain
\begin{align*}
\sum_{j=0}^1\int_{I_1}\1_{\mathcal O_{\lambda,\alpha}}\left(t_j(z),x\left(t_j(z)\right)\right)\dd z&\ge\sum_{j=0}^1\int_{I_1}\left(g_\gamma(\psi_j)+\frac14(w_++w_-)(\psi_j)\right)\dd z-\frac{C_{I_1}}M\\
&\ge\frac14\int_{I_1}\left(B_\gamma(\psi_0)+B_\gamma(\psi_1)\right)\dd z-\frac{C_{I_1}}M\\
&\ge\frac{d_1}{4}-\frac{C_{I_1}}M\ge\frac{d_1}{8}.
\end{align*}
The third line keeps only one quarter of the corner observation; the last uses the perturbed quarter-turn covering.

\smallskip
\noindent\emph{Intermediate mismatches: $c_g\gamma\le e_0\le e_c$.} The slow phases are monotone by \eqref{eq:slow-speed}. If both $\psi_0(z)$ and $\psi_1(z)$ lie outside $J_0$, the covering relation applied to $\psi_0(z)+\delta_{\rm ret}(z)$ gives
\[
\psi_0(z)+\delta_{\rm ret}(z)\in J_0.
\]
Since $\partial J_0\subset\mathcal K$, it follows that $\dist(\psi_0(z),\mathcal K)\le|\delta_{\rm ret}(z)|\le\left(C_0+2\right)e_0$. The range of $\psi_0$ on $I_1$ has length at most $\frac{C_se_0}{2d_1}\le\frac1{64}$, and $\left(C_0+2\right)e_0\le1/64$. Thus at most one periodic copy of the $\left(C_0+2\right)e_0$-neighborhood of each corner can meet this range. By \eqref{eq:slow-speed}, the preimage of each such neighborhood has length at most
\[
\frac{2\left(C_0+2\right)e_0}{m_se_0/(4d_1^2)} = \frac{8\left(C_0+2\right)d_1^2}{m_s}.
\]
Summing over the fixed corner set, the total loss is at most $\frac{8\left(C_0+2\right)(\#\mathcal K)d_1^2}{m_s}\le\frac{d_1}{4}$. Consequently,
\[
\sum_{j=0}^1\int_{I_1}\1_{J_0}\left(\psi_j(z)\right)\dd z\ge |I_1|-d_1/4=\frac{3d_1}{4}.
\]
Applying \eqref{eq:conditional-average} to both ray segments gives
\begin{align*}
\sum_{j=0}^1\int_{I_1}\1_{\mathcal R}\left(\phi_j(z),\psi_j(z)\right)\dd z&\ge\frac14\sum_{j=0}^1\int_{I_1}\1_{J_0}\left(\psi_j(z)\right)\dd z-\frac{C_{I_1}(1+e_c)}M\\
&\ge\frac{3d_1}{16}-\frac{C_{I_1}(1+e_c)}M\ge\frac{d_1}{8}.
\end{align*}

\smallskip
\noindent\emph{Larger mismatches within the fixed neighborhood: $e_0\ge e_c$.} On $I_2$, \eqref{eq:slow-speed} gives a slow-phase sweep of at least $\frac{m_se_0}{4d_2}\ge8$ and a speed ratio at most $4C_s/m_s$. Write
\[
\widetilde{\Delta}\psi_0:=\psi_0(2d_2)-\psi_0(d_2).
\]
Since $J_0$ occupies half of each period, discarding the incomplete periods gives
\begin{align*}
\int_{I_2} \1_{J_0}\left(\psi_0(z)\right)\dd z \ge \frac{\widetilde{\Delta}\psi_0/2-1}{\sup_{I_2}\psi_0'} \ge \frac{\widetilde{\Delta}\psi_0/4}{C_se_0/d_2^2} \ge \frac{m_sd_2}{16C_s} =:\nu_2>0.
\end{align*}
Equation~\eqref{eq:conditional-average} therefore bounds the matching measure below by
\[
\nu_2/4-C_{I_2}(M^{-1}+\rho_0^2).
\]
Choose $\rho_*$ small enough and then $\lambda_*$ small enough to make this at least $\nu_2/8$. These choices occur after $I_2$ is fixed.

\smallskip
Finally, we return to physical time. Since $l_1/l_0=1+O(\lambda)$ and $|dt/dz|=l_j/2$, the positive parameter fractions on the selected band $I_1$ or $I_2$ give
\begin{align*}
\sum_{j=0}^1 \int_{\text{segment }j} \1_{\mathcal O_{\lambda,\alpha}} \left(t,x(t)\right)\dd t &\ge \frac{\min(l_0,l_1)}2 \sum_{j=0}^1 \int_I \1_{\mathcal O_{\lambda,\alpha}} \left(t_j(z),x\left(t_j(z)\right)\right)\dd z \\
&\ge c(l_0+l_1).
\end{align*}
Here $t_j(z)=t_j+(1-z)l_j/2$ on the incoming band. In the larger-mismatch case, the contribution from the first segment alone gives the required lower bound. This finishes the proof.
\end{proof}

\subsection{Completing the proof of the one-scale estimate}\label{subsec:gcc-one-scale}

We combine the estimates for the limiting action ratios by contradiction, obtaining constants uniform as both the ray angle and the scale parameter $\lambda$ approach zero.

\begin{proof}[Proof of Proposition~\ref{prop:one-scale}]
We argue by contradiction. Suppose that no positive constants $\eta_0,\beta_*,\lambda_*$ satisfy the proposition. Then there exist admissible scales $\lambda_n\to0^+$, reference normal components $\beta_n\to0^+$, total phase shifts $\alpha_n$, and rays $x_n$ such that
\begin{equation*}
\frac1T\int_0^T\1_{\mathcal O_{\lambda_n,\alpha_n}}\left(t,x_n(t)\right)\dd t\rightarrow0.
\end{equation*}
Passing to a subsequence, we may assume that their initial action ratios $r_n$ converge to some $r_*\in[0,\infty]$. By the finite-time action stability \eqref{eq:action-stable}, this limiting regime is preserved throughout $[0,T]$.

\vspace{2mm} Table~\ref{tab:gcc-ranges} identifies the estimate for each limiting ratio. For $r_*=1$, the expansion $a_\lambda(s)=\lambda\kappa(s)^{1/3}+O(\lambda^3)$ and \eqref{eq:mismatch} give, along $x_n$,
\begin{equation*}
\rho_j=\frac{\beta_j-a_{\lambda_n}(s_j)}{\lambda_n} =\kappa(s_j)^{1/3}(r_j-1)+O(\lambda_n^2)\rightarrow0.
\end{equation*}
The convergence is uniform over the ray segments meeting $[0,T]$. We may therefore apply Lemma~\ref{lem:matched} to disjoint successive pairs of complete ray segments contained in $[0,T]$. At most two partial ray segments and one unpaired ray segment are discarded, with total length $O(\beta_n)$. In particular,
\begin{align*}
\int_0^T\1_{\mathcal O_{\lambda_n,\alpha_n}}\left(t,x_n(t)\right)\dd t &\ge c\sum_{\text{paired complete segments }j}l_j\\
&\ge c(T-C\beta_n).
\end{align*}

The regular-range estimates likewise lose only partial ray segments. For $r_*=0$, the sequence eventually satisfies both $r_0\le r_{\rm sh}$ and $\lambda_n\le\lambda_{\rm sh}$ in Lemma~\ref{lem:shallow}, which already covers $[0,T]$. Thus in every case, for a fixed positive constant from the corresponding estimate,
\begin{equation*}
\int_0^T\1_{\mathcal O_{\lambda_n,\alpha_n}}\left(t,x_n(t)\right)\dd t \ge cT-C\beta_n,
\end{equation*}
contradicting the assumed vanishing fraction. These cases exhaust the compactified action ratio, uniformly in the phase origin.
\end{proof}

\subsection{GCC for the fixed observation region}\label{subsec:gcc-fixed}

\subsubsection{Regular billiard rays}\label{subsubsec:gcc-regular}

Recall the hidden regions and the fixed observation set from \eqref{eq:fixed-G}:
\begin{equation*}
\mathcal H_k=\mathcal U_k\setminus(\Gamma_k\cup\calE_k) ,\text{ \; and \; } G=\left((0,T)\times\Omega\right)\setminus\bigcup_{k\ge1}\mathcal H_k.
\end{equation*}
We choose the interval and scale parameters to give one lower bound $D_G(\rho)\ge\eta>0$ for every regular ray on $[0,T]$, hence for almost every initial state, where
\begin{equation*}
D_G(\rho)=\int_0^T\1_G\left(t,x_\rho(t)\right)\dd t.
\end{equation*}
For a positive glancing ray with reference normal component $\beta$, we split the moving packet regions at a fixed small multiple $\delta\beta$. Stronger separation in $s-t$ and the ray's small drift allow it to meet at most one moving packet region with $\lambda_k\ge\delta\beta$; all moving packet regions with $\lambda_k<\delta\beta$ lie in one thin collar. Figure~\ref{fig:gcc-packing} separates these two estimates.

Negative glancing rays are controlled by the total interval width. Rays separated from glancing are controlled by their time in the common boundary collar. We prove these three bounds in turn after fixing the family of moving packet regions.
\vspace{1.5mm}

To implement this separation, choose intervals $I_k\subset\mathcal I_{\rm in}$ of lengths $w_k$ with pairwise disjoint concentric triple enlargements contained in $\mathcal I_{\rm in}$. Their total length will be small. Choose the scales afterwards, so that
\begin{equation}\label{eq:packing-scales}
w_k\ge R\lambda_k^2,\; \; \lambda_k^2/w_k\to0 ,\text{ \; and \; }\lambda_k\le\lambda_*.
\end{equation}
The constant $R$ will be fixed below. Use these intervals and scales in Definition~\ref{def:curved-packet} to obtain $\mathcal U_k$, and set $\gamma_k=\lambda_k^{1/4}$. With $\Gamma_k$ from Proposition~\ref{prop:curved-support}, Proposition~\ref{prop:curved-observation} defines $\calE_k$ and assembles $G^{(N)},G$; its finite-stage identity \eqref{eq:curved-stages} remains valid under the stronger separation.

A regular ray that never enters $V$ is fully observed, since every hidden region lies in $V$. For all other rays, Lemma~\ref{lem:local-geometry} identifies the trajectory over $[0,T]$ with a ray in $\Omega_{\rm c}$. We use this identification for the three cases below.%

\smallskip\noindent\emph{Case 1. Positive glancing rays.} The aim is to reduce the fixed-set estimate to one single-scale estimate with an error. If $k$ is the only index such that $\lambda_k\ge\delta\beta$ and the ray meets $\mathcal U_k$, and all moving packet regions below this threshold lie in $y\le D\delta^2\beta^2$, then along that ray
\begin{equation}\label{eq:packing-reduction}
\1_G\ge\1_{\mathcal O_{\lambda_k,\alpha_k}}-\1_{\{y\le D\delta^2\beta^2\}}.
\end{equation}
Here $D$ depends only on the geometry. For a positive glancing test ray with reference normal component $\beta$, Lemma~\ref{lem:action-stability} gives $\beta_j\sim\beta$ uniformly. On the $j$th segment we have the expansion
\begin{equation*}
y(z)=\frac{\beta_j^2}{2\kappa(s_j)}(1-z^2)+O(\beta_j^3),\;\frac{\dd}{\dd t}s=1 +O(\beta^2),\text{ for }z\in[-1,1].
\end{equation*}
Thus, the condition $y\le D\delta^2\beta^2$ yields $1-z^2\leq C\delta^2+C\beta$. Using $|\dd t|=\frac{l_j}{2}|\dd z|$, and returning to the original interval $[t_j,t_{j+1}]$, we obtain
\begin{equation*}
|\{t\in[t_j,t_{j+1}]:y(t)\le D\delta^2\beta^2\}| \le C(\delta^2+\beta)l_j.
\end{equation*}
Summing over complete ray segments and allowing at most two partial ray segments, we obtain
\begin{align*}
|\{t\in[0,T]:y(t)\le D\delta^2\beta^2\}| &\le C(\delta^2+\beta)\sum_{\text{complete segments }j}l_j+C\beta\\
&\le C\delta^2T+C_T\beta.
\end{align*}
Using $\frac{\dd}{\dd t}s=1 +O(\beta^2)$, we know that for any $t_1,t_2\in[0,T]$, we have
\begin{equation*}
|s(t_1)-t_1-(s(t_2)-t_2)|\leq CT\beta^2.
\end{equation*}

In summary, we get
\begin{equation}\label{eq:packing-ray}
\operatorname{osc}_{[0,T]}(s-t)\le A\beta^2 ,\text{ \; and \; }|\{t:y(t)\le D\delta^2\beta^2\}|\le B\delta^2T+B\beta.
\end{equation}
After increasing $B$ if necessary, it depends only on the fixed geometry and time, and is independent of both sufficiently small $\delta$ and the family of moving packet regions. If $\lambda_k,\lambda_l\ge\delta\beta$, disjoint triple enlargements give
\begin{equation*}
\dist(I_k,I_l)\ge w_k+w_l\ge2R\delta^2\beta^2>A\beta^2.
\end{equation*}
The oscillation bound in \eqref{eq:packing-ray} therefore allows a ray to meet at most one such moving packet region. Indeed, suppose that a ray meets two such moving packet regions at $t_k\neq t_l$. Then, $s(t_k)-t_k\in I_k$ and $s(t_l)-t_l\in I_l$. By \eqref{eq:packing-ray}, 
\begin{equation*}
A\beta^2<\dist(I_k,I_l)\leq|s(t_k)-t_k-(s(t_l)-t_l)|\leq A\beta^2,
\end{equation*}
which is a contradiction.

With the one-scale observation constant $\eta_0$ from Proposition~\ref{prop:one-scale}, choose the constants 
\begin{equation*}
B\delta^2\le\eta_0/4,\; \beta_0\le\beta_*,\;B\beta_0\le\eta_0T/4,\text{ \; and \; }R>A/(2\delta^2).
\end{equation*}
The remaining moving packet regions satisfy the common inclusion
\begin{equation*}
\bigcup_{\lambda_l<\delta\beta}\mathcal U_l \subset\{(t,s,y):0<y\le D\delta^2\beta^2\}.
\end{equation*}
Thus all their hidden portions lie in one thin collar.
Therefore, we obtain \eqref{eq:packing-reduction}. Combining it with the one-scale lower bound \eqref{eq:one-scale}, we get
\begin{align*}
D_G &\ge\int_0^T\1_{\mathcal O_{\lambda_k,\alpha_k}}\left(t,x(t)\right)\dd t -|\{t:y(t)\le D\delta^2\beta^2\}|\\
&\ge\eta_0T-B\delta^2T-B\beta\\
&\ge\eta_0T-\eta_0T/4-\eta_0T/4\\
&=\eta_0T/2.
\end{align*}
If the ray meets no moving packet region with $\lambda_k\geq\delta\beta$, then, along the ray, $\1_G\ge1-\1_{\{y\le D\delta^2\beta^2\}}$. 

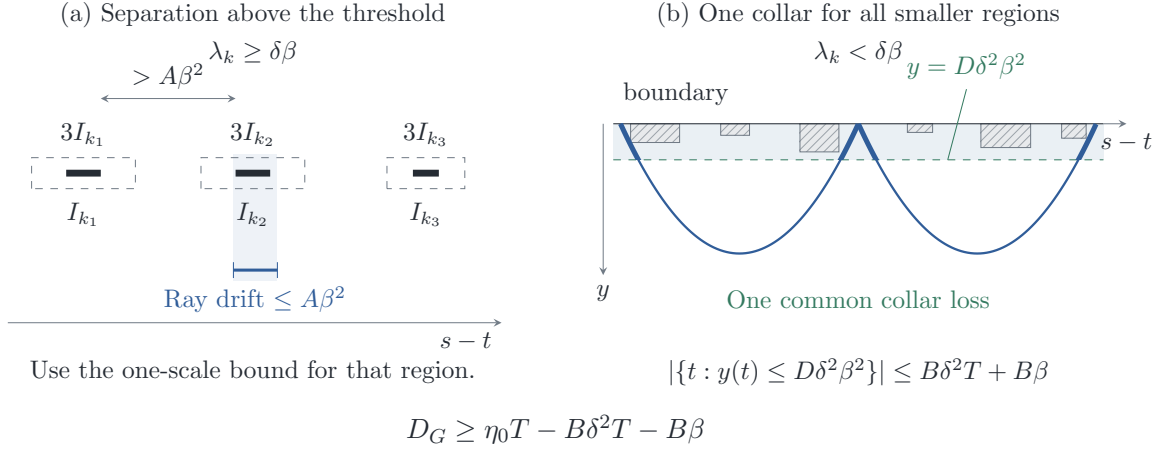
\begin{figure}[h]
\centering
\resizebox{\textwidth}{!}{%
\definecolor{packingink}{HTML}{252B33}
\definecolor{packingmuted}{HTML}{717B86}
\definecolor{packingblue}{HTML}{305D99}
\definecolor{packinggreen}{HTML}{347A60}
\definecolor{packingcollar}{HTML}{EAF0F4}
\begin{tikzpicture}[x=1cm,y=1cm,>=stealth,font=\small,
  every node/.style={text=packingink}]
  \begin{scope}
    \node at (3.4,3.02) {(a) Separation above the threshold};
    \node at (3.4,2.48) {$\lambda_k\ge\delta\beta$};
    \fill[packingblue!8] (3.13,-.68) rectangle (3.74,1.07);
    \foreach \i/\c/\w in {1/1.05/.48,2/3.4/.48,3/5.8/.36}{
      \draw[packingmuted,dashed,thin]
        ({\c-1.5*\w},.6) rectangle ({\c+1.5*\w},1.03);
      \draw[packingink,line width=2.7pt]
        ({\c-.5*\w},.815) -- ({\c+.5*\w},.815);
      \node at (\c,1.35) {$3I_{k_\i}$};
      \node at (\c,.24) {$I_{k_\i}$};
    }
    \draw[<->,packingmuted] (1.29,1.85) -- (3.16,1.85)
      node[midway,above] {$>A\beta^2$};
    \draw[packingblue,line width=1.2pt] (3.13,-.53) -- (3.74,-.53);
    \draw[packingblue] (3.13,-.64) -- (3.13,-.42);
    \draw[packingblue] (3.74,-.64) -- (3.74,-.42);
    \node[text=packingblue] at (3.43,-.93) {Ray drift $\le A\beta^2$};
    \draw[->,packingmuted] (0,-1.26) -- (6.85,-1.26)
      node[below left] {$s-t$};
    \node at (3.4,-1.94) {Use the one-scale bound for that region.};
  \end{scope}
  \begin{scope}[xshift=8.4cm]
    \node at (3.4,3.02) {(b) One collar for all smaller regions};
    \node at (3.4,2.48) {$\lambda_k<\delta\beta$};
    \fill[packingcollar] (0,1.5) rectangle (6.8,1.0);
    \foreach \x/\w/\d in {.24/.68/.26,1.49/.40/.16,2.59/.54/.39,4.08/.35/.12,5.1/.69/.33,6.22/.34/.20}{
      \fill[packingmuted!16] (\x,1.5) rectangle ({\x+\w},{1.5-\d});
      \begin{scope}
        \clip (\x,1.5) rectangle ({\x+\w},{1.5-\d});
        \foreach \h in {0,.12,...,1.2}{
          \draw[packingmuted!70,line width=.25pt]
            ({\x+\h-.5},1.0) -- ({\x+\h},1.5);
        }
      \end{scope}
      \draw[packingmuted,thin] (\x,1.5) rectangle ({\x+\w},{1.5-\d});
    }
    \draw[packingink] (0,1.5) -- (6.8,1.5);
    \node[anchor=south west] at (0,1.62) {boundary};
    \draw[packinggreen,dashed] (0,1.0) -- (6.8,1.0);
    \draw[packinggreen,thin] (4.9,1.98) -- (4.64,1.0);
    \node[text=packinggreen,anchor=south] at (4.9,1.98) {$y=D\delta^2\beta^2$};
    \foreach \left in {.10,3.4}{
      \draw[packingblue,line width=.9pt,domain=0:1,samples=75,variable=\z]
        plot ({\left+3.3*\z},{1.5-1.8*4*\z*(1-\z)});
      \begin{scope}
        \clip (0,1) rectangle (6.8,1.5);
        \draw[packingblue,line width=2pt,domain=0:1,samples=75,variable=\z]
          plot ({\left+3.3*\z},{1.5-1.8*4*\z*(1-\z)});
      \end{scope}
    }
    \draw[->,packingmuted] (-.14,1.5) -- (-.14,-.6)
      node[below] {$y$};
    \draw[->,packingmuted] (6.8,1.5) -- (7.12,1.5)
      node[below] {$s-t$};
    \node[text=packinggreen] at (3.4,-.93) {One common collar loss};
    \node at (3.4,-1.94)
      {$|\{t:y(t)\le D\delta^2\beta^2\}|\le B\delta^2T+B\beta$};
  \end{scope}
  \node[font=\normalsize] at (7.6,-2.74)
    {$D_G\ge\eta_0T-B\delta^2T-B\beta$};
\end{tikzpicture}
}
\caption{Combining scales for a positive glancing ray.}\label{fig:gcc-packing}
\end{figure}

\vspace{2mm} \noindent\emph{Case 2. Negative glancing rays.} These rays satisfy $(s-t)'=-2+O(\beta^2)$. Decrease the glancing threshold so that $|(s-t)'|\ge1$. The real coordinate $s-t$ is strictly decreasing and crosses each $I_k$ at most once. Choose the total interval length so that
\begin{equation*}
\left|\left\{t:s(t)-t\in\bigcup_k I_k\right\}\right| \le\sum_k w_k\le T/2.
\end{equation*}
They therefore have $D_G\ge T/2$.

\vspace{2mm} \noindent\emph{Case 3. Rays separated from glancing.} By definition, $\beta=\delta\kappa^{\frac{1}{3}}$. In this case, we require that $\beta_0>\beta_*>0$. By Lemma~\ref{lem:local-geometry}, we only focus on the fixed chart and the neighborhood $V$. Then, by the proof of Lemma~\ref{lem:action-stability}, we know that $\beta\sim l\sim\delta$, where $l$ is the length of the segment. By Lemma~\ref{lem:action-stability}, we have
\begin{equation*}
|\beta_{j+1}-\beta_j|\leq C\beta_j^2\leq Cl_j^2.
\end{equation*}
This implies that there exists a constant $C>0$, independent of $j$, such that
\begin{equation}\label{eq:log-beta-est}
|\log\beta_{j+1}-\log\beta_j|\le Cl_j.
\end{equation}
For $l_j$ large, we can use compactness and the continuity of $\beta$ to ensure the same estimate \eqref{eq:log-beta-est}. Moreover, \eqref{eq:log-beta-est} even holds in $\Omega_{\rm c}$ (see Lemma~\ref{lem:local-geometry}). Without loss of generality, we assume that the ray starts at a boundary point at $t=0$. Then, due to \eqref{eq:log-beta-est}, 
\begin{equation*}
|\log\beta_{j}-\log\beta_0|\le C\sum_ml_m\leq C(T+2\diam(\Omega_{\rm c})).
\end{equation*}
Therefore, we have $\beta_j\geq \beta_*e^{-C(T+2\diam(\Omega_{\rm c}))}>0$. Thus a fixed non-glancing angle cannot reach arbitrarily small angles in bounded time. Since $\beta_j\sim l_j$, we know that the segment length is uniformly bounded from below. Thus, there are uniformly finitely many reflections. 

Suppose that $l_j\geq l_*>0$ and $N_T$ is the number of finite segments of the ray. For the rays in this case that meet $V$, the trajectory over $[0,T]$ agrees with a ray in $\Omega_{\rm c}$ and stays in the fixed chart. Uniformly over these rays, we claim
\begin{equation*}
\sup_\rho|\{t\in[0,T]:0<y\left(x_\rho(t)\right)\le d\}|\rightarrow0, \text{ as }d\to0.
\end{equation*}
This shows that the complete unit-speed ray $ \rho$ spends little time in the very thin collar. In fact, we denote the ray segment by its endpoints at the boundary. Let
\begin{equation*}
P_{\partial}:=\{(a,b)\in\partial\Omega_{\rm c}\times\partial\Omega_{\rm c}:\text{ the length }|a-b|\geq l_*\}.
\end{equation*}
$P_{\partial}$ is a compact set. It is easy to see that
\begin{equation*}
T_d(a,b):=|a-b|\int_0^11_{0<\dist((1-u)a+ub,\partial\Omega_{\rm c})\leq d}\dd u
\end{equation*}
is precisely the time spent in the $d$-collar along the complete segment from $a$ to $b$. It suffices to prove that 
\begin{equation*}
\sup_{(a,b)\in P_{\partial}}T_d(a,b)\to0, \text{ as }d\to0.
\end{equation*}
We argue by contradiction. Suppose that there exists $\epsilon_0>0$ and a sequence $d_m\to0$, and pairs $(a_m,b_m)\in P_{\partial}$, such that $T_{d_m}(a_m,b_m)\geq\epsilon_0>0$. By compactness, up to a subsequence we may assume that $(a_m,b_m)$ has a limit $(a,b)\in P_{\partial}$. Using the preceding estimate, we know that $|a-b|\geq l_*>0$. For every fixed $u\in(0,1)$, we have $\dist((1-u)a_m+ub_m,\partial\Omega_{\rm c})\to\dist((1-u)a+ub,\partial\Omega_{\rm c})>0$. Then, as $m\to\infty$, we have $1_{0<\dist((1-u)a_m+ub_m,\partial\Omega_{\rm c})\le d_m}\to 0$. Moreover, the length of the segment has the trivial upper bound $\diam\Omega_{\rm c}$. The dominated convergence theorem therefore gives $T_{d_m}(a_m,b_m)\to0$, which is a contradiction.

All hidden regions lie in the collar of depth $d=C(\sup_k\lambda_k)^2$. Taking $\sup_k\lambda_k$ sufficiently small, we therefore obtain
\begin{equation*}
D_G(\rho)\ge T-|\{t\in[0,T]:0<y\left(x_\rho(t)\right)\le d\}|\ge T/2.
\end{equation*}

\vspace{2mm} The three geometric ranges now give
\begin{equation*}
D_G(\rho)\ge
\begin{cases}
\eta_0T/2,&\text{positive glancing rays},\\
T/2,&\text{negative glancing rays},\\
T/2,&\text{rays separated from glancing}.
\end{cases}
\end{equation*}
Consequently, with one constant independent of the initial ray and of the finite stage,
\begin{equation*}
\eta:=\frac T2\min\{\eta_0,1\}>0,\quad \essinf_{\rho\in S^*\Omega}D_G(\rho)\ge\eta,\text{ \; and \; } D_{G^{(N)}}\ge D_G\quad\text{for every }N.
\end{equation*}
Together with $D_G=T$ for rays avoiding $V$, this proves \eqref{eq:curved-gcc} for the fixed $G$ and for every $G^{(N)}\supset G$, whenever the parameters satisfy the choices above. The parameters and $\eta$ may depend on the fixed $0<T<T_{\rm loc}$.

Reflection instants and the null exceptional set of initial states do not affect the essential assertion; see \cite[Lemma~B.2 and Proposition~B.3]{AOSI} for the regular flow and invariance under null modifications.

\begin{remark}\label{rem:disc-time}
For the disc construction of Section~\ref{sec:glancing-disc}, the same argument works for every fixed $T>0$ on the arclength cover. Impose the additional choices
\begin{equation*}
A\beta_0^2<\pi,\text{ \; and \; } \left(\frac{3T}{2\pi}+2\right)\sum_k w_k\le T/2.
\end{equation*}
The first ensures that a positive glancing ray can meet only one lifted copy of a comparable moving packet region. A negative glancing ray meets at most $3T/(2\pi)+2$ copies of each interval, so the second bounds its total hidden time by $T/2$. This gives the same GCC conclusion. The realization is the periodized construction described in Proposition~\ref{prop:gl-disc-real}.
\end{remark}

\subsubsection{Boundary completion and generalized rays}\label{subsubsec:gcc-completion}

We now prove the stronger assertion in Theorem~\ref{thm:curved-gcc}(iii) for the specified Borel set $G$. Set $G^\sharp:=G\cup\left((0,T)\times\partial\Omega\right)$ and retain the constant $\eta$ from Section~\ref{subsubsec:gcc-regular}. We use the individual-ray estimates above and the location of the hidden regions.

\smallskip\noindent\emph{Localization of the relevant generalized rays.} Let $x_\gamma:[0,T]\to\overline\Omega$ be any unit-speed generalized ray, in the sense recalled in \cite[Appendix~B.1]{AOSI}. If it never meets $V$, then it meets no hidden region. At every $t\in(0,T)$, an interior point $x_\gamma(t)$ is observed in $G$, and a boundary point is observed in $G^\sharp$. Thus
\begin{equation*}
\int_0^T\1_{G^\sharp}\left(t,x_\gamma(t)\right)\dd t=T.
\end{equation*}
This case includes any diffractive, higher-order or infinite-order contacts outside the fixed neighborhood $V$.

Otherwise, choose $t_*\in[0,T]$ with $x_\gamma(t_*)\in V$. A generalized ray parametrized by physical time has a $1$-Lipschitz spatial projection. The distance bound in \eqref{eq:local-buffer} therefore gives
\begin{equation*}
|x_\gamma(t)-x_\gamma(t_*)|\le|t-t_*|\le T<\dist(\overline V,U^c), \quad\forall t\in[0,T].
\end{equation*}
Hence the entire trajectory stays in $U$, where $\Omega$ agrees with the strictly convex domain $\Omega_{\rm c}$. At a tangency in this chart, the free unit-speed line has
\begin{equation*}
\frac{\dd^2y}{\dd t^2}=-\kappa(s)\le-\kappa_*<0.
\end{equation*}
Thus all glancing contacts in $U$ are of the second-order gliding type described in \cite[Appendix~B.1]{AOSI}. Diffractive and higher-order contacts cannot occur there.

\smallskip\noindent\emph{Regularity from an interior point.} We first verify that every billiard trajectory from an interior state of $\Omega_{\rm c}$ remains regular on finite time intervals. Strict convexity makes every reflection of an interior chord transverse. To rule out an accumulation of reflections, suppose their successive segment lengths satisfy $\sum_jl_j<\infty$. Then the reflection points converge and $l_j\to0$. After discarding finitely many segments and relabeling, Lemma~\ref{lem:action-stability}, in either tangential orientation, gives
\begin{equation*}
l_j\sim\beta_j, \quad |\log\beta_{j+1}-\log\beta_j|\le Cl_j,
\end{equation*}
where $\beta_j>0$ is the inward normal component at the $j$th reflection. Consequently,
\begin{equation*}
\beta_j\rightarrow0, \quad \log\beta_N\ge\log\beta_0-C\sum_{j=0}^{N-1}l_j \ge\log\beta_0-C\sum_{j=0}^{\infty}l_j> -\infty.
\end{equation*}
These conclusions contradict each other. The same argument backwards in time excludes accumulation in either direction. Thus such a trajectory has only finitely many transverse reflections on every compact time interval.

Suppose now that $x_\gamma(t_0)\in\Omega$ for some $t_0\in(0,T)$. Start a billiard trajectory in $\Omega_{\rm c}$ from the same position and direction at $t_0$. The two trajectories coincide until their first reflection, and then across that reflection by the specular reflection law. We continue this comparison forwards and backwards. The auxiliary trajectory has no finite-time accumulation of impacts. Nor can it leave $U$ while the trajectories agree, since $x_\gamma\left([0,T]\right)$ stays a positive distance from $U^c$. Thus the comparison covers $[0,T]$. In particular, $x_\gamma$ is regular on this interval. This argument also excludes a transition between interior propagation and boundary gliding inside $U$.

The estimates in Section~\ref{subsubsec:gcc-regular} therefore apply to this ray. They are pointwise in the initial state: Proposition~\ref{prop:one-scale} covers every positive regular glancing ray, and the negative and non-glancing estimates are direct time bounds. The partial endpoint segments are already included there; the distance bound in Lemma~\ref{lem:local-geometry} accommodates the short extensions used in those estimates. Since reflection instants have zero time measure, we obtain
\begin{equation*}
\int_0^T\1_{G^\sharp}\left(t,x_\gamma(t)\right)\dd t =\int_0^T\1_G\left(t,x_\gamma(t)\right)\dd t\ge\eta.
\end{equation*}

\smallskip\noindent\emph{Trajectories on the boundary.} If the localized trajectory has no interior point, it lies on $\partial\Omega$ throughout $[0,T]$. Its observed time in $G^\sharp$ is then $T\ge\eta$. Together with the preceding cases, this proves the bound for every generalized ray. The same bound holds for the boundary completion of every $G^{(N)}\supset G$ by inclusion.%

\subsection{The exact waves on the same set}\label{subsec:gcc-energy}

It remains to coordinate the geometric choices with Theorem~\ref{thm:realization}. That theorem gives non-observability for any $\gamma_k\to0$, including the choice $\gamma_k=\lambda_k^{1/4}$ used above. The GCC conditions impose only additional upper bounds on the later scales.

\begin{proof}[Proof of Theorem~\ref{thm:curved-gcc}]
\noindent\emph{Step 1. Fixing the geometric data.} Keep the local geometry of Lemma~\ref{lem:local-geometry}, the time $0<T<T_{\rm loc}$, the diamond profiles $f_*,g_*$, the cutoff $\chi$ and the prescribed total phase shifts $\alpha_k$ fixed. As in Section~\ref{subsec:gcc-fixed}, choose intervals $I_k=z_k+(-w_k/2,w_k/2)$ with pairwise disjoint concentric triple enlargements contained in $\mathcal I_{\rm in}$ and
\begin{equation*}
0<w_k<1,\quad w_k\to0,\text{ \; and \; }\sum_k w_k\le T/2.
\end{equation*}
In particular, $\overline I_k\subset\mathcal I_{\rm in}$, and the intervals themselves are pairwise disjoint. By \eqref{eq:local-buffer}, their translates $t+\overline I_k$ remain inside the boundary chart for every $t\in[0,T]$.

\vspace{2mm} \noindent\emph{Step 2. Choosing scales compatible with realization and GCC.} Section~\ref{subsec:gcc-fixed} fixes $R$ and a common upper bound $\lambda_*$, including the restriction for rays separated from glancing. Decrease $\lambda_*$ if necessary so that Proposition~\ref{prop:curved-phases} applies and \eqref{eq:turning} gives $y_c^\lambda(s)<d_0$ throughout the chart whenever $0<\lambda\le\lambda_*$. Prescribe
\begin{equation*}
\overline\lambda_k:=\min\{\lambda_*,(w_k/R)^{1/2}\}.
\end{equation*}
The widths and centers from Step 1 satisfy the hypotheses of Theorem~\ref{thm:realization}; its common phase shifts and positive upper bounds may be prescribed arbitrarily. Apply that theorem to obtain Fourier cutoffs $\ell_k\to\infty$ and scales
\begin{equation*}
0<\lambda_k\le\overline\lambda_k,\quad \lambda_k\to0 ,\text{ \; and \; }\lambda_k^2/w_k\to0.
\end{equation*}
The upper bounds give the geometric inequalities
\begin{equation*}
\lambda_k\le\lambda_*,\text{ \; and \; }R\lambda_k^2\le w_k.
\end{equation*}
The stronger limit $\lambda_k^2/w_k\to0$ comes from the theorem's diagonal choice, not from $R\lambda_k^2\le w_k$ alone: after fixing each $\ell_k$, we take $\lambda_k$ still smaller to meet the realization estimates. Thus all three conditions in \eqref{eq:packing-scales} hold.

Determine $h_k$ by \eqref{eq:scales}. Theorem~\ref{thm:realization} also supplies the bounded positive normalization factors $c_k$, the normalized exact waves $u_k$ and the errors $r_k$ in \eqref{eq:full-realization}. Its choice keeps every retained mode's normal cutoff below $d_0$.

\vspace{2mm} \noindent\emph{Step 3. Defining the generated velocities and the common set.} For each $k$, Steps 1--2 give the moving packet region $\mathcal U_k:=\mathcal U_{\lambda_k,w_k,z_k}$ of Definition~\ref{def:curved-packet} and verify the assumptions of Proposition~\ref{prop:curved-support}. Use that proposition to define
\begin{equation*}
S_k:=c_kS_{\lambda_k,w_k,z_k,\alpha_k} ,\text{ \; and \; } \Gamma_k:=\Gamma_{\lambda_k,w_k,z_k,\alpha_k}.
\end{equation*}
The function $S_k$ is the full-profile velocity used in \eqref{eq:full-realization}, on this same moving packet region. Multiplication by $c_k$ preserves both conclusions of Proposition~\ref{prop:curved-support}:
\begin{equation*}
\1_{\mathcal U_k^c}S_k=0,\text{ \; and \; }\1_{\Gamma_k}S_k=0.
\end{equation*}

Set $\gamma_k:=\lambda_k^{1/4}\to0$. The intervals are disjoint with closures in $\mathcal I_{\rm in}$, the scales and reference widths tend to zero, and $c_k>0$. Thus Proposition~\ref{prop:curved-observation} applies to these same choices and defines $\calE_k,G^{(N)}$ and $G$. It gives
\begin{equation*}
G^{(N)}\downarrow G ,\text{ \; and \; } \1_{G^{(N)}}S_k=\1_GS_k=\1_{\calE_k}S_k, \quad\forall k\in\{1,\ldots,N\}.
\end{equation*}
This is one fixed Borel observation region for the entire wave sequence.

\vspace{2mm} \noindent\emph{Step 4. Non-observability and GCC on the same set.} Theorem~\ref{thm:realization}(ii) applies to this $G$, since the intervals are disjoint and $\gamma_k\to0$. Moreover,
\begin{equation*}
\frac{h_k}{w_k}=\frac{h_k}{\lambda_k^2}\frac{\lambda_k^2}{w_k}\rightarrow0, \quad\text{since }h_k/\lambda_k^2\sim\lambda_k^{5/2}.
\end{equation*}
Consequently,
\begin{equation*}
\|\1_G\partial_tu_k\|_{L^2((0,T)\times\Omega)}^2 \le2\|r_k\|_{L^2((0,T)\times\Omega)}^2 +C_T(\lambda_k^{1/4}+h_k/w_k)\rightarrow0.
\end{equation*}
Together with $E[u_k](0)=1$, this proves \eqref{eq:curved-nonobs} and the failure of observability.

Finally, the triple separation and total-width bound from Step 1, the scale bounds from Step 2, and the choice $\gamma_k=\lambda_k^{1/4}$ are precisely the conditions used in Section~\ref{subsec:gcc-fixed}. Section~\ref{subsubsec:gcc-regular} gives \eqref{eq:curved-gcc}, and Section~\ref{subsubsec:gcc-completion} gives the same lower bound for every generalized ray in $G^\sharp$. Both apply to this same $G$, with no further change of parameters. This finishes the proof of Theorem~\ref{thm:curved-gcc}.
\end{proof}

\begin{proof}[Proof of the {\hyperref[thm:main]{Main Theorem}}]
For the prescribed $\Omega$ and $T>0$, choose $0<\tau<\min\{T,T_{\rm loc}\}$ and apply Theorem~\ref{thm:curved-gcc} on $[0,\tau]$. Let $G_\tau$ and $u_k$ be its observation region and normalized exact waves.

Continue each exact wave to $[0,T]$ and set $G:=G_\tau$, viewed as a subset of $(0,T)\times\Omega$ with no observation after $\tau$. Then
\begin{equation*}
\int_0^T\1_G(t,\pi\Phi_t\rho)\dd t =\int_0^\tau\1_{G_\tau}(t,\pi\Phi_t\rho)\dd t\ge\eta \quad\text{for almost every }\rho,
\end{equation*}
whereas
\begin{equation*}
E[u_k](0)=1,\text{ \; and \; } \int_G|\partial_tu_k|^2\dd x\dd t =\int_{G_\tau}|\partial_tu_k|^2\dd x\dd t\rightarrow0.
\end{equation*}
These are the two conclusions in \eqref{eq:main} and \eqref{eq:main-nonobs}. Only the observation set is extended by zero in time; the local realization formula is used on $[0,\tau]$.

For the additional assertion in the footnote, set $G^\sharp:=G\cup\left((0,T)\times\partial\Omega\right)$ and $G_\tau^\sharp:=G_\tau\cup\left((0,\tau)\times\partial\Omega\right)$. The restriction of any generalized ray to $[0,\tau]$ is covered by Theorem~\ref{thm:curved-gcc}(iii). Hence
\begin{equation*}
\int_0^T\1_{G^\sharp}\left(t,x_\gamma(t)\right)\dd t \ge\int_0^\tau\1_{G_\tau^\sharp}\left(t,x_\gamma(t)\right)\dd t\ge\eta.
\end{equation*}
The interior wave observation is still the one on $G_\tau$.
\end{proof}

On the short interval of Theorem~\ref{thm:curved-gcc}, the relation to \cite[Section~2.7]{AOSI} is explicit: set $Z_k=\1_{\calE_k^c}S_k$ and $e_k=r_k+\1_{\calE_k}S_k$. Then
\begin{equation*}
\partial_tu_k=Z_k+e_k, \quad \1_{G^{(N)}}Z_k=0\quad\forall k\in\{1,\ldots,N\}, \quad \|e_k\|_{L^2}\to0 ,\text{ \; and \; }G^{(N)}\downarrow G.
\end{equation*}
The generated approximate velocity $S_k$ uses the full profiles $f_*,g_*$; $Z_k$ removes only the additional observed part whose energy tends to zero. This is the same finite-stage interpretation as in the radial disc construction of \cite[Section~3.3]{AOSI}, with that one additional error. Together with the fixed diamond symmetry, it verifies \cite[Definition~2.5]{AOSI} for glancing-reduced AOS.

The compatibility classification is complete within the specified scalar boundary--caustic loop. \hyperref[eq:gcc]{GCC} is realized for the diamond profile here and for the finite-step profiles in the radial construction of \cite[Section~4]{AOSI}.

A classification of all AOS, \hyperref[eq:gcc]{GCC} realization for all loop-compatible profiles, and observability after quantitative exclusion of all exact or approximate observable symmetries remain open. Smooth domains with flat or negatively curved boundary pieces are covered by the \hyperref[thm:main]{Main Theorem}.

\appendix
\numberwithin{figure}{section}
\section{Generalized rays in the local construction}\label{subsec:meas-boundary}

We use the generalized rays of Melrose--Sj\"ostrand \cite{MS1978}; see also \cite[Section~1C1]{LRLTT2017}. Their boundary types, the regular billiard flow, and the essential spacetime measurable GCC are recalled in \cite[Appendix~B]{AOSI}. Here we explain their role in the local construction of Section~\ref{sec:gcc} and distinguish the boundary-completed estimate from a limit of interior observed times.

\subsection{Local boundary completion}\label{subsec:rays-completion}

For the constructed Borel representative, Theorem~\ref{thm:curved-gcc}(iii) gives a uniform observed-time lower bound for every generalized ray after boundary completion,
\begin{equation*}
G^\sharp:=G\cup\left((0,T)\times\partial\Omega\right).
\end{equation*}
Section~\ref{subsubsec:gcc-completion} proves this assertion directly.

During the short interval fixed in Lemma~\ref{lem:local-geometry}, every hidden region lies in $V\Subset U$. A generalized ray avoiding $V$ is fully observed in $G^\sharp$. This includes rays with diffractive, higher-order or infinite-order contacts during that interval: such contacts can occur outside the positive-curvature chart $U$, but the distance bound prevents these rays from also meeting $V$.

A ray meeting $V$ remains in $U$ throughout the interval. If it has an interior point, comparison with $\Omega_{\rm c}$ and exclusion of finite-time accumulation make it a regular ray, to which the interior estimates apply. Otherwise it is entirely on the observed boundary. Thus no global finite-order-contact or convexity assumption is needed.

The interior glancing estimates concern near-tangent rays with transverse reflections and positive normal component, not boundary-gliding trajectories. Their bound is uniform as the normal component tends to zero. Section~\ref{subsec:gcc-energy} passes the all-ray assertion to arbitrary observation times by restricting to the short construction interval.

\subsection{Gliding limits and measurable observation}\label{subsec:rays-limits}

Along a positively curved arc, near-tangent billiards can approximate a boundary-gliding trajectory for as long as it remains in a compact subarc. This applies on every fixed finite interval when the whole boundary is strictly positively curved.

To see the geometry, choose initial points and directions approaching a given boundary point and tangent direction. The short-chord estimates of Section~\ref{subsec:gcc-action} keep the incidence angles small on the interval under consideration. The chord depths tend to zero, and the accumulated difference between boundary arc length and chord length also tends to zero. Thus suitably chosen trajectories satisfy
\begin{equation*}
\sup_{0\le t\le T}|x_j(t)-x_{\rm gl}(t)|\rightarrow0.
\end{equation*}

For a continuous weight $b$ on $[0,T]\times\overline\Omega$, uniform continuity then gives
\begin{equation*}
\int_0^T b\left(t,x_j(t)\right)\dd t\rightarrow \int_0^T b\left(t,x_{\rm gl}(t)\right)\dd t.
\end{equation*}
This implication need not hold for a measurable indicator $\1_G$. Hence approximation of trajectories alone does not transfer a lower bound for their observed times to the gliding limit.

For a regular billiard ray, reflection instants have zero one-dimensional measure. Adding any Borel subset of $(0,T)\times\partial\Omega$ to $G$ therefore changes neither its observed time nor the interior wave observation. It can, however, change the observed time of a boundary-gliding ray. Unlike relative openness, such a boundary assignment provides no observed interior neighborhood.

\section{Phase, Airy, and return calculations}\label{sec:phase}

We provide the phase and Airy calculations used in the following two parts of the proof.
\begin{itemize}[leftmargin=*,itemsep=1mm]
\item \emph{Section~\ref{sec:curved}: phase construction and residual.} Section~\ref{subsec:phase-coord} constructs the corrected Airy coordinate in Proposition~\ref{prop:curved-phases}. Section~\ref{subsec:residual} proves the Airy bounds \eqref{eq:airy-norms} and the normalized residual estimate \eqref{eq:forcing} used in Proposition~\ref{prop:finite-airy}.
\item \emph{Section~\ref{sec:gcc}: nearly matched rays.} Sections~\ref{subsec:phase-return}--\ref{subsec:jets} give the momentum return, differentiated phases, and quarter-turn comparison \eqref{eq:slow-jets}--\eqref{eq:quarter-return} used in Lemma~\ref{lem:matched} of Section~\ref{subsec:gcc-matched}.
\end{itemize}
The same finite-order phases are used for the approximate waves and the ray comparison. Throughout, we work on the fixed positive-curvature chart of Lemma~\ref{lem:local-geometry}. Estimates are uniform in the boundary point on its fixed compact arc; dependence on a retained incoming band is recorded below. 

\subsection{Phase construction for Proposition~\ref{prop:curved-phases}}\label{subsec:phase-coord}

We construct the Airy coordinate and quartic correction used in Proposition~\ref{prop:curved-phases}, preserving the exact boundary action in \eqref{eq:phase-data}. The scaled normal variable $Y=y/\lambda^2$ keeps the turning region at finite distance. 
We seek a correction of the form
\begin{equation*}
\zeta_\lambda=2^{-\frac{2}{3}}\lambda^2-(2\kappa(s))^{\frac{1}{3}}y-\lambda^2u(s)y-v(s)y^2.
\end{equation*}
Recall the form of $\Phi_{0,\lambda}$ in \eqref{eq:W0}, leaving $c_4$ to be determined.

\begin{lemma}\label{lem:airy-coord}
There are smooth functions $c_4(s)$, $u(s)$, and $v(s)$ such that, for all sufficiently small $\lambda>0$,
\begin{equation*}
\zeta_\lambda(s,0)=2^{-2/3}\lambda^2, \;E_0=O(\lambda^6) ,\text{ \; and \; }E_1=O(\lambda^4),
\end{equation*}
where $E_0,E_1$ are defined in \eqref{eq:eik}. On each fixed scaled collar $0\le y\le c\lambda^2$, these estimates hold uniformly on the fixed compact arc, with every fixed number of derivatives in $(s,y/\lambda^2)$.
\end{lemma}
\begin{proof}
\noindent\emph{Step 1. The quartic correction and boundary value.} We choose $u,v$, and $c_4$ to cancel the quartic terms in $E_0$. Set
\[
D:=3\kappa^2-\frac{\kappa''}{3\kappa} +\frac{4(\kappa')^2}{9\kappa^2}.
\]
Writing $Y=y/\lambda^2$, we seek the corrected Airy coordinate in the form
\[
\zeta_\lambda(s,y) =\lambda^2\left(2^{-2/3}-(2\kappa(s))^{1/3}Y\right) -\lambda^4\left(u(s)Y+v(s)Y^2\right).
\]
Expansion in the collar metric gives
\begin{align*}
|\dd\Phi_{0,\lambda}|_g^2 ={}&1+\lambda^2\left(2\kappa Y-\kappa^{2/3}\right)\\
&+\lambda^4\left( 2c_4+\frac14\kappa^{4/3} -2\kappa^{5/3}Y+DY^2 \right)+O(\lambda^6),\\
-\zeta_\lambda|\dd\zeta_\lambda|_g^2 ={}&\lambda^2\left(2\kappa Y-\kappa^{2/3}\right)\\
&+\lambda^4(2\kappa)^{2/3} \left( uY+vY^2 +2\left(Y-\frac{1}{2\kappa^{1/3}}\right) \partial_Y(uY+vY^2) \right) +O(\lambda^6).
\end{align*}
Taking the difference, we obtain
\begin{align*}
2c_4+\frac14\kappa^{4/3}-2\kappa^{5/3}Y+DY^2 ={}&n_0+n_1\left(Y-\frac{1}{2\kappa^{1/3}}\right)+D\left(Y-\frac{1}{2\kappa^{1/3}}\right)^2,
\end{align*}
where $n_0:=2c_4-\frac34\kappa^{4/3} +\frac{D}{4\kappa^{2/3}}$, and $n_1:=-2\kappa^{5/3}+\frac{D}{\kappa^{1/3}}$. Thus the correction must solve
\begin{align*}
&(2\kappa)^{2/3} \left( 1+2\left(Y-\frac{1}{2\kappa^{1/3}}\right)\partial_Y \right)(uY+vY^2)\\
&\quad= n_0+n_1\left(Y-\frac{1}{2\kappa^{1/3}}\right) +D\left(Y-\frac{1}{2\kappa^{1/3}}\right)^2.
\end{align*}
Since
\[
\left( 1+2\left(Y-\frac{1}{2\kappa^{1/3}}\right)\partial_Y \right) \left(Y-\frac{1}{2\kappa^{1/3}}\right)^j = (2j+1)\left(Y-\frac{1}{2\kappa^{1/3}}\right)^j,
\]
we obtain a polynomial solution:
\begin{align*}
uY+vY^2 =(2\kappa)^{-2/3}\left( n_0+\frac{n_1}{3} \left(Y-\frac{1}{2\kappa^{1/3}}\right) +\frac{D}{5} \left(Y-\frac{1}{2\kappa^{1/3}}\right)^2 \right).
\end{align*}
This solution is smooth at the leading turning point $Y=1/(2\kappa^{1/3})$.

At the boundary $Y=0$, the requirement that the correction vanish becomes
\[
n_0-\frac{n_1}{6\kappa^{1/3}} +\frac{D}{20\kappa^{2/3}}=0,
\]
which determines $c_4$. The remaining coefficients of $Y$ and $Y^2$ determine $u$ and $v$. We therefore obtain
\begin{equation}\label{eq:airy-coeff}
\begin{aligned}
c_4&=\frac{\kappa^{4/3}}{120} +\frac{\kappa''}{45\kappa^{5/3}} -\frac{4(\kappa')^2}{135\kappa^{8/3}},\\
u&=2^{1/3}\left( -\frac{2\kappa}{15} -\frac{\kappa''}{45\kappa^2} +\frac{4(\kappa')^2}{135\kappa^3} \right),\\
v&=2^{-2/3}\left( \frac{3\kappa^{4/3}}5 -\frac{\kappa''}{15\kappa^{5/3}} +\frac{4(\kappa')^2}{45\kappa^{8/3}} \right),\\
\zeta_\lambda(s,y) &:=2^{-2/3}\lambda^2-(2\kappa(s))^{1/3}y -\lambda^2u(s)y-v(s)y^2.
\end{aligned}
\end{equation}
The construction gives $E_0=O(\lambda^6)$ and 
\[
\zeta_\lambda(s,0)=2^{-2/3}\lambda^2,\;\partial_y\zeta_\lambda =-(2\kappa(s))^{1/3}+O(\lambda^2)<0.
\]
The implicit function theorem therefore gives a simple turning zero near the leading turning curve, with $y_c^\lambda(s) =\frac{\lambda^2}{2\kappa(s)^{1/3}}+O(\lambda^4)$.

\smallskip
\noindent\emph{Step 2. The mixed phase error.} Write
\[
U_4(s,Y)=-\frac{\kappa'(s)}{6\kappa(s)}Y^2,
\]
so that $U(s,\lambda^2Y)=\lambda^4U_4(s,Y)$. The quadratic coefficient of $E_1$ vanishes because
\begin{align*}
\partial_s\left(2^{-2/3}-(2\kappa)^{1/3}Y\right) -(2\kappa)^{1/3}\partial_YU_4= -\frac{2^{1/3}}3\kappa^{-2/3}\kappa'Y +\frac{(2\kappa)^{1/3}\kappa'}{3\kappa}Y =0.
\end{align*}
Thus $E_1=O(\lambda^4)$.

After substituting $y=\lambda^2Y$, the exact expressions for both defects are smooth in $(s,Y,\lambda^2)$ on every fixed compact scaled collar, since $\mu$ stays bounded away from zero. The cancellations above imply that
\[
\lambda^{-6}E_0(s,\lambda^2Y) \quad\text{and}\quad \lambda^{-4}E_1(s,\lambda^2Y)
\]
extend smoothly to $\lambda=0$. Consequently, the same estimates hold with every fixed number of derivatives in $(s,Y)$.

When $\kappa$ is constant, $c_4$ reduces to $\kappa^{4/3}/120$, consistent with the expansion of the exact circular action \eqref{eq:circular-action} under the same boundary normalization.
\end{proof}

\subsection{Normalized Airy residual for Proposition~\ref{prop:finite-airy}}\label{subsec:residual}

We prove the Airy integral bounds and residual estimate used in Step~1 of the proof of Proposition~\ref{prop:finite-airy}. Separating the phase-equation, transport, and envelope terms identifies the scale restrictions in \eqref{eq:forcing}.

\begin{lemma}\label{lem:airy-residual}
Fix the local geometry, envelope $\chi$, Fourier truncation, and time $T$ of Proposition~\ref{prop:finite-airy}. Use the scales \eqref{eq:scales} and an admissible width $w=w(\lambda)$ with $\lambda^2/w\to0$. Let $v_\ell$ be the real finite sum \eqref{eq:airy-wave}, with its amplitude factor $b_{\lambda,w}=(\lambda w)^{-1/2}$, normal cutoff, and zero extension. Fix the cutoff depths $Y_0<Y_1$ as in Section~\ref{subsec:curved-boundary}. Then the following estimates hold.
\begin{enumerate}[label=(\roman*),leftmargin=*,itemsep=1mm]
\item Uniformly in $s$ on the fixed compact arc,
\begin{align*}
\int_0^{Y_1\lambda^2}|\Ai(-\zeta_\lambda/h^{2/3})|^2\dd y\sim h^{1/3}\lambda,\; \int_0^{Y_1\lambda^2}|\Ai'(-\zeta_\lambda/h^{2/3})|^2\dd y\le Ch^{-1/3}\lambda^3.
\end{align*}
The same bounds hold for each retained mode $(\lambda_n,h_n)$, with a constant depending on $\ell$.
\item The residual satisfies
\begin{equation*}
\|(\partial_t^2-\Delta)v_\ell\|_{L^1(0,T;L^2(\Omega))}\le C_{\ell,T}\left(\frac{\lambda^5}{h}+\lambda+\frac{\lambda^2}{w} +h(1+w^{-1}+\lambda^2w^{-2})\right)+o(1).
\end{equation*}
Here $o(1)$ is the normal-cutoff contribution. The constants may depend on the fixed chart, $\chi$, and normal cutoff.
\end{enumerate}
\end{lemma}

\begin{proof}
\noindent\emph{Step 1. The residual identity.} Write $\mu=1-\kappa y$, $Q(t,s)=\kappa(s)^{1/6}\chi\left((s-t-z)/w\right)$. Before inserting the amplitude normalization and normal cutoff, consider the single complex mode $v^{\mathbb C} =h e^{i(\Phi_{0,\lambda}-t)/h} Q\Ai\left(-\frac{\zeta_\lambda}{h^{2/3}}\right)$. Direct differentiation, using the Airy equation $\Ai''\left(-\frac{\zeta_\lambda}{h^{2/3}}\right) = -\frac{\zeta_\lambda}{h^{2/3}} \Ai\left(-\frac{\zeta_\lambda}{h^{2/3}}\right)$, gives
\begin{equation}\label{eq:residual-id}
\begin{aligned}
(\partial_t^2-\Delta)v^{\mathbb C} =e^{i(\Phi_{0,\lambda}-t)/h}\left(
h^{-1}E_0Q \Ai\left(-\frac{\zeta_\lambda}{h^{2/3}}\right) +2ih^{-2/3}E_1Q \Ai'\left(-\frac{\zeta_\lambda}{h^{2/3}}\right)\right.\\
-i(2Q_t+2\langle d\Phi_{0,\lambda},dQ\rangle +Q\Delta\Phi_{0,\lambda}) \Ai\left(-\frac{\zeta_\lambda}{h^{2/3}}\right)\\
+h^{1/3}(2\langle dQ,d\zeta_\lambda\rangle +Q\Delta\zeta_\lambda) \Ai'\left(-\frac{\zeta_\lambda}{h^{2/3}}\right)\\
\left.+h(Q_{tt}-\Delta Q) \Ai\left(-\frac{\zeta_\lambda}{h^{2/3}}\right)
\right).    
\end{aligned}
\end{equation}

\smallskip
\noindent\emph{Step 2. Airy norms and normalization.} Set $L:=\lambda^2h^{-2/3}\to\infty$. The change of variable to $x=-\zeta_\lambda/h^{2/3}$ has inverse Jacobian comparable to $h^{2/3}$, and the negative integration range has length comparable to $L$. The identity
\[
\left( -\frac{\zeta_\lambda}{h^{2/3}} \Ai\left(-\frac{\zeta_\lambda}{h^{2/3}}\right)^2 - \Ai'\left(-\frac{\zeta_\lambda}{h^{2/3}}\right)^2 \right)' = \Ai\left(-\frac{\zeta_\lambda}{h^{2/3}}\right)^2,
\]
where the derivative is taken with respect to the Airy argument, and the Airy asymptotics give
\begin{equation*}
\int_{-L}^0 \left| \Ai\left(-\frac{\zeta_\lambda}{h^{2/3}}\right) \right|^2 \dd\left(-\frac{\zeta_\lambda}{h^{2/3}}\right) \sim \frac1\pi L^{1/2}, \quad \int_{-L}^0 \left| \Ai'\left(-\frac{\zeta_\lambda}{h^{2/3}}\right) \right|^2 \dd\left(-\frac{\zeta_\lambda}{h^{2/3}}\right) \le CL^{3/2}.
\end{equation*}
The lower bound averages the oscillations; on the positive half-line both integrals are bounded by Airy decay. Returning to $y$, we obtain
\begin{gather*}
\int|\Ai(-\zeta_\lambda/h^{2/3})|^2\dd y \sim h^{2/3}L^{1/2} =h^{1/3}\lambda,\\
\int|\Ai'(-\zeta_\lambda/h^{2/3})|^2\dd y \le Ch^{2/3}L^{3/2} =Ch^{-1/3}\lambda^3.
\end{gather*}
In particular,
\begin{equation*}
\frac{ \|\Ai'(-\zeta_\lambda/h^{2/3})\|_{L^2_y} }{ \|\Ai(-\zeta_\lambda/h^{2/3})\|_{L^2_y} } \le C\lambda h^{-1/3}.
\end{equation*}
These are the bounds in \eqref{eq:airy-norms}. The leading time derivative of $v^{\mathbb C}$ is $-ie^{i(\Phi_{0,\lambda}-t)/h} Q\Ai\left(-\frac{\zeta_\lambda}{h^{2/3}}\right)$, and we use its $L^2$ norm on the scaled collar as the normalization for the residual terms below.

\smallskip
\noindent\emph{Step 3. Transport and cutoff terms.} The factor $\kappa^{1/6}$ cancels the leading transport term. Indeed, $\Delta\Phi_{0,\lambda}=-\kappa'/(3\kappa)+O(\lambda^2)$ and $\partial_s\Phi_{0,\lambda}\mu^{-2}=1+O(\lambda^2)$ give
\begin{align*}
2\mu^{-2}\partial_s\Phi_{0,\lambda} \frac{(\kappa^{1/6})'}{\kappa^{1/6}} +\Delta\Phi_{0,\lambda} &= \frac{\kappa'}{3\kappa} -\frac{\kappa'}{3\kappa} +O(\lambda^2)=O(\lambda^2).
\end{align*}
The derivative of the moving cutoff has the same small factor $\partial_s\Phi_{0,\lambda}\mu^{-2}-1$. Using also $\partial_s\zeta_\lambda=O(\lambda^2)$ and $\Delta\zeta_\lambda=O(1)$, we obtain the following bounds relative to the leading time-derivative norm, in the order of \eqref{eq:residual-id}:
\begin{equation*}
\begin{array}{c|c}
\text{term in \eqref{eq:residual-id}} & \text{normalized bound} \\ \hline
h^{-1}E_0Q\Ai & C\lambda^6/h \\
+2ih^{-2/3}E_1Q\Ai' & C\lambda^5/h \\
-i(2Q_t+2\langle d\Phi_{0,\lambda},dQ\rangle+Q\Delta\Phi_{0,\lambda})\Ai & C\lambda^2(1+w^{-1}) \\
+h^{1/3}(2\langle dQ,d\zeta_\lambda\rangle+Q\Delta\zeta_\lambda)\Ai' & C\lambda(1+\lambda^2w^{-1}) \\
h(Q_{tt}-\Delta Q)\Ai & Ch(1+w^{-1}+\lambda^2w^{-2}).
\end{array}
\end{equation*}
For the last bound, we use
\[
Q_{tt}-\Delta Q = \kappa^{1/6}(1-\mu^{-2})\chi''/w^2 +O(1+w^{-1})
\]
in the corresponding cutoff norms.

Summing the five bounds gives the bracket in \eqref{eq:forcing}. In particular, the second term requires $\lambda^5/h\to0$, while the terms involving the moving cutoff require $\lambda^2/w\to0$. The normal cutoff varies only where $\zeta_\lambda\le-c\lambda^2$. There the Airy decay is $e^{-c\lambda^3/h}$, which absorbs every cutoff derivative under the chosen scale $h\sim\lambda^{9/2}$.

For fixed $\ell$, replacing $(\lambda,h)$ by $(\lambda_n,h_n)$ changes only a finite constant. Each mode in \eqref{eq:airy-wave} is the preceding expression multiplied by $-2\sqrt\pi\,d_{*,n}e^{2\pi in\alpha} b_{\lambda,w}h_n^{-1/6}$. The residual and leading time derivative acquire the same factor. Moreover, the Airy bounds and tangential width give, uniformly in $t$,
\begin{align*}
b_{\lambda,w}h_n^{-1/6} \|Q\Ai(-\zeta_{\lambda_n}/h_n^{2/3})\| _{L^2(\{0\le y\le Y_1\lambda^2\})}\le C_\ell(\lambda w)^{-1/2}h_n^{-1/6} w^{1/2}h_n^{1/6}\lambda_n^{1/2}\le C_\ell.
\end{align*}
Thus the normalized single-mode bounds give the stated $L^1_tL^2_x$ estimate for the finite sum. Taking the real part preserves the upper bound; the normal cutoff contributes the stated $o(1)$.

\end{proof}

\subsection{Boundary return and momentum matching}\label{subsec:phase-return}

Lemma~\ref{lem:matched} in Section~\ref{subsec:gcc-matched} uses the change of boundary phase and normal momentum between two successive impacts. We prove these return estimates here; Section~\ref{subsec:jets} combines them with the phase variation within each ray segment. Let $p_\lambda=C_\lambda'$ and $a_\lambda=\sqrt{1-p_\lambda^2}$. Recall the next impact $s^+(s,\beta)$, reflected normal component $B(s,\beta)$, and travel time $l(s,\beta)$ for a positively oriented ray launched from $s$ with normal component $\beta$. Define the boundary return error by
\begin{equation}\label{eq:return-error}
F_\lambda(s,\beta):=C_\lambda\left(s^+(s,\beta)\right)-C_\lambda(s)-l(s,\beta)+2\lambda^3/3.
\end{equation}

\begin{lemma}\label{lem:phase-return}
On the fixed arc, the matched initial momentum satisfies
\begin{equation}\label{eq:matched-return}
F_\lambda\left(s,a_\lambda(s)\right)=O_{C_s^1}(\lambda^6).
\end{equation}
Fix a sufficiently small preliminary mismatch bound $\rho_*>0$. For all sufficiently small $\lambda>0$ and $\delta=\beta-a_\lambda(s)$ with $|\delta|\le\rho_*\lambda$, we have
\begin{equation}\label{eq:graph-return}
\begin{split}
B(s,\beta)-a_\lambda\left(s^+(s,\beta)\right)&=\left(1+O(\lambda)\right)\delta+O(\lambda^5),\\
F_\lambda(s,\beta)&=O(\lambda^6)+O(\lambda^6|\delta|)+O(\lambda\delta^2).
\end{split}
\end{equation}
The constants depend only on the fixed arc and the preliminary mismatch neighborhood. They are independent of any subsequently chosen incoming band.
\end{lemma}

\begin{proof}
\noindent\emph{Step 1. The matched return error.} A segment launched with normal component $a_\lambda(s)$ has arclength increment
\begin{equation*}
d=2\kappa^{-2/3}\lambda-\frac{4\kappa'}{3\kappa^{7/3}}\lambda^2+O(\lambda^3).
\end{equation*}
Its length as a function of boundary increment has the expansion
\begin{equation}\label{eq:chord-length}
l(s,d)=d-\frac{\kappa^2d^3}{24}-\frac{\kappa\kappa'd^4}{24}+\left(\frac{\kappa^4}{1920}-\frac{\kappa\kappa''}{80}-\frac{(\kappa')^2}{90}\right)d^5+O(d^6).
\end{equation}
Substituting this expansion into \eqref{eq:return-error} cancels the cubic and quartic terms. The fifth-order coefficient is
\begin{equation*}
2c_4\kappa^{-2/3}-\frac{\kappa^{2/3}}{60} -\frac{2\kappa''}{45\kappa^{7/3}}+\frac{8(\kappa')^2}{135\kappa^{10/3}}=0.
\end{equation*}
The unknown cubic coefficient of $d$ drops out by stationarity of the leading expression. This proves \eqref{eq:matched-return}. The same $c_4$ has canceled both an error in the phase equation and the boundary return error.

\smallskip
\noindent\emph{Step 2. The reflected normal momentum.} Put $\widehat s(s)=s^+\left(s,a_\lambda(s)\right)$ and let $\widehat B$ be the next reflected normal component. First variation of chord length gives
\begin{equation*}
\frac d{ds}F_\lambda\left(s,a_\lambda(s)\right) =\left(C_\lambda'(\widehat s)-\sqrt{1-\widehat B^2}\right)\widehat s'(s).
\end{equation*}
The initial tangential terms cancel exactly. Since $\widehat s'=1+O(\lambda)$, the differentiated return estimate \eqref{eq:matched-return} yields
\begin{equation*}
C_\lambda'(\widehat s)-\sqrt{1-\widehat B^2}=O(\lambda^6).
\end{equation*}
Recall that $a_\lambda^2=1-(C_\lambda')^2$. Both normal momenta are comparable to $\lambda$, so rationalization gives
\begin{align*}
\widehat B-a_\lambda(\widehat s) =\frac{\left(C_\lambda'(\widehat s)-\sqrt{1-\widehat B^2}\right) \left(C_\lambda'(\widehat s)+\sqrt{1-\widehat B^2}\right)}{\widehat B+a_\lambda(\widehat s)}=O(\lambda^5).
\end{align*}
Thus the normal matching loses one power of $\lambda$ relative to the tangential estimate.

\smallskip
\noindent\emph{Step 3. A perturbed initial momentum.} In the preliminary neighborhood $|\delta|\le\rho_*\lambda$, we have
\begin{equation*}
\partial_\beta\left(B(s,\beta)-a_\lambda\left(s^+(s,\beta)\right)\right)=1+O(\lambda).
\end{equation*}
Together with Step~2, this proves the first line of \eqref{eq:graph-return}. For the second, differentiation of the return gives
\begin{gather*}
\partial_\beta F_\lambda(s,\beta) =\left(C_\lambda'(s^+)-\sqrt{1-B^2}\right)\partial_\beta s^+,\\
\partial_\beta F_\lambda\left(s,a_\lambda(s)\right)=O(\lambda^6) ,\text{ \; and \; }\partial_\beta^2F_\lambda(s,\beta)=O(\lambda).
\end{gather*}
Taylor's expansion reads
\begin{align*}
F_\lambda(s,a_\lambda+\delta) &=F_\lambda(s,a_\lambda)+\delta\,\partial_\beta F_\lambda(s,a_\lambda)+\delta^2\int_0^1(1-\theta)\partial_\beta^2F_\lambda(s,a_\lambda+\theta\delta)\dd\theta.
\end{align*}
Together with \eqref{eq:matched-return}, this proves the second line of \eqref{eq:graph-return}, with the stated uniformity.
\end{proof}

\subsection{Differentiated incoming phases}\label{subsec:jets}
We prove the phase estimates \eqref{eq:slow-jets}--\eqref{eq:quarter-return} used in Lemma~\ref{lem:matched}. Lemma~\ref{lem:phase-return} controls the boundary return; the remaining calculation concerns the first half of each ray segment, from the boundary toward the turning region.

Fix a compact regular band $I\subset(0,1]$, bounded away from zero. Let $\xi_{s,\beta}:=\sqrt{1-\beta^2}\,\boldsymbol\tau(s)+\beta\,\mathbf n(s)$ be the positively oriented unit initial velocity, and write
\[
t(s,\beta,z)=\tfrac12(1-z)l(s,\beta), \quad x(s,\beta,z) =x(s,0)+t(s,\beta,z)\xi_{s,\beta}.
\]
Here $t(s,\beta,z)$ is the time elapsed since the initial boundary collision. All regular-sheet estimates and coefficient constructions below are taken on $[\min I,1]$.

Define the centered slow phase
\begin{equation}\label{eq:centered-V}
V_\lambda(s,\beta,z) := \Phi_\lambda^-\left(x(s,\beta,z)\right) -t(s,\beta,z)-C_\lambda(s)+\lambda^3/3.
\end{equation}
The boundary normalization gives
\[
\Phi_\lambda^-(x(s,0)) =C_\lambda(s)-\lambda^3/3, \quad V_\lambda(s,\beta,1)=0.
\]
Write $\rho=\frac{\beta-a_\lambda(s)}{\lambda}$, and recall $M=(2\pi h)^{-1}\lambda^3$ and $\epsilon=(2\pi h)^{-1}\lambda^6$.

For two successive positively oriented ray segments in the fixed chart, let $t_j,s_j,\beta_j$ denote the initial impact time, boundary point, and normal component. Let $\phi_j,\psi_j$ be the phases \eqref{eq:actual-phases} evaluated at
\[
\left( t_j+t(s_j,\beta_j,z), x(s_j,\beta_j,z) \right),
\]
and put
\[
\rho_j=\frac{\beta_j-a_\lambda(s_j)}{\lambda}, \quad b_j=\psi_j(1), \quad e_j=M\rho_j^2.
\]

\begin{lemma}\label{lem:phase-jets}
Fix the preliminary mismatch neighborhood of Lemma~\ref{lem:phase-return}. For each fixed band $I$, there are smaller mismatch and scale bounds such that the following statements hold uniformly in the boundary point on the fixed arc.
\begin{enumerate}[label=(\roman*),leftmargin=*,itemsep=1mm]
\item The centered phase has the smooth factorization
\begin{equation}\label{eq:V-factor}
V_\lambda(s,a_\lambda(s)+\lambda\rho,z) = \lambda^6D_\lambda(s,\rho,z) +\lambda^3\rho^2H_\lambda(s,\rho,z),
\end{equation}
where $D_\lambda,H_\lambda$ extend smoothly through $\lambda=0$ and have uniformly bounded fixed derivatives on the chosen parameter neighborhood and band. They may be chosen so that
\[
D_\lambda(s,\rho,1)=H_\lambda(s,\rho,1)=0.
\]

\item For two successive segments in this neighborhood,
\begin{gather*}
\psi_j = b_j+\epsilon D_\lambda(s_j,\rho_j,z) +e_jH_\lambda(s_j,\rho_j,z), \quad \partial_zH_0(s,0,z) =\frac{\kappa(s)^{-2/3}}{2z^2}, \\
\phi_j'/M\rightarrow2z^2 \quad\text{as }(\lambda,\rho_j)\to(0,0), \\
\rho_1=\left(1+O(\lambda)\right)\rho_0+O(\lambda^4), \quad |e_1-e_0|\le C\lambda e_0+C\epsilon.
\end{gather*}
The phase limit is uniform on $I$ with every fixed number of $z$-derivatives.

\item With the scales \eqref{eq:scales}, the slow phases satisfy
\[
\begin{gathered}
\psi_1(z) = \psi_0(z)+1/4+\delta_{\rm ret}(z)\pmod1, \\
\|\delta_{\rm ret}\|_{L^\infty(I)} \le(C_0+\lambda C_I)e_0+D_I\epsilon.
\end{gathered}
\]
\end{enumerate}
Here $C_0$ is fixed on the preliminary mismatch neighborhood and is independent of $I$. The constants $C_I,D_I$ and the smaller mismatch and scale bounds may depend on $I$. All estimates are uniform in the common phase origin.
\end{lemma}

\begin{proof}
\noindent\emph{Step 1. Smooth factorization.} The quotient $l(s,\beta)/\beta$ extends smoothly through $\beta=0$, with value $2/\kappa(s)$. Since $a_\lambda(s) =\lambda\left(\kappa(s)^{1/3}+O(\lambda^2)\right)$, with smooth dependence on the parameters, we have $y\left(x(s,\beta,z)\right)=O(\lambda^2)$, $s\left(x(s,\beta,z)\right)-s-t(s,\beta,z)=O(\lambda^3)$. Together with the expansion of $\Phi_{0,\lambda}$, this shows that
\[
\lambda^{-3} V_\lambda(s,a_\lambda(s)+\lambda\rho,z)
\]
extends jointly smoothly through $\lambda=0$.

Next take $\rho=0$. Along this matched chord put
\[
p=\nabla\Phi_\lambda^-\left(x(t)\right), \quad d=p-\xi_{s,a_\lambda(s)}, \quad E=|p|^2-1.
\]
Away from the turning region, we have
\[
E=O_I(\lambda^5), \;\nabla E=O_I(\lambda^3),\; \nabla^2\Phi_\lambda^-=O_I(\lambda^{-1}), \;d(0)=O(\lambda^5).
\]
The boundary estimate for $d(0)$ uses exact tangential matching and the scalar boundary error $O(\lambda^6)$. The identity
\[
d' = \tfrac12\nabla E -(\nabla^2\Phi_\lambda^-)d
\]
over time $O(\lambda)$ gives $d=O_I(\lambda^4)$ by Gronwall. Therefore
\[
\frac{\dd}{\dd t} \left( \Phi_\lambda^-\left(x(t)\right)-t \right) = \tfrac12(E-|d|^2),
\]
together with $V_\lambda(s,a_\lambda(s),1)=0$, implies
\[
V_\lambda(s,a_\lambda(s),z)=O_I(\lambda^6).
\]

Differentiating \eqref{eq:centered-V} with respect to the initial normal component at fixed $s,z$, and evaluating at $\beta=a_\lambda(s)$, we obtain
\[
\partial_\beta V_\lambda = \tfrac12t_\beta(E-|d|^2) +t\,d\cdot\partial_\beta\xi_{s,\beta} = O_I(\lambda^5),
\]
since $t_\beta=O(1)$ and $\partial_\beta\xi_{s,\beta}=O(1)$.

Thus both
\[
\lambda^{-3}V_\lambda(s,a_\lambda(s),z) \quad\text{and}\quad \lambda^{-2}\partial_\beta V_\lambda(s,a_\lambda(s),z)
\]
are $O_I(\lambda^3)$. By joint smoothness, their Taylor coefficients of orders zero, one, and two in $\lambda$ vanish identically. Smooth division by $\lambda^3$, followed by Taylor's formula in $\rho$, gives \eqref{eq:V-factor}. Here $D_\lambda$ collects the constant and linear terms in $\rho$, and $H_\lambda$ is the second-order Taylor remainder. Their fixed derivatives are uniformly bounded. Since $V_\lambda(s,\beta,1)=0$ for every $\beta$, this construction also gives
\[
D_\lambda(s,\rho,1)=H_\lambda(s,\rho,1)=0.
\]

\smallskip
\noindent\emph{Step 2. The slow and fast derivatives.} To identify the leading slow derivative, freeze $s$. The leading calculation gives
\begin{align*}
\lim_{\lambda\to0^+} \lambda^{-3} \partial_zV_\lambda (s,a_\lambda(s)+\lambda\rho,z)= \frac{\kappa^{1/3}+\rho}{2\kappa} \left( \sqrt{ \kappa^{2/3} -(\kappa^{1/3}+\rho)^2 +(\kappa^{1/3}+\rho)^2z^2 } -(\kappa^{1/3}+\rho)z \right)^2.
\end{align*}
At $\rho=0$, the expression in parentheses vanishes, and its $\rho$-derivative is $-1/z$. Hence
\[
\partial_zH_0(s,0,z) = \frac{\kappa(s)^{-2/3}}{2z^2}.
\]
Replacing the minus sign in parentheses by a plus sign gives the fast limit $2z^2$ as $(\lambda,\rho)\to(0,0)$, uniformly in $s$ on the fixed arc and on the fixed regular band $I$. These are the limits in \eqref{eq:slow-jets}--\eqref{eq:next-jets}.

\smallskip
\noindent\emph{Step 3. Comparing successive ray segments.} The return estimates \eqref{eq:short-chord} and \eqref{eq:graph-return} give
\[
s_1-s_0=O(\lambda), \quad \rho_1=\left(1+O(\lambda)\right)\rho_0+O(\lambda^4),
\]
\begin{align*}
|e_1-e_0| \le CM\left( \lambda\rho_0^2+\lambda^4|\rho_0|+\lambda^8 \right)\le C\lambda e_0+C\epsilon.
\end{align*}
The coefficient derivative bounds imply
\[
\left\| H_\lambda(s_1,\rho_1,\cdot) -H_\lambda(s_0,\rho_0,\cdot) \right\|_{L^\infty(I)} \le C_I\lambda.
\]
By \eqref{eq:centered-V}, the boundary normalization, and \eqref{eq:V-factor},
\[
\psi_j-b_j = \epsilon D_\lambda(s_j,\rho_j,\cdot) +e_jH_\lambda(s_j,\rho_j,\cdot).
\]
Using the preceding bounds and the uniform bounds on $D_\lambda,H_\lambda$, we obtain
\[
\|(\psi_1-b_1)-(\psi_0-b_0)\|_{L^\infty(I)} \le C_I\lambda e_0+C_I\epsilon.
\]
This controls the change of the centered phases on the fixed band.

The boundary values are compared separately. Using \eqref{eq:return-error} and the quarter-period compatibility in \eqref{eq:scales}, we obtain
\[
b_1-b_0 = 1/4+\frac{F_\lambda(s_0,\beta_0)}{2\pi h} \pmod1,
\]
and
\[
\frac{|F_\lambda(s_0,\beta_0)|}{2\pi h} \le \frac{C}{2\pi h} \left( \lambda^6+\lambda^7|\rho_0|+\lambda^3\rho_0^2 \right) \le C(e_0+\epsilon).
\]
This constant is independent of $I$. Combining the two contributions gives
\[
\|\delta_{\rm ret}\|_{L^\infty(I)} \le(C_0+\lambda C_I)e_0+D_I\epsilon,
\]
which is \eqref{eq:quarter-return}, with the boundary-return constant separated from the band-dependent constants.
\end{proof}

\bibliographystyle{abbrv}
\bibliography{references}

\end{document}